\documentclass[11pt]{amsart}
\usepackage{color, soul}
\usepackage{graphicx}
\usepackage{tikz-cd}
\usepackage{framed}
\usepackage{simpler-wick}
\usepackage{pgf,tikz,pgfplots}
\usepackage{mathrsfs}
\usepackage{mathpple}
\usepackage{graphicx}
\usepackage[all,cmtip]{xy}
\usepackage{amsmath}
\usepackage{tikz}
\usepackage{mathdots}
\usepackage{yhmath}
\usepackage{cancel}
\usepackage{color}
\usepackage{siunitx}
\usepackage{array}
\usepackage{multirow}
\usepackage{amssymb}
\usepackage{gensymb}
\usepackage{tabularx}
\usepackage{booktabs}
\usepackage{makecell}
\usetikzlibrary{fadings}
\usetikzlibrary{patterns}
\usetikzlibrary{shadows.blur}
\usetikzlibrary{shapes}
\usepackage{hyperref,xcolor}
\definecolor{wine-stain}{rgb}{0.5,0,0}
\hypersetup{
  colorlinks,
  linkcolor=wine-stain,
  linktoc=all
}

\usepackage{latexsym,bm,amsmath,amssymb}

\usetikzlibrary{arrows}
\newtheorem{thm}{Theorem}[section]
\newtheorem{cor}[thm]{Corollary}
\newtheorem{lem}[thm]{Lemma}
\newtheorem{prop}[thm]{Proposition}
\theoremstyle{definition}
\newtheorem{exa}[thm]{Example}
\newtheorem{defn}[thm]{Definition}

\theoremstyle{remark}
\newtheorem{rem}[thm]{Remark}
\numberwithin{equation}{section}

\newcommand{\shift}{\text{\small{[1]}}}

\newcommand{\omegaQI}{\omega_{X^I}(*\Delta_I)_{\mathcal{Q}}}
\newcommand{\omegaQIP}{\omega_{X^{I'}}(*\Delta_{I'})_{\mathcal{Q}}}
\newcommand{\BVk}{O_{\mathrm{BV},\mathbf{k}}}
\newcommand{\OL}{O(\mathbf{L})}
\newcommand{\ap}{a^p}
\newcommand{\aq}{a^q}
\newcommand{\as}{a^s}

\newcommand{\apc}{a^{p\dagger}}
\newcommand{\aqc}{a^{q\dagger}}
\newcommand{\asc}{a^{s\dagger}}

\newcommand{\hoco}{\mathrm{hocolim}}
\newcommand{\BVL}{O_{\mathrm{BV}}(\mathbf{L})}
\newcommand{\BVLk}{O_{\mathrm{BV},\mathbf{k}}(\mathbf{L})}
\newcommand{\Rk}{\mathcal{R}_{\mathbf{k}}}
\newcommand{\equiQ}{{\textsf{Q}}}
\newcommand{\fO}{\hat{\mathbb{O}}_N}
\newcommand{\RtensorT}{(\mathcal{R}[1])^{\boxtimes T}(*\Delta_T)_{\mathcal{Q}}}
\newcommand{\DistributionS}{\mathscr{D}'(X^{\mathcal{S}})}
\newcommand{\dDistribution}{d_{\mathscr{D}}}
\newcommand{\TraceBV}{\mathrm{\mathbf{{Tr}}}^{\mathrm{BV}}}
\newcommand{\Trace}{\mathrm{\mathbf{{Tr}}}}
\newcommand{\trace}{\mathrm{\mathbf{{tr}}}}

\begin{document}
\setulcolor{green}
\setstcolor{blue}
\sethlcolor{yellow}

\title[]{Elliptic trace map on chiral algebras}%
 \author{Zhengping Gui and Si Li}

  \address{
Z. Gui: International Centre for Theoretical Physics, Trieste, Italy;
}
\email{zgui@ictp.it}
  \address{
S. Li:
Yau Mathematical Sciences Center, Tsinghua University, Beijing, China
}

\email{sili@mail.tsinghua.edu.cn}

\address{}%
\email{}%

\thanks{}%
\subjclass{}%
\keywords{}%

\begin{abstract} Trace map on deformation quantized algebra leads to the algebraic index theorem. In this paper, we investigate a two-dimensional chiral analogue of the algebraic index theorem via the theory of chiral algebras developed by Beilinson and Drinfeld. We construct a trace map on the elliptic chiral homology of the free $\beta\gamma-bc$ system using the BV quantization framework. As an example, we compute the trace evaluated on the unit constant chiral chain and obtain the formal Witten genus in the Lie algebra cohomology. We also construct a family of elliptic trace maps on coset models.
\end{abstract}
\maketitle
\tableofcontents
\def\Xint#1{\mathchoice
{\XXint\displaystyle\textstyle{#1}}%
{\XXint\textstyle\scriptstyle{#1}}%
{\XXint\scriptstyle\scriptscriptstyle{#1}}%
{\XXint\scriptscriptstyle\scriptscriptstyle{#1}}%
\!\int}
\def\XXint#1#2#3{{\setbox0=\hbox{$#1{#2#3}{\int}$}
\vcenter{\hbox{$#2#3$}}\kern-.5\wd0}}
\def\ddashint{\Xint=}
\def\dashint{\Xint-}

\section{Introduction}

It is well known that the Atiyah-Singer index theorem is closely related to supersymmetric/topological quantum mechanics \cite{getzler1983pseudodifferential,windey1983supersymmetric,witten1982constraints}. Though not rigorous, this physics interpretation provides a clear and deep insight into the origin of index theorems via the geometry of the loop space. As a natural generalization,  one can replace the circle with a two-dimensional torus. This leads to Witten's proposal  \cite{witten1987elliptic,witten1988index} for the index of Dirac operators on loop spaces, which is still a mystery in geometry and topology.

In \cite{Fedosov:1996fu,nest1995algebraic}, Fedosov and Nest-Tsygan established the algebraic index theorem for deformation quantized algebras as the algebraic analogue of the Atiyah-Singer index theorem. It was further shown  \cite{nest1996formal} that the original Atiyah-Singer index theorem can be deduced from this algebraic one. In \cite{Grady:2011jc}, Gwilliam and Grady  studied the one-dimensional Chern-Simons theory (or topological quantum mechanics) on the total space of a cotangent bundle and showed that the one-loop effective theory encodes the Todd genus. Later \cite{BVQandindex,gui2021geometry} established the exact connection between the algebraic index theorem and topological quantum mechanics via a trace map constructed in the Batalin-Vilkovisky(BV) formalism. This sets up a mathematical understanding of the physics approach to index theorem in terms of an exact low-energy effective quantum field theory \cite{gui2021geometry}.

 In this paper, we investigate a two-dimensional chiral analogue of the algebraic index theorem via the theory of chiral algebras developed by Beilinson and Drinfeld \cite{book:635773}.

\subsection{Batalin-Vilkovisky(BV) quantization framework}\label{sec: BV}
We briefly describe a version of mathematical structures appearing in quantum field theories within the BV formalism. This will play an essential role in our exploration of the trace map and chiral index later.

 Let us denote by $\mathbf{k}$ the field of Laurent series $\mathbb{C}((\hbar))$. Roughly speaking, BV quantization in quantum field theory on $X$ leads to the following data
  \begin{enumerate}
    \item A factorization algebra  of local observables (we follow the set-up in \cite{costello2021factorization}).
 $$
 \mathrm{Obs}:\quad \text{a $\mathbf{k}$-module equipped with certain algebra structure.}
 $$
 It carries an algebraic structure called factorization product (or operator product expansion in physics terminology).
    \item  A (factorization) chain complex
 $$
 C_{\bullet}(\mathrm{Obs}):\quad \text{a $\mathbf{k}$-chain complex}, \qquad d: \quad \text{the differential}.
 $$
It captures the algebraic structure and global information from local observables.
\item A BV algebra $(O_{\mathrm{BV}}, \Delta)$
  $$
  O_{\mathrm{BV}}:\quad \text{a BV algebra over}\ \mathbb{C}, \qquad \text{$\Delta:\quad$ the BV operator}
  $$
  together with a BV integration map
   $$
 \int_{\mathrm{BV}}: O_{\mathrm{BV}}\to \mathbb{C}, \qquad \text{such that} \qquad \int_{\mathrm{BV}} \Delta(-)=0.
 $$
 In physics, $O_{\mathrm{BV}}$ are functions on the space of zero modes at low energy. $\int_{\mathrm{BV}}$ is a choice (related to the gauge fixing) of the integration map on zero modes.
 \item A $\mathbf{k}$-linear map (encoding the path integral in physics)
 $$
 \mathrm{\mathbf{Tr}}:C_{\bullet}(\mathrm{Obs})\rightarrow O_{\mathrm{BV},\mathbf{k}}=O_{\mathrm{BV}}\otimes_{\mathbb{C}}\mathbb{C}((\hbar))
 $$
  satisfying the \textbf{quantum master equation} (QME)
  $$
  (d+\hbar\Delta) \mathrm{\mathbf{Tr}}=0.
  $$
 In other words, QME says that $\mathbf{Tr}$ is a chain map intertwining $d$ and $-\hbar \Delta$. In physics, it describes the quantum gauge consistency condition within BV quantization.
  \end{enumerate}

The index is obtained as the partition function of the model, which can be formulated as
$$
\text{Index}= \int_{\mathrm{BV}}  \mathrm{\mathbf{Tr}}(1).
$$

\subsection{One-dimensional topological theory and the algebraic index} Let us illustrate the above data via an example of one-dimensional topological theory and explain how it is related to the algebraic index theorem.

One way to describe topological quantum mechanics is the one-dimensional $\sigma$-model of AKSZ type  \cite{Alexandrov:276069} describing maps (with additional data required by supersymmetry)
$$
X=S^1\to Y
$$
from the circle $S^1$ to a symplectic manifold $Y$ (the phase space). Let us first describe the case $Y=\mathbb{R}^{2n}$ with the standard symplectic form $\omega=\sum\limits_{i=1}^n dp^i \wedge dq^i$. The BV quantization in the sense of Section \ref{sec: BV} is worked out in \cite{gui2021geometry} and is summarized as follows.

 \begin{itemize}
    \item The algebra  of local observables is the Weyl algebra $\mathcal{W}_{2n}$ in $2n$ variables equipped with the Moyal-Weyl product
 $$
 \mathrm{Obs}=\mathcal{W}_{2n}\simeq \mathbb{C}[[p^1,\dots,p^{n},q^1,\dots,q^{n}]]((\hbar)).
 $$
 We work in formal series for convenience in perturbative theory.

    \item   The factorization complex is the Hochschild chain complex
 $$
 C_{\bullet}(\mathrm{Obs})=(C_{\bullet}(\mathcal{W}_{2n}),b).
 $$
\item The BV algebra is space of the (formal) differential forms on the phase space
  $$
  O_{\mathrm{BV}}\simeq \mathbb{C}[[p^1,\dots,p^{n},q^1,\dots,q^{n},dp^{1},\dots,dp^{n},dq^{1},\dots,dq^{n}]],
  $$
  and we set
  $$
  \deg(p^i)=\deg(q^i)=0,\quad\deg(dp^i)= \deg(dq^i)=-1, i=1,\dots,n.
  $$
The BV operator is the Lie derivative $\Delta=\mathcal L_{\omega^{-1}}$ with respect to the Poisson tensor
$$
\omega^{-1}=\sum_{i=1}^{2n}\partial_{p^i}\wedge\partial_{q^{i}}.
$$
It is easy to see that $\Delta$ is a degree 1 operator.
 The BV integration map is  (here $\iota_{V}$ is the contraction with the vector field $V$)
 $$
  \int_{\mathrm{BV}}(\alpha)=\left. \iota_{\partial_{q^1}}\cdots \iota_{\partial_{q^n}}\iota_{\partial_{p^1}}\cdots \iota_{\partial_{p^n}}\alpha \right|_{p^i=q^i=0}, \qquad \alpha \in O_{\mathrm{BV}}.
 $$
 \item   A $\mathbf{k}$-linear map (using Feynman diagrams)
 $$
 \mathrm{\mathbf{Tr}}:C_{\bullet}(\mathcal{W}_{2n})\rightarrow O_{\mathrm{BV},\mathbf{k}}=O_{\mathrm{BV}}\otimes_{\mathbb{C}}\mathbb{C}((\hbar))
 $$
 is constructed in \cite{gui2021geometry}  and it satisfies the quantum master equation (QME)
  $$
  (b+\hbar\Delta) \mathrm{\mathbf{Tr}}=0.
  $$
  \end{itemize}
The cohomology of $(O_{\mathrm{BV},\mathbf{k}},\hbar\Delta)$ is one-dimensional and concentrated in degree $-2n$. By direct computation, one can show that $\mathrm{\mathbf{Tr}}$ is not trivial cohomologically. On the other hand, the Hochschild homology of the Weyl algebra $\mathcal{W}_{2n}$ is also one-dimensional and concentrated in degree $-2n$ (see \cite{feigin1983cohomologies,brylinski1987some}). We conclude that $\mathrm{\mathbf{Tr}}$ is a quasi-isomorphism. Later we will see that it is also true in the two-dimensional chiral theory. In fact, $\mathrm{\mathbf{Tr}}$ will be a quasi-isomorphism from the chiral chain complex to a BV algebra.

 Furthermore, the trace map
  $\mathrm{\mathbf{Tr}}$ can be extended to be valued in certain Lie algebra cohomology and $\mathrm{\mathbf{Tr}}(1)$ is equal to the formal $\widehat{A}$-genus times the exponential of the formal characteristic class associated to $\frac{1}{\hbar}\omega_{\hbar}$. This allows us to glue the above BV quantization on an arbitrary symplectic manifold $Y$ via the Gelfand-Fuks construction. Then  $\mathrm{\mathbf{Tr}}$ becomes the canonical trace $\mathrm{Tr}$ on the deformation quantized algebra and one obtains the usual algebraic index theorem
     $$
      \mathrm{Tr}(1)=\int_Y\widehat{A}(TY)\exp(\frac{\omega_{\hbar}}{\hbar}).
      $$

\subsection{Two-dimensional chiral theory and the chiral algebraic index }\label{IntroChiral}
In this paper, we focus on the two-dimensional chiral analogue of the one-dimensional theory discussed in the previous subsection. The chiral analogue of the Weyl algebra $\mathcal{W}_{2n}$ is the chiral algebra $\mathcal{A}^{\beta\gamma-bc}$ associated to the free $\beta\gamma-bc$ system.  In the seminal work of Beilinson and Drinfeld \cite{book:635773}, they give a construction of chiral chain complex $C^{\mathrm{ch}}_{\bullet}(-)$ associated to chiral algebras (in the paper we will use a specific construction of $C^{\mathrm{ch}}_{\bullet}(-)$, denoted by $\tilde{C}^{\mathrm{ch}}_{\bullet}(-)_{\mathcal{Q}}$). The main point of this paper is that the chiral chain complex of $\mathcal{A}^{\beta\gamma-bc}$ plays a role similar to the Hochschild chain complex of the Weyl algebra, where we can construct naturally an explicit (elliptic) trace map $\mathrm{\mathbf{Tr}}_{\mathcal{A}^{\beta\gamma-bc}}$ satisfying quantum master equation. See Table \ref{1dvs2d}.
\begin{table}
  \centering
\begin{tabular}{|c|c|}
\hline
 1d TQM (on $\mathbb{R}^{2n}$) & 2d Chiral CFT( free $\displaystyle \beta \gamma -bc$ system) \\
\hline\hline
 $\displaystyle S^{1}\rightarrow ( \mathbb{R}^{2n},\omega )$ & $\displaystyle X\rightarrow \left(\mathbb{C}^{n|m} ,\langle -,-\rangle \right)$ \\
\hline
 \makecell[c]{Associative algebra: deformation \\quantization algebra $\displaystyle \mathcal{W}_{2n}$} & \makecell[c]{Chiral algebra: free $\displaystyle \beta \gamma -bc$ \\system $\displaystyle \mathcal{A}^{\beta\gamma -bc}$ }\\
\hline
 Hochschild chain complex: $\displaystyle ( C_{\bullet }(\mathcal{W}_{2n}) ,b)$ & Chiral chain complex: $\displaystyle \left( C_{\bullet }^{\mathrm{ch}}\left(\mathcal{A}^{\beta \gamma -bc}\right) ,d \right)$ \\
\hline
 \makecell[c]{Quantum master equation:\\$\displaystyle ( \hbar \Delta +b) \mathrm{\mathbf{Tr}} =0$} & \makecell[c]{Quantum master equation:\\$\displaystyle ( \hbar \Delta +d) \mathrm{\mathbf{Tr}}_{\mathcal{A}^{\beta\gamma-bc}}=0$} \\
 \hline
\makecell[c]{$\displaystyle  \mathrm{\mathbf{Tr}}:C_{\bullet}(\mathcal{W}_{2n})\rightarrow O^{\mathrm{1d}}_{\mathrm{BV},\mathbf{k}}$ \\is a quasi-isomorphism }& \makecell[c]{$\displaystyle  \mathrm{\mathbf{Tr}}_{\mathcal{A}^{\beta\gamma-bc}}: C_{\bullet }^{\mathrm{ch}}\left(\mathcal{A}^{\beta \gamma -bc}\right) \rightarrow O^{\mathrm{2d}}_{\mathrm{BV},\mathbf{k}}$ \\is a quasi-isomorphism} \\
\hline
 \makecell[c]{n-point correlator:\\$\displaystyle \mathrm{\mathbf{Tr}}(\mathcal{O}_{1} \otimes \cdots \otimes \mathcal{O}_{n} ) =$integrals of \\differential forms on \ configuration\\spaces of the circle $\displaystyle S^{1}$} & \makecell[c]{n-point correlator:\\$\displaystyle \mathrm{\mathbf{Tr}}_{\mathcal{A}^{\beta\gamma-bc}}( \mathcal{O}_{1} \otimes \cdots \otimes \mathcal{O}_{n} ) =$regularized \\integrals of differential forms on $\displaystyle X^{n}$} \\
 \hline
\end{tabular}

  \caption{Here $\langle-,-\rangle$ is a symplectic form on the super space $\mathbb{C}^{n|m}$ .}\label{1dvs2d}
\end{table}

To explain our result, let us first give an informal review of Beilinson and Drinfeld's construction of chiral homology. The precise definition will be given in Section \ref{ChiralHomology}. A chiral algebra $\mathcal{A}$ on a curve $X$ (which is a $\mathcal{D}$-module on $X$) can be viewed as a Lie algebra object in the tensor category $(\mathcal{M}(X^{\mathcal{S}}),\otimes^{\mathrm{ch}})$. Roughly speaking, each element $M$ in the category $\mathcal{M}(X^{\mathcal{S}})$ is a collection that assigns every finite set $I$ a $\mathcal{D}_{X^I}$-module on the product $X^I$ satisfying certain compatibility conditions. This category carries a tensor product $\otimes^{\mathrm{ch}}$. Let $\mathcal{M}(X)$ be the category of $\mathcal{D}_X$-modules. There is an embedding
$$
\Delta^{(\mathcal{S})}_*:\mathcal{M}(X)\hookrightarrow\mathcal{M}(X^{\mathcal{S}}),
$$
which sends a $\mathcal{D}_X$-module $M$ to the collection of $\mathcal{D}_{X^I}$-module $\Delta^{(I)}_*M$ for each finite set $I$. Here $\Delta^{(I)}:X\hookrightarrow X^I$ is the diagonal embedding. A chiral algebra $\mathcal{A}$ can be viewed as a Lie algebra object  in $(\mathcal{M}(X^{\mathcal{S}}),\otimes^{\mathrm{ch}})$ via the embedding $\Delta^{(\mathcal{S})}_*$: we have a binary operation
$$
\mu^{(\mathcal{S})} : \Delta^{(\mathcal{S})}_*(\mathcal{A})\otimes^{\mathrm{ch}}\Delta^{(\mathcal{S})}_*(\mathcal{A})\rightarrow \Delta^{(\mathcal{S})}_*(\mathcal{A})
$$
induced by the chiral bracket $\mu$ on $\mathcal{A}$ (see Definition \ref{chiralDefn}). Then $\mu^{(\mathcal{S})}$ is a (graded) Lie bracket since $\mu$ is (graded) skew-symmetric and satisfies the Jacobi identity.

Let $C(\mathcal{A})$ be the reduced Chevalley-Eilenberg complex of the Lie algebra object corresponding to the chiral algebra $\mathcal{A}$
$$
(C(\mathcal{A}),d_{\mathrm{ch}})=(\bigoplus_{\bullet>0}\mathrm{Sym}_{\otimes^{\mathrm{ch}}}^{\bullet}(\Delta^{(S)}_*\mathcal{A}\shift),d_{\mathrm{ch}}),
$$
which is an admissible complex in $\mathcal{M}(X^{\mathcal{S}})$ in the sense of \cite{book:635773}. This means that we can view the complex $C(\mathcal{A})$ as  a complex of $\mathcal{D}$-modules on the Ran space $\mathcal{R}(X)$. The chiral homology is defined as follows
$$
C^{\mathrm{ch}}(X,\mathcal{A}):=R\Gamma_{\mathrm{DR}}(\mathcal{R}(X),C(\mathcal{A})),\quad H^{\mathrm{ch}}_{\bullet}(X,\mathcal{A}):=H^{-\bullet}C^{\mathrm{ch}}(X,\mathcal{A}).
$$
Any actual functorial complex representing $C^{\mathrm{ch}}(X,\mathcal{A})$ is called a \textit{chiral chain complex}.  We will use the one in \cite[Section 4.2.15.]{book:635773} which involves the Dolbeault and Spencer resolutions. See Section \ref{ChiralHomology} for details, there we use the same notation $\tilde{C}^{\mathrm{ch}}(X,\mathcal{A})_{\mathcal{Q}}$ as in \cite{book:635773}.

Now let $X$ be an elliptic curve $\mathbb{C}/\mathbb{Z}+\mathbb{Z}\tau$, and we focus on the free $\beta\gamma-bc$ system. The general set up in Section \ref{sec: BV} becomes (see Section \ref{sec:expanded} for a more detailed introduction)
\begin{itemize}
  \item The observable algebra is the chiral algebra of free $\beta\gamma-bc$ system
  $$
 \mathrm{Obs}=\mathcal{A}^{\beta\gamma-bc}.
 $$
  \item  The factorization complex is the chiral chain complex
 $$
 C_{\bullet}(\mathrm{Obs})=C^{\mathrm{ch}}_{\bullet}:=(\tilde{C}^{\mathrm{ch}}(X,\mathcal{A}^{\beta\gamma-bc})_{\mathcal{Q}},d).
 $$
 See Section \ref{ChiralHomology} for details.
 \item The BV algebra $O_{\mathrm{BV}}$ comes from cohomology groups on $X$. See Section \ref{BVchiral}.
\end{itemize}
The main results in this paper are summarized as follows (see Theorem \ref{mainThm}, Corollary \ref{mainCor} and Theorem \ref{WittenGenus}).
  \begin{thm}

  (1) We construct an explicit chain map from the chiral chain complex of  Beilinson and Drinfeld to a BV algebra
    $$
    \mathrm{\mathbf{Tr}}_{\mathcal{A}^{\beta\gamma-bc}}: (\tilde{C}^{\mathrm{ch}}(X,\mathcal{A}^{\beta\gamma-bc})_{\mathcal{Q}},d)\rightarrow (O_{\mathrm{BV},\mathbf{k}},-\hbar\Delta_{\mathrm{BV}}).
    $$
    In other words, it satisfies the quantum master equation
    $$
    (d+\hbar\Delta_{\mathrm{BV}}) \mathrm{\mathbf{Tr}}_{\mathcal{A}^{\beta\gamma-bc}}=0.
    $$
    Furthermore, the chain map $    \mathrm{\mathbf{Tr}}_{\mathcal{A}^{\beta\gamma-bc}}$ is a quasi-isomorphism.

   (2) For any Lie subalgebra $\mathfrak{g}$ of inner derivations of $\mathcal{A}^{\beta\gamma-bc}$, the above construction can be extended to the Lie algebra cochain. That is, we can get an element
     $$
     \Trace_{\mathfrak{g},\mathcal{A}^{\beta\gamma-bc}}(-)\{-\}\in C_{\mathrm{Lie},\mathrm{BV}}=C_{\mathrm{Lie}}\big(\mathfrak{g},\mathrm{Hom}_{\mathbf{k}}(\tilde{C}^{\mathrm{ch}}(X,\mathcal{A})_{\mathcal{Q}},\mathbf{k})\big)\otimes_{\mathbf{k}}\BVk
     $$
     satisfying
     $$
     \Trace_{\mathfrak{g},\mathcal{A}^{\beta\gamma-bc}}(-)\{-\}|_{C^0_{\mathrm{Lie},\mathrm{BV}}}=     \mathrm{\mathbf{Tr}}_{\mathcal{A}^{\beta\gamma-bc}}(-)\quad\text{and}\  (\partial_{\mathrm{Lie}}+d+\hbar\Delta_{\mathrm{BV}})\Trace_{\mathfrak{g},\mathcal{A}^{\beta\gamma-bc}}(-)\{-\}=0.
     $$

     (3) If we take $\mathcal{A}^{\beta\gamma-bc}$ to be the free $\beta\gamma$-system $\mathcal{A}^{\beta\gamma}_N$ of rank $N$ and $\mathfrak{g}$ to be $\widetilde{W}_N$ which is an extension of the Lie algebra of formal vector fields $W_N$ in $N$ variables. Then there is an element $\theta\in \mathrm{C}^1_{\mathrm{Lie}}(\widetilde{W}_N,\mathbb{C})\otimes_{\mathbb{C}} O_{\mathrm{BV}}$ such that the chain $e^{-\frac{\pi}{i\hbar}\theta}\Trace_{\widetilde{W}_N,\mathcal{A}^{\beta\gamma}_N}(\underline{\mathrm{\mathbf{1}}})\{-\}$  can be identified with an element in $\mathrm{C}_{\mathrm{Lie}}(W_N, \mathrm{GL}_N;\Omega_{\fO})$ which is cohomologous to the formal Witten genus
     $$
e^{-\frac{\pi}{i\hbar}\theta}\Trace_{\widetilde{W}_N,\mathcal{A}^{\beta\gamma}}(\mathrm{\underline{\mathbf{1}}})\{-\}      =\mathrm{Wit}_N(\tau)\quad \text{in}\  H_{\mathrm{Lie}}(W_N, \mathrm{GL}_N;\Omega_{\fO}).
     $$

     Here $\underline{\mathrm{\mathbf{1}}}$ is a normalized constant chiral chain defined in (\ref{ContantUnit}) and $\Omega_{\fO}$ is the space of formal differential forms in $N$ variables viewed as a $W_N$-module via Lie derivative.
  \end{thm}

 This work is motivated and influenced by several mathematical works  \cite{cattaneo2000path,costello2010geometric,2011WittenGenus,gorbounov2016chiral,kontsevich2003deformation,Li:2016gcb,1999Chiral} related to the study of two-dimensional quantum field theories. Kontsevich's deformation quantization \cite{kontsevich2003deformation} is interpreted as the two-dimensional Poisson sigma model by Cattaneo and Felder \cite{cattaneo2000path}.  The chiral de Rham complex \cite{1999Chiral} constructed by Malikov, Schechtman and Vaintrob is the global version of the free $\beta\gamma-bc$. Costello \cite{costello2010geometric,2011WittenGenus} gives a construction of the Witten genus via the 2d curved $\beta\gamma$-system. Later Gorbounov, Gwilliam and Williams \cite{gorbounov2016chiral} combine the BV quantization and formal geometry to study this model following the philosophy of \cite{2011WittenGenus} and obtain the formal Witten genus in the Lie algebra cohomology.  In \cite{Li:2016gcb}, the second named author of the present paper studied $\beta\gamma-bc$ systematically and established an exact connection between the solution of the renormalized quantum master equation and the Maurer-Cartan elements in the vertex algebra.
\subsection{A bit more on main results}\label{sec:expanded}

Here we explain the main results of this paper in a bit more detail. By applying the machinery of BV quantization to the 2d chiral $\beta\gamma-bc$ system on an elliptic curve $X$, one can get the following data
\begin{itemize}
   \item A graded vertex algebra $V^{\beta\gamma-bc}$ over the base field $\mathbf{k}:=\mathbb{C}((\hbar))$;
   \item  A BV algebra $(O_{\mathrm{BV}},\Delta_{\mathrm{BV}})$ over $\mathbb{C}$. The BV algebra we will use is $(\BVk= O_{\mathrm{BV}}\otimes_{\mathbb{C}}\mathbf{k},\hbar\Delta_{\mathrm{BV}})$. This BV algebra corresponds to the algebra of the zero modes in effective BV quantization;
    \item A map $\mathcal{W}$ from $\tilde{C}^{\mathrm{ch}}(X,V^{\beta\gamma-bc}\otimes \omega_X)_{\mathcal{Q}}$ to $\tilde{C}^{\mathrm{ch}}(X,\omega_X)_{\mathcal{Q}}\otimes_{\mathbb{C}}\BVk$, where  $\tilde{C}^{\mathrm{ch}}(X,-)_{\mathcal{Q}}$ is the chiral chain complex.
\end{itemize}

The map $\mathcal{W}$ can be viewed as a mathematical construction of the physical n-point correlation functions. One may construct the correlation function using representation theory or using perturbation theory. In this paper, we use the Feynman diagram construction in perturbative quantum field theory.

 Roughly speaking,  the Feynman diagram gives rise to a map that sends a tensor product of some vertex operators labeled by a finite set $I$ to a singular differential form on $X^{I}$. This singular differential form has possible meromorphic poles along the big diagonals. We can write
$$
\mathcal{W}: \{\text{Vertex operator insertions on}\  X\}\rightarrow \{\text{singular forms on } \ X^I\}
$$
$$
V_1(z_1)\otimes\cdots\otimes V_n(z_n)\mapsto \langle V_1(z_1)\cdots V_n(z_n) \rangle ,
$$
where $\langle V_1(z_1)\cdots V_n(z_n) \rangle $ is the n-point function. In the language of chiral algebras, the above constructions can be realized as a chain map (see Theorem \ref{mainThm})
$$
\boxed{\mathcal{W}:(\tilde{C}^{\mathrm{ch}}(X,\mathcal{A}^{\beta\gamma-bc})_{\mathcal{Q}},d)\rightarrow (\tilde{C}^{\mathrm{ch}}(X,\omega_X)_{\mathcal{Q}}\otimes_{\mathbb{C}}\BVk,d-\hbar\Delta_{\mathrm{BV}}).}
$$
Here $\mathcal{A}^{\beta\gamma-bc}$ is the chiral algebra corresponding to the $\beta\gamma-bc$ system $V^{\beta\gamma-bc}$ and $\tilde{C}(X,-)_{\mathcal{Q}}$ is the chiral chain complexes constructed in \cite{book:635773}, see  Definition (\ref{ChiralChain}).

According to \cite{book:635773}, there is a canonical trace map on unit chiral algebra $\omega_X$
$$
\mathrm{tr}:(\tilde{C}^{\mathrm{ch}}(X,\omega_X)_{\mathcal{Q}},d)\xrightarrow{\sim}\mathbb{C}
$$
which is a quasi-isomorphism. In Theorem \ref{RegularizedANDTrace}, we prove that this trace map can be realized as regularized integrals introduced in \cite{li2020regularized}.

Combining the above constructions, we get the so-called \textit{trace map} on chiral algebra $\mathcal{A}^{\beta\gamma-bc}$
$$
\boxed{\mathrm{\mathbf{Tr}}_{\mathcal{A}^{\beta\gamma-bc}}:=\mathrm{tr}\circ \mathcal{W}: (\tilde{C}^{\mathrm{ch}}(X,\mathcal{A}^{\beta\gamma-bc})_{\mathcal{Q}},d)\rightarrow (\BVk,-\hbar\Delta_{\mathrm{BV}}).}
$$
Using some other realizations of the chiral homology constructed in \cite{book:635773} and direct computations, we prove that $\mathrm{\mathbf{Tr}}_{\mathcal{A}^{\beta\gamma-bc}}$ is a quasi-isomorphism (see Corollary \ref{mainCor}). As we mentioned before, this is parallel to the one-dimensional case.

This construction can be directly extended to the Lie algebra cochain. One obtains the formal Witten genus in this case, see Section \ref{TraceWittenGenus}.

If we further choose a BV integration map $\int_{\mathrm{BV}}:(\BVk,\hbar\Delta_{\mathrm{BV}})\rightarrow (\mathbf{k},0)$, then we obtain a chain map (see Corollary \ref{mainCor}) which is a quasi-isomorphism
$$
\boxed{\mathrm{\mathbf{Tr}}^{\mathrm{BV}}_{\mathcal{A}^{\beta\gamma-bc}}:=\int_{\mathrm{BV}}\circ\mathrm{\mathbf{Tr}}_{\mathcal{A}^{\beta\gamma-bc}}:(\tilde{C}^{\mathrm{ch}}(X,\mathcal{A}^{\beta\gamma-bc})_{\mathcal{Q}},d) \rightarrow (\mathbf{k},0).}
$$

There is a degree shift on the right-hand side, but we ignore this issue here (see Corollary \ref{mainCor}). By abuse of notation, we continue to call $\mathrm{\mathbf{Tr}}^{\mathrm{BV}}_{\mathcal{A}^{\beta\gamma-bc}}$ the trace map.
\subsection{More general trace maps}
Now we turn to the construction of elliptic trace map for more general class of chiral algebras.
 Free field realization is a powerful technique in the study of 2d conformal field theory. Suppose the chiral algebra $\mathcal{A}$ admits a free field realization, that is, we have an injective morphism of chiral algebras
$$
\rho: \mathcal{A}\rightarrow \mathcal{A}^{\beta\gamma-bc}.
$$
 Since the construction of chiral chain complex is functorial, it induces a chain map
$$
\rho^{\mathrm{ch}}:(\tilde{C}^{\mathrm{ch}}(X,\mathcal{A})_{\mathcal{Q}},d)\rightarrow (\tilde{C}^{\mathrm{ch}}(X,\mathcal{A}^{\beta\gamma-bc})_{\mathcal{Q}},d).
$$
Composing it with $\mathrm{\mathbf{Tr}}^{\mathrm{BV}}_{\mathcal{A}^{\beta\gamma-bc}}$ we get a trace map for $\mathcal{A}$
$$
\boxed{\mathrm{\mathbf{Tr}}^{\mathrm{BV}}_{\rho^{\mathrm{ch}}}:=\mathrm{\mathbf{Tr}}^{\mathrm{BV}}_{\mathcal{A}^{\beta\gamma-bc}}\circ\rho^\mathrm{ch}: (\tilde{C}^{\mathrm{ch}}(X,\mathcal{A})_{\mathcal{Q}},d)\rightarrow (\mathbf{k},0).}
$$

Lastly, let us remark that our construction can be naturally generalized to chiral bosons. It simply changes the map $\mathcal{W}$ using Feynman rules from chiral bosons. In this paper we focus on the $\beta\gamma-bc$ system to avoid complications from too heavy notations, and leave chiral bosons for future study.

\subsection{Outline of the paper}
In Section \ref{sec:ChiralHomology} we review the notion of vertex operator algebra and chiral algebra, and present the construction of the chiral chain complex in detail. We prove that the canonical trace map on the unit chiral algebra is given by regularized integrals. In Section \ref{BV} we review the BV formalism and introduce the quantum master equation. In particular, a trace map of the chiral chain complex can be viewed as a solution to the quantum master equation. Next, we construct the trace map as a chain map from the elliptic chiral chain complex of the free $\beta\gamma-bc$ system and a BV algebra. Furthermore, we prove that the trace map is a quasi-isomorphism. Finally, we discuss some applications of the above constructions. We construct a trace map valued in the Lie algebra cohomology. In a special case, we compute the trace map evaluated on the unit constant chiral chain and find the formal Witten genus. The second application is the construction of the trace map on coset models.

\subsection{Conventions}\label{Conventions}

\begin{itemize}
\item In this paper, we denote by $X$  a smooth complex algebraic curve and $\omega_X$ the canonical sheaf. And $X^I$  stands for the Cartesian product
    $$
    X^I=\underbrace{X\times\cdots\times X}_{I}
    $$ of $X$ , where $I$ is a finite index set $I$. We also denote by $\Delta_{I}$ the big diagonal in $X^I$
    $$
    \Delta_I=\bigcup_{i,j\in I}\Delta_{ij}, \quad \text{where}\quad \Delta_{ij}=\{(\dots,z_{k},\dots)\in X^I| k\in I, z_i=z_j\}.
    $$

  \item   We will work in the analytic category in this paper. For $\mathcal{D}_X$-modules ($\mathcal{O}_X$-modules) on $X$, we mean analytic right $\mathcal{D}_{X^{\mathrm{an}}}$-modules ($\mathcal{O}_{X^{\mathrm{an}}}$-modules) on the analytification $X^{\mathrm{an}}$ and we omit the superscript $"\mathrm{an}"$ in this paper.

  \item Following the convention in \cite[pp.280]{book:635773}, for any quasi-coherent sheave $F$ on a smooth complex variety $X$, we denote $F_{\mathcal{Q}}$ to be $F\otimes_{\mathcal{O}_X}\mathcal{Q}_X$. Here $\mathcal{Q}_X$ is the Dolbeault complex
$$
\mathcal{Q}_{X^I}(U):\Omega^{0,0}(U)\xrightarrow{\bar{\partial}}\Omega^{0,1}(U)\xrightarrow{\bar{\partial}}\cdots.
$$
The Dolbeault complex $\mathcal{Q}_X$ can be also viewed as a complex of non-quasicoherent left $\mathcal{D}_X$-modules, since the holomorphic differential $\partial$ is an integrable connection and commutes with the Dolbeault differential $\bar{\partial}$.

\item Following the convention in \cite{book:635773}, we use $f_*$ for the push forward functor of $\mathcal{D}$-modules and $f_{\bullet}$ for the sheaf theoretical push forward.

\item We denote the de Rham complex (usually called the Spencer complex in the context of $\mathcal{D}$-module, here we follow the convention in \cite{book:635773}) of a $\mathcal{D}_X$-module by $\mathrm{DR}(M)$,  here
$$
\mathrm{DR}^i(M):=M\otimes_{\mathcal{O}_X}\wedge^{-i}\Theta_X,
$$
where $\Theta_X$ is the tangent sheaf. The differential $d_{\mathrm{DR}}$ is given by
\begin{align*}
  m\otimes\xi_1\wedge\cdots\wedge \xi_k &\xrightarrow{d_{\mathrm{DR}}}\sum_{i=1}^{k}(-1)^{i-1}m\xi_i\otimes \xi_1\wedge \cdots \wedge \hat{\xi}_i\wedge\cdots \wedge_k  \\
   & +\sum_{i<j}(-1)^{i+j}m\otimes [\xi_i,\xi_j]\wedge\xi_1\wedge\cdots \wedge\hat{\xi}_i\wedge\cdots \wedge\hat{\xi}_j\wedge \xi_k.
\end{align*}
\item Let $L$ be a $\mathbb{Z}$ or $\mathbb{Z}/2\mathbb{Z}$-graded $\mathbb{C}$-vector space. We denote $L_d=\{x\in L|\deg x=d\}\subset L$ the subspace of degree $d$ elements. Let $L\text{\small{[m]}}$ denote the degree shifting by $m$ such that
    $$
    L\text{\small{[m]}}_d=L_{d+m}.
    $$
    When $L$ is $\mathbb{Z}/2\mathbb{Z}$-graded, the above degree shifting is understood modulo $2\mathbb{Z}$, i.e.,
    $$
    L\shift_{\bar{0}}=L_{\bar{1}},\quad L\shift_{\bar{1}}=L_{\bar{0}}, \quad \text{where}\quad \bar{0},\bar{1}\in \mathbb{Z}/2\mathbb{Z}.
    $$
    There we also use $p(a)$ for the parity of a homogenous element $a\in L$.
\item For any graded ($\mathbb{Z}$ or $\mathbb{Z}/2\mathbb{Z}$) smooth vector bundle $E\rightarrow X$ over $X$, we will denote by $E^{\boxtimes I}\rightarrow X^I$ the bundle $\otimes_{i\in I}\pi^*_iE$, where $\pi_i: X^I\rightarrow X$ is the i-th projection. The bundle $E^{\boxtimes I}$ equips with a natural grading : $|e_1\boxtimes\cdots\boxtimes e_{|I|}|=|e_1|+\cdots+|e_{|I|}|$. The notation $E\text{\small{[k]}}$ means that we shift the degree of each fiber of $E$ by $k$.

\end{itemize}

\noindent\textbf{Acknowledgment.} The authors would like to thank Lei Fu, Xinxing Tang, Brian Williams, Kai Xu, Dingxin Zhang and Jie Zhou for helpful communications. S.~L.   thanks specially to Ziyu Cheng from Qiuzhen College, whose tremendous help has saved him alive from administrative service to finish this work. This work of S.~L.  is supported by the National Key R\&D Program of China  (NO. 2020YFA0713000).

\section{Chiral algebras and chiral homology}\label{sec:ChiralHomology}
\subsection{Vertex algebras and chiral algebras}
We begin with the definition of vertex algebras. We refer to the standard textbooks \cite{book:1415117,book1415599} for more details.

Let $V$ be a superspace, i.e., a vector space decomposed into a direct sum of two
subspaces:
$$
V=V_{\overline{0}}\oplus V_{\overline{1}}
$$
Here and further $\overline{0}$ and $\overline{1}$ stand for the cosets in $\mathbb{Z} / 2 \mathbb{Z}$ of 0 and $1 .$ We shall say that
an element $a$ of $V$ has parity $p(a) \in \mathbb{Z} / 2 \mathbb{Z}$ if $a \in V_{p(a)} .$

A field is a series of the form $a(z)=\sum_{n \in \mathbb{Z}} a_{(n)} z^{-n-1}$ where $a_{(n)} \in \operatorname{End} V$ and
for each $v \in V$ one has
$$
a_{(n)}(v)=0 \quad \text { for } \quad n \gg 0.
$$
\begin{defn}
A vertex algebra is the following data:
\begin{itemize}
  \item the space of states: a superspace $V$,
  \item the vacuum vector: a vector $|0\rangle \in V_{\overline{0}}$,
  \item the state-field correspondence: a parity preserving linear map
of $V$ to the space of fields, $a \mapsto Y(a, z)=\sum_{n \in \mathbb{Z}} a_{(n)} z^{-n-1}$,
satisfying the following axioms:
\end{itemize}

(translation covariance): $[T, Y(a, z)]=\partial Y(a, z)$,
where $T \in \mathrm{End} V$ is defined by
$T(a)=a_{(-2)}|0\rangle$;

(vacuum): $Y(|0\rangle, z)=\mathrm{Id}_{V},\left.Y(a, z)|0\rangle\right|_{z=0}=a$;

(locality): $(z-w)^{N} Y(a, z) Y(b, w)$
$$
=(-1)^{p(a) p(b)}(z-w)^{N} Y(b, w) Y(a, z) \text { for } N \gg 0 .
$$

If the $\mathbb{Z}_2$-grading on $V$ is induced from a $\mathbb{Z}-$grading such that $|0\rangle \in V_{0}$ and the state-field correspondence is of degree 0, we call $V$ a graded vertex algebra.

\end{defn}

Now we turn to the definition of chiral algebra. See \cite{book:635773} for details. Here we follow the presentations in \cite{2018Chiral,book:1415117,1999Notes}.

Let $X$ be a smooth complex curve. Throughout this paper, we work in the analytic category. Let $I$ be a finite set and $\Delta_I=\cup_{i,j\in I}\Delta_{ij}$ be the big diagonal, where $\Delta_{ij}=\{(\dots,z_{k},\dots)\in X^I| k\in I, z_i=z_j\}.$ For an $\mathcal{O}_{X^I}$-module $\mathcal{F}$, we define the $\mathcal{O}_{X^I}$-module $\mathcal{F}(*\Delta_I)$ to be
$$
\mathcal{F}(*\Delta_I)=\lim_{\rightarrow} \mathrm{Hom}_{\mathcal{O}_{X^I}}(\mathcal{I}^n,\mathcal{F}),
$$
here $\mathcal{I}$ is the defining ideal of $\Delta_I$.

For $\mathcal{D}$-modules, we always mean right $\mathcal{D}$-modules. See Section \ref{Conventions} for conventions.
\begin{defn}\label{chiralDefn}
Let $\mathcal{A}$ be a $\mathbb{Z}$-graded $\mathcal{D}_X$-module. A graded chiral algebra structure on $\mathcal{A}$ is a degree 0 $\mathcal{D}_{X^2}$-module map:
$$
\mu:(\mathcal{A}\boxtimes\mathcal{A})(*\Delta_{\{1,2\}})\rightarrow \Delta^{X\rightarrow X^2}_*(\mathcal{A}),\quad \ \Delta^{X\rightarrow X^2}:X\rightarrow X^2\ \text{the diagonal embedding}
$$
that satisfies the following two conditions:
\begin{itemize}
  \item Antisymmetry:

  If $f(z_1,z_2)\cdot a\boxtimes b$ is a local section of $\mathcal{A}\boxtimes\mathcal{A}(*\Delta_{\{1,2\}})$, then
  \begin{equation}\label{Antisymmetry}
  \mu(f(z_1,z_2)\cdot a\boxtimes b)=-(-1)^{p(a) p(b)}\sigma_{1,2}\mu(f(z_2,z_1)\cdot b\boxtimes a),
  \end{equation}
  where $\sigma_{1,2}$ acts on $\Delta^{X\rightarrow X^2}_*\mathcal{A}=\mathcal{A}\otimes_{{\mathcal{D}_{X}}}\mathcal{D}_{X\rightarrow X^2}=\mathcal{A}\otimes_{{\mathcal{D}_{X}}}(\mathcal{O}_{X}\otimes_{\mathcal{O}_{X^2}}\mathcal{D}_{X^2})$ by permuting two factors of $X^2$.

  \item Jacobi identity:

  If $a\boxtimes b\boxtimes c\cdot f(z_1,z_2,z_3)$ is a local section of $\mathcal{A}^{\boxtimes 3}(*\Delta_{\{1,2,3\}})$, then
  \begin{align*}
      \mu(\mu(f(z_1,z_2,z_3)\cdot &a\boxtimes b)\boxtimes c)+(-1)^{p(a)\cdot (p(b)+p(c))}\sigma_{1,2,3}\mu(\mu(f(z_2,z_3,z_1)\cdot b\boxtimes c )\boxtimes a)+  \\
     & (-1)^{p(c)\cdot (p(a)+p(b))}\sigma_{1,2,3}^{-1}\mu(\mu(f(z_3,z_1,z_2)\cdot c\boxtimes a )\boxtimes b)=0,
  \end{align*}
  here $\sigma_{1,2,3}$ denotes the cyclic permutation action on $\Delta^{X\rightarrow X^3}_*\mathcal{A}=\mathcal{A}\otimes_{\mathcal{D}_X}\mathcal{D}_{X\rightarrow X^3}$, $\Delta^{X\rightarrow X^3}:X\rightarrow X^3\ \text{is the diagonal embedding}$.
\end{itemize}
\end{defn}

\begin{rem}
  In \cite{book:635773}, the chiral operation is defined to be an algebraic $\mathcal{D}_{X^2}$-module map $\mu:j_*j^*\mathcal{A}\boxtimes\mathcal{A}\rightarrow \Delta_*\mathcal{A}$ satisfying the antisymmetry and Jacobi identity. Here $j:U=X^2-\Delta_{\{1,2\}}\hookrightarrow X^2$ is the open immersion.  Our definition is compatible with theirs because the analytification of $j_*j^*\mathcal{F}$ is $\mathcal{F}^{\mathrm{an}}(*\Delta_{\{1,2\}})$ for any coherent sheaf $\mathcal{F}$ on $X^2$.
\end{rem}

\begin{exa}
  The canonical sheaf $\omega_X$ is a chiral algebra and the chiral operation is given by the residual operation. It is also called unit chiral algebra.
\end{exa}

From a quasi-conformal vertex algebra, one can get a chiral algebra on a smooth complex curve $X$. See the construction in \cite{book:1415117}. Here we recall this construction briefly. Let $\mathrm{Aut}_{\mathbb{C}[[z]]}$ be the group of continuous automorphisms of the formal power series ring $\mathbb{C}[[z]]$. Any quasi-conformal vertex algebra is a left $\mathrm{Aut}_{\mathbb{C}[[z]]}$-module. The set of pairs $(x,t_x)$  consisting of a point $x\in X$ and a formal local coordinate $t_x$ at $x$ is an $\mathrm{Aut}_{\mathbb{C}[[z]]}$-torsor $\mathrm{Aut}_X$ over $X$.  Applying the associated bundle construction to the $\mathrm{Aut}_{\mathbb{C}[[z]]}$-module $V$ yields a vector bundle
$$
\mathcal{V}=\mathrm{Aut}_X\times_{\mathrm{Aut}_{\mathbb{C}[[z]]}} V
$$
over $X$. The bundle $\mathcal{V}$ carries a flat connection $\nabla:\mathcal{V}\rightarrow \mathcal{V}\otimes_{\mathcal{O}_X}\omega_X$, thus a left $\mathcal{D}_X$-module structure. The connection can be written locally as follows
$$
\nabla_{\partial_z}=\partial_z+L_{-1},
$$
but is independent of the choice of local coordinates. The corresponding right $\mathcal{D}_X$-module is $\mathcal{V}^r=\mathcal{V}\otimes_{\mathcal{O}_X}\omega_X$.

Locally there is a binary operation
$$
\mu:\mathcal{V}^r\boxtimes\mathcal{V}^r(*\Delta_{\{1,2\}})\rightarrow\Delta_*\mathcal{V}^r,
$$
$$
\mu(\frac{f(z_1,z_2)}{(z_1-z_2)^k}v_1 dz_1\boxtimes v_2 dz_2)=\sum_{n}\sum_{l\geq 0}\frac{1}{(n+k-l)!}(\partial^{n+k-l}_{z_1}f(z_1,z_2))|_{z_1=z_2=w}{v_1}_{(n)}v_2dw\otimes_{\mathcal{D}_{X_{12}}}\frac{1}{l!}\partial^l_{z_1},
$$
here $X_{12}$ is the diagonal of $X_1\times X_2$ and $w$ is the local coordinate. The tensor product $\otimes_{\mathcal{D}_{X_{12}}}$ indicates the $\mathcal{D}$-module relation $\partial_w=\partial_{z_1}+\partial_{z_2}$. Using the notations in \cite{2018Chiral}, we can rewrite the above expression in terms of the residue

\begin{equation}\label{OPECHIRAL}
\big(\mathrm{Res}_{z_1\rightarrow z_2}e^{(z_1-z_2)\otimes\vec{\partial}_{z_1}}\cdot\frac{f(z_1,z_2)}{(z_1-z_2)^k}\sum_{n\in \mathbb{Z}}{v_1}_{(n)}v_2\cdot (z_1-z_2)^{-n-1}dz_1 \big)\cdot dz_2\otimes_{\mathcal{D}_{X_{12}}} 1, \ z_2=w,
\end{equation}
here the notation $\vec{\partial}_{z_1}$ indicates moving all powers of $\partial_{z_1}$ to the right-hand side of the tensor product symbol $-\otimes_{\mathcal{D}_{X_{12}}}-$. More explicitly, we have
\[
(\ref{OPECHIRAL})="\wick{(\mathrm{Res}_{z_1\rightarrow z_2}\sum_{l\geq 0}\frac{(z_1-z_2)^l}{l!}\c3{\boxed{{\partial}^l_{z_1}}}\cdot \frac{f(z_1,z_2)}{(z_1-z_2)^k}\sum_{n\in\mathbb{Z}}v_{1(n)}v_2\cdot (z_1-z_2)^{-n-1} dz_1)\cdot dz_2\otimes_{\mathcal{D}_{X_{12}}} \c3{\mathop{1}\limits^{\blacktriangledown}}}"
\]

$$
=\sum_{n}\sum_{l\geq 0}\frac{1}{(n+k-l)!}(\partial^{n+k-l}_{z_1}f(z_1,z_2))|_{z_1=z_2=w}{v_1}_{(n)}v_2dw\otimes_{\mathcal{D}_{X_{12}}}\frac{1}{l!}\partial^l_{z_1}.
$$

The operation $\mu$ is independent of the choice of local coordinates, see \cite{book:1415117} (based on the coordinate change formula in \cite{huang2012two}). Moreover, the binary operation $\mu$ satisfies the antisymmetry and Jacobi identity. Thus the $\mathcal{D}_X$-module $\mathcal{V}^r$ is a chiral algebra.

Now we extend the chiral operation which is defined on the space of analytic sections to a certain space of smooth sections. Let $\mathcal{Q}_{X^I}$ be the Dolbeault resolution of the analytic functions on $X^I$
$$
\mathcal{Q}_{X^I}(U):\Omega^{0,0}(U)\xrightarrow{\bar{\partial}}\Omega^{0,1}(U)\xrightarrow{\bar{\partial}}\cdots,
$$
where $U\subset X^I$ is an open subset. For any quasi-coherent sheave $F$ on $X^I$, we denote $F^{\mathrm{an}}_{\mathcal{Q}}$ to be $F^{\mathrm{an}}\otimes_{\mathcal{O}^{\mathrm{an}}_{X^I}}\mathcal{Q}_{X^I}$. We will omit the superscript "an" in this paper. See Section \ref{Conventions} for conventions.
\begin{defn}\label{ChiralDolbeault}
The chiral operation $\mu$ induces a binary operation $$\mu^{\mathcal{Q}}:\mathcal{V}^r\boxtimes\mathcal{V}^r(*\Delta_{\{1,2\}})_{\mathcal{Q}}=\mathcal{V}^r\boxtimes\mathcal{V}^r(*\Delta_{\{1,2\}})\otimes_{\mathcal{O}_{X^2}}\mathcal{Q}_{X^2}\rightarrow\Delta_*(\mathcal{V}^r_{\mathcal{Q}})=\Delta_*(\mathcal{V}^r\otimes_{\mathcal{O}_X}\mathcal{Q}_X).
$$
A local section of $\mathcal{V}^r\boxtimes\mathcal{V}^r(*\Delta_{\{1,2\}})_{\mathcal{Q}}$ can be written as $$\frac{f(z_1,z_2)}{(z_1-z_2)^k}\cdot\eta_1(z_1)\cdot\eta_2(z_2)\cdot v_1dz_1\boxtimes v_2dz_2,$$ where $f(z_1,z_2)$ is a smooth function and $\eta_i(z_i)$'s are smooth sections of $\mathcal{Q}_{X_i}$. We can define
$$
\mu^{\mathcal{Q}}(\frac{f(z_1,z_2)}{(z_1-z_2)^k}\cdot\eta_1(z_1)\cdot\eta_2(z_2)\cdot v_1dz_1\boxtimes v_2dz_2)
$$
$$
=(-1)^{|\eta_1|+|\eta_2|}\sum_{n}\sum_{l\geq 0}\frac{(\partial^{n+k-l}_{z_1}(f(z_1,z_2)\cdot \eta_{1}(z_1)))|_{z_1=z_2=w}}{(n+k-l)!}\cdot \eta_2(z_2)|_{z_2=w}{v_1}_{(n)}{v_2}dw\otimes_{\mathcal{D}_{X_{12}}}\frac{\partial^l_{z_1}}{l!}
$$
Here we use the super commutative algebra structure on $\mathcal{Q}_{X_{12}}$. To get an expression similar to (\ref{OPECHIRAL}), one needs to extend the definition of the residue (see Section \ref{Tracemap}). The operation $\mu^{\mathcal{Q}}$ is independent of the choice of coordinates by the same argument.
\end{defn}

\begin{rem}
The above construction is discussed in \cite[pp.309]{book:635773}.
\end{rem}
\subsection{The chiral chain complex $\tilde{C}^{\mathrm{ch}}(X,-)_{\mathcal{Q}}$}\label{ChiralHomology}
 Locally, chiral algebra or vertex operator algebra corresponds to the operator formalism in two-dimensional conformal field theory and encodes the algebraic structure of the operator product expansions.   A global object associated to a chiral algebra called chiral homology is defined by \cite{book:635773}. The 0-th chiral homology, also known as the space of conformal blocks, has been extensively studied by mathematicians and physicists. For mathematical discussions, see \cite{book:1415117,1996Modular}. Various chiral chain complexes computing chiral homology are constructed in \cite{book:635773}. A more explicit chain complex computing the 1-st and 2-nd chiral homology on an elliptic curve is constructed by \cite{2018Chiral,2021FirstCHiral} and they study the space of the first elliptic chiral homology groups in \cite{2021FirstCHiral}. We also mention that a version of the (co)chain complex for conformal algebras on the complex plane $X=\mathbb{C}$ is constructed in \cite{bakalov1999cohomology} (see also \cite{book1415599}).

In this section, we review the definition of a version of the chiral chain complex following \cite{book:635773}. We first recall the Chevalley-Cousin complex $C(\mathcal{A})$ of a chiral algebra $\mathcal{A}$ on a complex curve $X$.

 For a finite index set $I$, we will denote by ${\textsf{Q}}(I)$ the set of equivalence relations on $I$. Note that $\equiQ(I)$ is a disjoint union $\sqcup_{n\geq 1}\equiQ(I,n)$, where $\equiQ(I,n)$ is the set of equivalence relations on $I$ such that the number of the equivalence classes is $n$. An element of $\equiQ(I,n)$ can be described by a pair $(T,\pi:I\twoheadrightarrow T)$, where $T$ is a finite set such that $|T|=n$ and $\pi$ is a surjection. We denote $X^I$ to be the product of $X_i$'s, where $X_i$ is a copy of $X$ and $i\in I$.

 Given a surjection $\pi:I\twoheadrightarrow T$,  we have the corresponding diagonal embedding
 $$
 \Delta^{(\pi)}=\Delta^{(I/T)}:X^T\hookrightarrow X^I,
 $$
 such that $\mathrm{pr}_i\Delta^{(I/T)}=\mathrm{pr}_j\Delta^{(I/T)}$ if and only if $\pi(i)=\pi(j)$, here $\mathrm{pr}_i: X^I\rightarrow X$ is the i-th projection.

\begin{rem}
  With this notation, $\Delta^{X\rightarrow X^2}$ in Definition \ref{chiralDefn} is the same as $\Delta^{(\{1,2\}/\{\bullet\})}:X\rightarrow X^2$.
\end{rem}
 Let $\mathcal{A}$ be a graded chiral algebra. Now we define a $\mathbb{Z}$-graded $\mathcal{D}_{X^I}-$module $C(\mathcal{A})_{X^I}$ on $X^I$
$$
\boxed{C(\mathcal{A})^{\bullet}_{X^I}=\bigoplus_{T\in \equiQ(I)} \Delta^{(I/T)}_*\big((\mathcal{A}\shift)^{\boxtimes T}(*\Delta_T)\big).}
$$
The differential looks as follows. Its component
$$
d_{\mathrm{ch}, (T,T')}: \Delta_*^{(I/T)}\big((\mathcal{A}\shift)^{\boxtimes T}(*\Delta_T)\big)\rightarrow\Delta_*^{(I/T')}\big((\mathcal{A}\shift)^{\boxtimes T'}(*\Delta_{T'})\big)
$$
can be non-zero only for $T'\in \equiQ(T,|T|-1)$. Then $T=T''\bigsqcup\{\alpha',\alpha''\}, T'=T''\bigsqcup\{\alpha\}$ and $d_{\mathrm{ch},(T,T')}$ is the exterior tensor product of the chiral operation
$$
d_{\mathrm{ch},(T,T')}=\mu_{\mathcal{A}}\shift: (\mathcal{A}_{\alpha'}\shift\boxtimes \mathcal{A}_{\alpha''}\shift)(*\Delta_{\{\alpha',\alpha''\}})\rightarrow\Delta_*\mathcal{A}_{\alpha}\shift
$$
and the identity map for $\mathcal{A}^{\boxtimes T''}$.

\begin{rem}
  This definition is exactly the reduced Chevalley-Eilenberg complex discussed in the introduction (see Section \ref{IntroChiral}).
\end{rem}

 We can extend the above construction of the Chevalley-Cousin complex to the case with coefficient  $\mathcal{Q}_{X^\bullet}$
$$
\boxed{C(\mathcal{A})_{\mathcal{Q},X^I}:=\bigoplus_{T\in \equiQ(I)} \Delta^{(I/T)}_*((\mathcal{A}\shift)^{\boxtimes T}(*\Delta_T)_\mathcal{Q}).}
$$
where the differential is $$d^{\mathcal{Q}}_{\mathrm{ch},(T,T')}:=\mu^{\mathcal{Q}}_{\mathcal{A}}\shift:(\mathcal{A}_{\alpha'}\shift\boxtimes \mathcal{A}_{\alpha''}\shift)(*\Delta_{\{\alpha',\alpha''\}})_\mathcal{Q}\rightarrow \Delta_*((\mathcal{A}_{\alpha}\shift)_\mathcal{Q}),$$
see the Definition \ref{ChiralDolbeault}. More precisely,
$$
d^{\mathcal{Q}}_{\mathrm{ch},(T,T')}(\frac{f(z_1,z_2)}{(z_1-z_2)^k}\cdot\eta_1(z_1)\cdot\eta_2(z_2)\cdot v_1dz_1\shift\boxtimes v_2dz_2\shift)
$$
$$
=(-1)^{p(v_1)}\sum_{n}\sum_{l\geq 0}\frac{(\partial^{n+k-l}_{z_1}(f(z_1,z_2)\cdot \eta_{1}(z_1)))|_{z_1=z_2=w}}{(n+k-l)!}\cdot \eta_2(z_2)|_{z_2=w}{v_1}_{(n)}{v_2}dw\shift\otimes_{\mathcal{D}_{X_{12}}}\frac{\partial^l_{z_1}}{l!}
$$
 Since we mainly deal with $(C(\mathcal{A})_{\mathcal{Q},X^I},d^{\mathcal{Q}}_{\mathrm{ch}})$, from now on we omit the superscript ${\mathcal{Q}}$ of the differential $d^{\mathcal{Q}}_{\mathrm{ch},(T,T')}$.
\begin{rem}
  Here we view a smooth form in $\Omega^{0,\bullet}$ as a degree $\bullet$ element.
\end{rem}
Let $\mathcal{S}$ be the category of finite non-empty sets whose morphisms are surjections and $\mathcal{S}^{\mathrm{op}}$ be the opposite category. We now recall the notion of !-sheaf on $X^{\mathcal{S}}$.
\begin{defn}
A !-sheaf $F$ on $X^{\mathcal{S}}$ consists of a sheaf $F_{X^I}$ of vector spaces on $X^I$ for each finite set $I$ and a morphism $\theta^{(\pi)}:\Delta^{(\pi)}_{\bullet}(F_{X^I})\rightarrow F_{X^J}$ for each surjections $\pi:J\twoheadrightarrow I$, subject to the compatibility conditions
$$
\theta^{(\pi_1\circ\pi_2)}=\theta^{(\pi_2)}\circ \Delta_{\bullet}^{(\pi_2)}(\theta^{(\pi_1)}),
$$
for all $\pi_2:K\twoheadrightarrow J, \pi_1:J\twoheadrightarrow I$ and $\theta^{(\mathrm{id}_I)}=\mathrm{id}$.
We denote by $CSh^!(X^{\mathcal{S}})$ the category of complexes of !-sheaves. By abuse of notation, we continue to call an element in $CSh^!(X^{\mathcal{S}})$ a !-sheaf.
\end{defn}

Given a !-sheaf $F$ on $X^{\mathcal{S}}$ yields an $\mathcal{S}^{\mathrm{op}}$-diagrams of vector space, $I\mapsto \Gamma(X^I,F_{X^I})$. Denote by $\Gamma(X^{\mathcal{S}},F)$ its inductive limit.

According to \cite{book:635773}, it is necessary to consider a nice resolution $\tilde{F}\rightarrow F$ for a given !-sheaf on $X^{\mathcal{S}}$.
\begin{prop}{\cite{book:635773}}
  Every $F\in CSh^!(X^{\mathcal{S}})$ admits a canonical nice resolution $\tilde{F}\rightarrow F$.  In fact, $\tilde{F}_{X^I}$ can be chosen to be the homotopy direct limit $\hoco(\equiQ(I),\Delta^{(I/T)}_{\bullet}F_{X^T})$ where the structure morphisms are embeddings $\theta^{(I/J)}:\Delta^{(I/J)}_{\bullet}\tilde{F}_{X^J}\hookrightarrow\tilde{F}_{X^I}$ coming from the embedding $\equiQ(J)\subset \equiQ(I)$.
\end{prop}

We now give an explicit definition of the homotopy direct limit $\hoco(\equiQ(I),\Delta^{(I/T)}_{\bullet}F_{X^T})$. This material is standard, see \cite{dold1961homologie}.
\begin{defn}
  For each sequence of surjections $(I_0\twoheadleftarrow I_1\twoheadleftarrow\cdots\twoheadleftarrow I_n)$ in the classifying space $B\equiQ(I)$ of $\equiQ(I)$, define
  $$
  F_{(I_0\twoheadleftarrow I_1\twoheadleftarrow \cdots\twoheadleftarrow I_n)}=F_{I_0}=\Delta^{(I/I_0)}_{\bullet}F_{X^{I_0}}.
  $$
  Set $\mathcal{F}_n=\mathop{\bigoplus}\limits_{(I_0\twoheadleftarrow I_1\twoheadleftarrow \cdots\twoheadleftarrow I_n)\in B\equiQ(I)}F_{(I_0\twoheadleftarrow I_1\twoheadleftarrow \cdots\twoheadleftarrow I_n)}$, define the boundary maps
  $$
  \partial_{i}:\mathcal{F}_{n}\rightarrow \mathcal{F}_{n-1}, \ 1\leq i\leq n.
  $$
  $$
  \partial_i:=\mathop{\bigoplus}\limits_{I_i,(I_0\twoheadleftarrow I_1\twoheadleftarrow \cdots\twoheadleftarrow I_n)\in B\equiQ(I)}\mathrm{id}(F_{(I_0\twoheadleftarrow I_1\twoheadleftarrow \cdots\twoheadleftarrow I_n)}\rightarrow F_{(I_0\twoheadleftarrow I_1\twoheadleftarrow \cdots\twoheadleftarrow \hat{I}_i\twoheadleftarrow \cdots\twoheadleftarrow I_n)})
  $$
  where $\mathrm{id}(F_{(I_0\twoheadleftarrow I_1\twoheadleftarrow \cdots\twoheadleftarrow I_n)}\rightarrow F_{(I_0\twoheadleftarrow I_1\twoheadleftarrow \cdots\twoheadleftarrow \hat{I}_i\twoheadleftarrow \cdots\twoheadleftarrow I_n)})$ is the identity map from the component $F_{(I_0\twoheadleftarrow I_1\twoheadleftarrow \cdots\twoheadleftarrow I_n)}=F_{I_0}$ in $\mathcal{F}_n$ to the component $F_{(I_0\twoheadleftarrow I_1\twoheadleftarrow \cdots\twoheadleftarrow \hat{I}_i\twoheadleftarrow \cdots\twoheadleftarrow I_n)}=F_{I_0}$ in $\mathcal{F}_{n-1}$.
    And $\partial_0:\mathcal{F}_n\rightarrow\mathcal{F}_{n-1}$ is defined to be the map
    $$
    \partial_0=\mathop{\bigoplus}\limits_{(I_0\twoheadleftarrow I_1\twoheadleftarrow \cdots\twoheadleftarrow I_n)\in B\equiQ(I)}{\Delta^{(I/I_1)}_{\bullet}\theta^{(I_1/I_0)}}(F_{(I_0\twoheadleftarrow I_1\twoheadleftarrow \cdots\twoheadleftarrow I_n)}\rightarrow F_{(I_1\twoheadleftarrow \cdots\twoheadleftarrow I_n)}),
    $$
    here the $\Delta^{(I/I_1)}_{\bullet}\theta^{(I_1/I_0)}(F_{(I_0\twoheadleftarrow I_1\twoheadleftarrow \cdots\twoheadleftarrow I_n)}\rightarrow F_{(I_1\twoheadleftarrow \cdots\twoheadleftarrow I_n)})$ is the structure morphism from the component $F_{(I_0\twoheadleftarrow I_1\twoheadleftarrow \cdots\twoheadleftarrow I_n)}=F_{I_0}=\Delta^{(I/I_0)}_{\bullet}F_{X^{I_0}}=\Delta^{(I/I_1)}_{\bullet}\Delta^{(I_1/I_0)}_{\bullet}F_{X^{I_0}}$ in $\mathcal{F}_n$ to the component $F_{(I_1\twoheadleftarrow \cdots\twoheadleftarrow I_n)}=F_{I_1}=\Delta^{(I/I_1)}_{\bullet}F_{X^{I_1}}$ in $\mathcal{F}_{n-1}$.

    Define the normalization complex
  $$
  \mathrm{Norm}(\equiQ(I),F)^{-n}:=\cap^{n}_{i=1}\mathrm{Ker}(\partial_i:\mathcal{F}_{n}\rightarrow\mathcal{F}_{n-1}),\quad n\geq 1,\quad
   $$
   $$
   \mathrm{Norm}(\equiQ(I),F)^{0}=\bigoplus_{I_0\in \equiQ(I)}\Delta^{(I/I_0)}_{\bullet}F_{X^{I_0}}
  $$
  $$
  \mathrm{Norm}(\equiQ(I),F)^{-n}:=0,\quad n<0,\quad
   $$
  with the differential $d_{\mathrm{Norm}}=\partial_{0}|_{ \mathrm{Norm}(\equiQ(I),F)^{-n}}: \mathrm{Norm}(\equiQ(I),F)^{-n}\rightarrow \mathrm{Norm}(\equiQ(I),F)^{-n+1}$. The homotopy direct limit is defined to be the corresponding total complex
  $$
  \hoco(\equiQ(I),\Delta^{(I/T)}_{\bullet}F_{X^T}):=\mathrm{tot} \mathrm{Norm}(\equiQ(I),F).
  $$

Here for a complex $F$, the above total complex is
$$
\mathrm{tot} \mathrm{Norm}(\equiQ(I),F)=\oplus_{n\in\mathbb{Z}} \mathrm{Norm}(\equiQ(I),F)^n[-n],
$$
and the differential is the sum of $d_{\mathrm{Norm}}$ and the structure differentials of $\mathrm{Norm}(\equiQ(I),F)^n[-n]$.
\end{defn}

Similar to the discussion on !-sheaf, a right $\mathcal{D}-$module $M$ on $X^{\mathcal{S}}$ consists of a right $\mathcal{D}_{X^I}-$module $M_{X^I}$ for each finite set $I$ and a morphism $\theta^{(\pi)}:\Delta^{(\pi)}_*(M_{X^I})\rightarrow M_{X^J}$ for each surjections $\pi:J\rightarrow I$, subject to the compatibility conditions
$$
\theta^{(\pi_1\pi_2)}=\theta^{(\pi_2)}\circ \Delta_{*}^{(\pi_2)}(\theta^{(\pi_1)})\ \ \text{and}\ \ \theta^{(\mathrm{id}_I)}=\mathrm{id}.
$$

Now consider the de Rham complex $\mathrm{DR}(M)$ as a complex of !-sheaves on $X^{\mathcal{S}}$, that is, we define $\mathrm{DR}(M)_{X^I}$ to be $\mathrm{DR}(M_{X^I})$ and the structure morphism is the obvious one. Now we are ready to introduce the chiral chain complex.

\begin{defn} {\cite{book:635773}}\label{ChiralChain}
  Let $\mathcal{A}$ be a chiral algebra on $X$. The chiral chain complex is defined as follows
  $$
\boxed{\tilde{C}^{\mathrm{ch}}(X,\mathcal{A})^{-\bullet}_{\mathcal{Q}}:=\Gamma(X^{\mathcal{S}},P)=  \lim _{\rightarrow } (I\in\mathcal{S}, \Gamma(X^I, P^{\bullet}_{X^I})),\ P^{\bullet}_{X^I}=\hoco(\equiQ(I),\Delta^{(I/T)}_{\bullet}F_{X^T} )^{\bullet},}
  $$
  where $F_{X^I}=\mathrm{DR}(C(\mathcal{A})_{\mathcal{Q},X^I})=\mathrm{DR}\big(\mathop{\bigoplus}\limits_{T\in \equiQ(I)}\Delta_*^{(I/T)}((\mathcal{A}\shift)^{\boxtimes T}(*\Delta_T)_\mathcal{Q})\big)$. The differential is
  $$
  d_{\mathrm{tot}}=d_{\mathrm{DR}}+\bar{\partial}+2\pi id_{\mathrm{ch}}+d_{\mathrm{Norm}}.
   $$
 The coefficient $2\pi i $ is not important, see Remark \ref{Rescaling}.
\end{defn}
\begin{rem}
  In \cite[pp.310]{book:635773}, they use $F_{X^T}=\mathrm{DR}\big(\mathop{\bigoplus}\limits_{T\in \equiQ(I)}\Delta_*^{(I/T)}((j_*^{(T)}j^{(T)*}(\mathcal{A}\shift)^{\boxtimes T})_\mathcal{Q})\big)$, where $j^{(T)}:X^T-\Delta_T\hookrightarrow X^T$ is the open immersion. In this paper, we use the analytic version of their chiral chain complex $\tilde{C}^{\mathrm{ch}}(X,-)_{\mathcal{Q}}$ and denote it by the same notation.
\end{rem}

 From the definition of the chiral chain complex, the homological grading is
$$
\mathrm{hodeg}(\alpha)=-\mathrm{deg}(\alpha)=-q+|T|-p-|a|,\quad\alpha\in \Omega^{0,q}(X^I,\mathrm{DR}^p\Delta_*^{(I/T)}(\mathcal{A}\shift)^{\boxtimes T}(*\Delta_T)),
$$
here we write
$$
\alpha=\eta\cdot a,\quad\eta\in \Omega^{0,q}(X^I,\mathrm{DR}^p\Delta_*^{(I/T)}(\omega_X\shift)^{\boxtimes T}(*\Delta_T)), a\in \Omega^{0,0}(X^I,\mathcal{V}^{\boxtimes I}).
$$
\begin{defn}
  {\cite{book:635773}} The chiral homology is defined as the homology the chiral chain complex
$$
H^{\mathrm{ch}}_{i}(X,\mathcal{A})=H_i\left(\tilde{C}^{\mathrm{ch}}(X,\mathcal{A})^{\bullet}_{\mathcal{Q}}\right).
$$
\end{defn}

\begin{rem}
  The original construction of \cite{book:635773} works for general differential graded chiral algebras $(\mathcal{A},d_{\mathcal{A}})$, here we only need the case when $d_{\mathcal{A}}=0$.
\end{rem}
\subsection{Regularized integrals and the unit chiral chain complex }\label{Tracemap}

In this section, we will show that the canonical trace map $$\mathrm{tr}:\tilde{C}^{\mathrm{ch}}(X,\omega_X)_{\mathcal{Q}}\xrightarrow{\sim}
H_0^{\mathrm{ch}}(X,\omega_X)\cong \mathbb{C}$$
  is exactly given by regularized integrals introduced in \cite{li2020regularized}.

The definition of the regularized integral is reviewed in Appendix \ref{ReInt}. Here we recall some important properties of regularized integrals.

\begin{itemize}
  \item Let $D\subset X$ be a finite set of points on $X$. The regularized integral $\dashint$ is a $\mathbb{C}-$linear map
      $$
      \dashint_X: \Gamma(X,\omega_X(*D)_\mathcal{Q})=\Omega^{0,\bullet}(X,\omega_X(*D))\rightarrow \mathbb{C}.
      $$
  \item $\dashint_X$ is well defined on the quotient space $$\frac{\Gamma(X,\omega_X(*D)_\mathcal{Q})}{d_{\mathrm{DR}}(\Gamma(X,(\omega_X(*D)\otimes_{\mathcal{O}_X}\Theta_X)_{\mathcal{Q}}))}=\frac{\Gamma(X,\omega_X(*D)_\mathcal{Q})}{\partial\Gamma(X,\mathcal{O}_X(*D)_{\mathcal{Q}}))},
      $$ namely,  $\dashint_X$ vanishes on holomorphic total derivatives

      \begin{equation}\label{Killholo}
      \dashint_X \partial \eta=0, \eta \in \Gamma(X,\mathcal{O}_X(*D)_\mathcal{Q}).
      \end{equation}
  \item We have
  $$
  \dashint_X\bar{\partial}\eta=-2\pi i\sum_{p\in D}\mathrm{Res}_{z\rightarrow p}\eta,
  $$
  here the $\mathrm{Res}_{z\rightarrow p}\eta$ is a generalization of the notion of residue for meromorphic forms. For the precise definition, see Appendix \ref{ReInt}.  Let $p\in D$ and $z$ be a local coordinate around $p$ such that $z(p)=0$. Then $\eta$ can be written as $\frac{f(z)}{z^{k+1}}dz$ around $p$, here $f(z)$ is a smooth function and $z(p)=0$, then
  $$
  \mathrm{Res}_{z\rightarrow p}\eta=\frac{\partial^{k}_zf(0)}{k!}.
  $$
 \item A Fubini-type theorem is satisfied by $\dashint$. In other words, one can define a regularized integral
     $$
     \dashint_{X^I}:\Gamma(X^I,\omega_{X^I}(*\Delta_{I})_\mathcal{Q})\rightarrow \mathbb{C}
     $$
     which vanishes on the subspace $\mathrm{Im}(d_{\mathrm{DR}})\subset \Gamma(X^I,\omega_{X^I}(*\Delta_{I})_\mathcal{Q})$.
     And the above integral is equal to the iterated integral
     $$
     \dashint_{X^I}=\dashint_{X_1}\cdots \dashint_{X_{|I|}},
     $$
     here we choose an identification $I\simeq\{1,\dots,|I|\}$ to get $X^I\simeq X_1\times\cdots\times X_{|I|}$ and $X_i$ is a copy of $X$.
\end{itemize}
We also use the notation $\mathrm{Res}_{z_{\alpha'}\rightarrow z_{\alpha''}} (\alpha',\alpha''\in I)$ to denote the residue operation on $\Gamma(X^I,\omegaQI)$
$$
\mathrm{Res}_{z_{\alpha'}\rightarrow z_{\alpha''}}:\Gamma(X^I,\omega_{X^I}(*\Delta_I)_\mathcal{Q})\rightarrow \Gamma(X^{I'},\omega_{X^{I'}}(*\Delta_{I'})_\mathcal{Q}),
$$
here $(I',\pi :I\twoheadrightarrow I')\in \equiQ(I)$ with $|I'|=|I|-1$ and $\pi(\alpha')=\pi(\alpha'')$. See Appendix \ref{ReInt} for details. Some important properties of this residue operation are summarized in the following proposition.

\begin{prop}\label{Stokes}
For any $(I', \pi:I\twoheadrightarrow I')\in \equiQ(I,|I|-1)$, let $\alpha'\neq\alpha''\in I$ such that $\pi(\alpha')=\pi(\alpha'')$. We write $\alpha'(I'), \alpha''(I')$ to emphasis the dependence on $I'\in \equiQ(I,|I|-1)$. Then residues $\mathrm{Res}_{z_{\alpha'(I')}\rightarrow z_{\alpha''(I')}}$  and $\mathrm{Res}_{z_{\alpha''(I')}\rightarrow z_{\alpha'(I')}}$ are equal in the quotient space $\frac{\Gamma(X^{I'},\omegaQIP)}{\mathrm{Im}(d_{\mathrm{DR}})}$. Thus, we have
$$
\dashint_{X^{I'}}\mathrm{Res}_{z_{\alpha'(I')}\rightarrow z_{\alpha''(I')}}\eta=\dashint_{X^{I'}}\mathrm{Res}_{z_{\alpha''(I')}\rightarrow z_{\alpha'(I')}}\eta.
$$
We denote the corresponding equivalent class by $\mathrm{Res}_{I,I'}\eta$. We have the following version of the Stokes theorem
\begin{equation}\label{StokesThm}
\dashint_{X^{I}}(\bar{\partial}+\partial)\eta=\dashint_{X^{I}}\bar{\partial}\eta=-2\pi i\sum_{I'\in \equiQ(I,|I|-1)}\dashint_{X^{I'}}\mathrm{Res}_{I,I'}\eta.
\end{equation}

\end{prop}
\begin{proof}
  See Appendix \ref{ReInt}.
\end{proof}
\begin{rem}
  When $\eta$ has only logarithmic singularities, the above proposition reduces to the Stokes theorem for the usual integrals.
\end{rem}
\begin{rem}\label{Rescaling}
  One can rescale the definition of the regularized integral by a non-zero constant $a\in \mathbb{C}^{\times}$
  $$
  \dashint^{a}_X:=a\cdot\dashint_X.
  $$
  Then $\dashint^{a}_{X^n}=a^n\cdot \dashint_{X^n}$. If we choose $a=\frac{1}{2\pi i}$, then
  $$
  \dashint^a_{X^{I}}\bar{\partial}\eta=-\sum_{I'\in \equiQ(I,|I|-1)}\dashint^a_{X^{I'}}\mathrm{Res}_{I,I'}\eta.
  $$
  Later we will see that the correct differential on chiral chain complex will be $d_{\mathrm{DR}}+\bar{\partial}+d_{\mathrm{ch}}+d_{\mathrm{Norm}}$ if we use $\dashint^{a}_X$ as our definition of the regularized integral.
\end{rem}

Now let $I=\{1,\dots,n\}$ and $i,j\in I$. In terms of local coordinates, we write $\eta=\frac{f(z_1,\dots,z_n)}{(z_i-z_j)^{k+1}}dz_i\wedge\eta'\wedge\eta''$, here $\eta'$ is a smooth section of $\mathcal{Q}_{X_i}$ , $\eta''\in \Gamma(X^{n-1},\omega_{X^{n-1}}(*\Delta_{\{1,\dots,n-1\}})_{\mathcal{Q}_{X^{n-1}}})$ and $f$ is a smooth function. Then
\begin{equation}\label{Local}
  \mathrm{Res}_{z_i\rightarrow z_j}\eta=\sum_{l}\frac{\partial^{k-l}_{z_i}f(z_1,\dots,z_n)|_{z_i=z_j}}{(k-l)!}\cdot \frac{ \partial^{l}_{z_i}\eta'|_{z_i=z_j}}{l!}\cdot \eta''.
\end{equation}
\begin{rem}
With this notation, for a local section $\eta_0\cdot v_1dz_1\boxtimes v_2dz_2$ of $\mathcal{V}^r\boxtimes\mathcal{V}^r(*\Delta_{\{1,2\}})_\mathcal{Q}$,  we have
$$
\mu^{\mathcal{Q}}(\eta_0\cdot v_1dz_1\boxtimes v_2dz_2)=
$$
$$
\big(\mathrm{Res}_{z_1\rightarrow z_2}e^{(z_1-z_2)\otimes\vec{\partial}_{z_1}}\cdot\eta_0\cdot \sum_{n\in\mathbb{Z}}{v_1}_{(n)}v_2(z_1-z_2)^{-n-1}dz_1\big)\cdot dz_2\otimes_{\mathcal{D}_{X_{12}}} 1, \ z_2=w,
$$
which is an extension of (\ref{OPECHIRAL}). Here $z_2=w$ means that we substitute $z_2$ by  $w$ which is the local coordinate of the diagonal.
\end{rem}

It is shown in \cite[pp.315,Proposition.4.3.3(i)]{book:635773} that the chiral homology of the unit chiral algebra $\omega_X$ is one-dimensional and concentrated in degree 0, that is, the projection
$$
\tilde{C}^{\mathrm{ch}}(X,\omega_X)_{\mathcal{Q}}\rightarrow
H_0^{\mathrm{ch}}(X,\omega_X)\simeq \mathbb{C}
$$
is a quasi-isomorphism. We denote the projection map by $\mathrm{tr}.$ Here we normalize the second isomorphism such that $\mathrm{tr}(\mathrm{vol}_X\shift)=\int_X\mathrm{vol}_X\shift=1,\mathrm{vol}_X\shift\in \Omega^{0,1}(X,\omega_X\shift)$ is a volume form (it is a degree 0 element in the chiral chain complex).

Recall that $\dashint$ vanishes on $\mathrm{Im}(d_{\mathrm{DR}})$. The same property holds for the trace map $\mathrm{tr}$.

\begin{lem}\label{VanishOnDR}
  The trace map $\mathrm{tr}$ annihilates the $d_{\mathrm{DR}}$-exact terms in $\tilde{C}^{\mathrm{ch}}(X,\omega_X)_{\mathcal{Q}}$.
\end{lem}
\begin{proof}
  Since $\mathrm{tr}$ is the projection to the 0-th homology, we only need to consider the degree 0 component of $\tilde{C}^{\mathrm{ch}}(X,\omega_X)_{\mathcal{Q}}$. Suppose we have $\alpha=d_{\mathrm{DR}}\eta$, such that
  $$
  \eta\in \Omega^{0,|T|}(X^I,\mathrm{DR}^{-1}\Delta_*^{(I/T)}(\mathcal{A}\shift)^{\boxtimes T}(*\Delta_T)).
  $$
  Then we have $d_{\mathrm{ch}}\eta=\bar{\partial}\eta=0$ by type reasons. Since $\eta\in \mathrm{Norm}^0$, we have $d_{\mathrm{Norm}}\eta=0$. This implies that $\alpha$ is also $d_{\mathrm{tot}}-$exact and $\mathrm{tr}(\alpha)=0$.
\end{proof}

Following \cite[pp.307,Section 4.2.13]{book:635773}, we now introduce a subcomplex $\tilde{C}^{\mathrm{ch}}_{\log}(X,\omega_X)_{\mathcal{Q}}$ of $\tilde{C}^{\mathrm{ch}}(X,\omega_X)_{\mathcal{Q}}$ which consists of forms with logarithmic singularities.  Let $\mathrm{DR}^{\log}_{X^I}\subset \mathrm{DR}(\omega_X\shift^{\boxtimes I}(*\Delta
_I))$ be the DG subalgebra generated by $\mathcal{O}_{X^I}$ and 1-forms $\frac{df}{f}$, where $f=0$ is an equation of component of the diagonal divisor. Then define a subcomplex of $\mathrm{DR}(C(\mathcal{A})_{\mathcal{Q},X^I})$ as follows
$$
\mathrm{DR}C^{\log}_{X^I,\mathcal{Q}}:=\bigoplus_{T\in \equiQ(I)} \Delta^{(I/T)}_{\bullet}\mathrm{DR}^{\log}_{X^T,\mathcal{Q}}\subset \mathrm{DR}(C(\mathcal{A})_{\mathcal{Q},X^I}).
$$
The definition made in Section \ref{ChiralHomology} can be carried over in a
straightforward manner to this case by simply replacing $F_{X^I}=\mathrm{DR}(C(\mathcal{A})_{\mathcal{Q},X^I})$ with $F^{\log}_{X^I}=\mathrm{DR}C^{\log}_{X^I,\mathcal{Q}}$. The total differential is denoted by $d_{\mathrm{tot},\log}.$

We have quasi-isomorphisms (see \cite[pp.309,(4.2.14.3)]{book:635773})
$$
\tilde{C}^{\mathrm{ch}}_{\log}(X,\omega_X)_{\mathcal{Q}}\xrightarrow{\sim}\tilde{C}^{\mathrm{ch}}(X,\omega_X)_{\mathcal{Q}}\xrightarrow
{\mathrm{tr}} \mathbb{C}.
$$
We denote the first quasi-isomorphism by $\iota$ and the composition $\mathrm{tr}\circ\iota$ by $\mathrm{tr}_{\log}$.

\begin{prop}
  The trace map on $\tilde{C}^{\mathrm{ch}}_{\log}(X,\omega_X)_{\mathcal{Q}}$ is given by the usual integration map.
\end{prop}

\begin{proof}

We denote the integration map by
$$
\int :\tilde{C}^{\mathrm{ch}}_{\log}(X,\omega_X)_{\mathcal{Q}}\rightarrow \mathbb{C}.
$$
For $\alpha\in \tilde{C}^{\mathrm{ch}}_{\log}(X,\omega_X)_{\mathcal{Q}}$ which can be represented as a differential form on $X^I$, the integration map is defined to be
$$
\int\alpha:=\int_{X^I}\alpha.
$$
 From the fact that $\int d_{\mathrm{Norm}}(-)=0$ and Stokes formula $\int(d_{\mathrm{DR}}+\bar{\partial}+2\pi id_{\mathrm{ch}})(-)=0$, we conclude that $\int$ is a chain map. We need to prove that $\int=\mathrm{tr}_{\log}$.

Consider a degree 0 subspace
$$
 \Omega^{0,|I|}(X^I,(\omega_X\shift)^{\boxtimes I}(\log\Delta_I))\hookrightarrow \Gamma(X^I, F^{\log}_{X^I})\hookrightarrow \Gamma(X^I, \mathrm{Norm}(\equiQ(I),F^{\log})^{0})\rightarrow\tilde{C}^{\mathrm{ch}}_{\log}(X,\omega_X)_{\mathcal{Q}}.
$$
Let $\alpha\in \Omega^{0,|I|}(X^I,(\omega_X\shift)^{\boxtimes I}(\log\Delta_I))$. By abuse of notation, we continue to write $\alpha$ for its image in $\tilde{C}^{\mathrm{ch}}_{\log}(X,\omega_X)_{\mathcal{Q}}$. We will prove that
$$
\mathrm{tr}_{\log}(\alpha)=\int\alpha=\int_{X^I}\alpha.
$$

  We restrict the trace map to the subspace of regular forms $\Omega^{0,|I|}(X^I,(\omega_X\shift)^{\boxtimes I})$. Since it vanishes on $d_{\mathrm{DR}}$ and $\bar{\partial}$ exact forms, we can express it as integration up to a scalar factor. This means that for regular $\alpha$
  $$
  \mathrm{tr}_{\log}(\alpha)=c_I\cdot \int_{X^I}\alpha,
  $$
  here $c_I$ is a constant depending on $I$. We will prove that $c_I=1$ of all $I$. We know that $c_{\{\bullet\}}=1$ from the definition. Let $P^{1,0}\in \Omega^{1,0}(X^2,\mathcal{O}_{X^2}(\log \Delta))$ be the Szeg\"{o} kernel, see \cite{1992Kernel}. It has the following property
  $$
  \bar{\partial}_{1}P^{1,0}=\mathrm{vol}_{X_1}+\sum(\text{product of 1-forms on}\  X_1 \ \text{and} \ X_2).
  $$
  Here we write $X^2=X_1\times X_2$ and $  \bar{\partial}_{1}$ is the Dolbeault operator along the first factor.

 Then
  $$
  \mathrm{tr}_{\mathrm{log}}(\bar{\partial}P^{1,0}\shift\cdot \mathrm{vol}_{X_2}\shift)=\mathrm{tr}_{\mathrm{log}}(\mathrm{vol}_{X_1} \shift\cdot  \mathrm{vol}_{X_2}\shift)=c_{\{\bullet\bullet\}}\cdot \int_{X^2}\mathrm{vol}_{X_1} \shift\cdot \mathrm{vol}_{X_2}\shift
  $$
  $$
=c_{\{\bullet\bullet\}}\int_{X^2}\bar{\partial}P^{1,0}\shift\cdot \mathrm{vol}_{X_2}  \shift =-c_{\{\bullet\bullet\}}\cdot \int_X 2\pi i\cdot\mathrm{Res} (P^{1,0}\cdot \mathrm{vol}_{X_2}\shift).
  $$
  On the other hand, we have
  $$
  \mathrm{tr}_{\mathrm{log}}(\bar{\partial}P^{1,0}\shift\cdot \mathrm{vol}_{X_2}\shift)=\mathrm{tr}_{\mathrm{log}}(-2\pi i d_{\mathrm{ch}}(P^{1,0}\shift\cdot \mathrm{vol}_{X_2}\shift))=- \int_X 2\pi i\cdot\mathrm{Res} (P^{1,0}\cdot \mathrm{vol}_{X_2}\shift).
  $$
  This implies that $c_{\{\bullet\bullet\}}=1$. The same argument works for general $I$.

 Now for $\alpha\in  \Omega^{0,|I|}(X^I,(\omega_X\shift)^{\boxtimes I}(\log\Delta_I))$, consider the element
 $$\eta=\alpha-\mathrm{tr}_{\log}(\alpha)\cdot \mathrm{vol}_{X^I}[|I|].
 $$
  Then
  $$
  \mathrm{tr}_{\log}(\eta)=\mathrm{tr}_{\log}(\alpha-\mathrm{tr}_{\log}(\alpha)\cdot \mathrm{vol}_{X^I}[|I|])=0.
  $$
  Since $\mathrm{tr}_{\log}$ is a quasi-isomorphism, we can find $\eta'$ such that
  $$
  \eta=d_{\mathrm{tot},\log}\eta'.
  $$
  Then
  $$
  \int\alpha=\int \mathrm{tr}_{\log}(\alpha)\cdot \mathrm{vol}_{X^I}[|I|]+\int d_{\mathrm{tot},\log}\eta'=\mathrm{tr}_{\log}(\alpha).
  $$

\end{proof}

One can reformulate the definition of regularized integral as follows.

\begin{prop}\label{RegularizedIntDefn2}
   Let $\alpha\in \Omega^{0,|I|}(X^I,(\omega_X\shift)^{\boxtimes I}(*\Delta_I))$. We can find an element $\alpha_{\mathrm{log}}\in \tilde{C}^{\mathrm{ch}}_{\log}(X,\omega_X)_{\mathcal{Q}}$ such that $\iota(\alpha)$ is homologous to  $\alpha$ in $ \tilde{C}^{\mathrm{ch}}(X,\omega_X)_{\mathcal{Q}}$ . Then
   $$
   \dashint_{X^I}\alpha=\mathrm{tr}_{\log}(\alpha_{\log}).
   $$
\end{prop}
\begin{proof}
  This follows from the definition of regularized integrals. See Appendix \ref{ReInt}.
\end{proof}

Now we can prove the main theorem of this section, which establishes a simple relation between $\mathrm{tr}$ and regularized integrals.

\begin{thm}\label{RegularizedANDTrace}
  Let $\alpha\in \Omega^{0,|I|}(X^I,(\omega_X\shift)^{\boxtimes I}(*\Delta_I))$. By abuse of notation,  the composition of following maps
 $$
 \Omega^{0,|I|}(X^I,(\omega_X\shift)^{\boxtimes I}(*\Delta_I))\rightarrow \tilde{C}^{\mathrm{ch}}(X,\omega_X)_{\mathcal{Q}}\xrightarrow{\mathrm{tr}}
\mathbb{C}
$$
is still denoted by $\mathrm{tr}$.  Here the first map is given by
$$
 \Omega^{0,|I|}(X^I,(\omega_X\shift)^{\boxtimes I}(*\Delta_I))\hookrightarrow \Gamma(X^I, F_{X^I})\hookrightarrow \Gamma(X^I, \mathrm{Norm}(\equiQ(I),F)^{0})\rightarrow\tilde{C}^{\mathrm{ch}}(X,\omega_X)_{\mathcal{Q}}
$$
Then we have
$$
\mathrm{tr}(\alpha)=\dashint_{X^I}\alpha.
$$

\end{thm}
\begin{proof}
  We choose $\alpha_{\log}$ as in Proposition \ref{RegularizedIntDefn2}, then
  $$
  \mathrm{tr}(\alpha)=\mathrm{tr}(\iota(\alpha_{\log}))=\mathrm{tr}_{\log}(\alpha_{\log})=\dashint_{X^I}\alpha.
  $$
\end{proof}
\section{Quantum master equation and elliptic trace map}\label{BV}
\subsection{BV formalism and quantum master equation}

The BV formalism \cite{batalin1984gauge} is the most general method to quantize quantum field theory with gauge symmetries. Costello \cite{costello2011renormalization} gives one approach to the rigorous mathematical formulation of the perturbative BV formalism using a homotopic renormalization framework, which is an important motivating source of development in this work. See also \cite{costello2021factorization} for factorization algebras in the same framework. In this section, we summarize the mathematical structures in the BV quantization that we will use in this paper.

We start with the definition of the BV algebras. A Batalin-Vilkovisky (BV) algebra is a pair $(O_{\mathrm{BV}},\Delta_{\mathrm{BV}})$ where
\begin{itemize}
  \item $O_{\mathrm{BV}}$ is a $\mathbb{Z}$-graded commutative associative unital algebra over $\mathbb{C}$.
  \item $\Delta:O_{\mathrm{BV}}\rightarrow O_{\mathrm{BV}}$ is a linear operator of degree 1 such that $\Delta^2=0$.
  \item Define $\{-,-\}:O_{\mathrm{BV}}\otimes O_{\mathrm{BV}}\rightarrow O_{\mathrm{BV}}$ by
  $$
  \{a,b\}:=\Delta(ab)-(\Delta a)b-(-1)^{|a|}a\Delta b, \ a,b\in O_{\mathrm{BV}}.
  $$
  Then $\{-,-\}$ satisfies the following graded Leibnitz rule
  $$
  \{a,bc\}:=\{a,b\}c+(-1)^{(|a|+1)|b|}b\{a,c\},\ \ a,b,c\in O_{\mathrm{BV}}.
  $$
\end{itemize}
Let $(O_{\mathrm{BV}},\Delta)$ be a BV algebra. Let $I=I_0+I_1\hbar+\cdots\in O_{\mathrm{BV}}[[\hbar]]$ be a degree 0 element, that is, $I_i$ is degree 0 for each $i\geq 0$.
\begin{defn}
Assume that $I_0$ and $I_1$ are nilpotent. The element $I$ is said to satisfy quantum master equation (QME) if
$$
{{\hbar\Delta e^{I/\hbar}=0.}}
$$
\end{defn}

This is equivalent to
\begin{equation}\label{QME2}
\hbar\Delta I+\frac{1}{2}\{I,I\}=0.
\end{equation}

\begin{rem}
 In the Equation (\ref{QME2}),  we do not need the conditions on $I_0$ and $I_1$, but we assume these in order to get the well defined exponential series $e^{I/\hbar}$ for our purpose.
\end{rem}

In general, if $(C_{\bullet},d_{C})$ is a $\mathbb{C}((\hbar))$-chain complex and $(B_{\bullet
},d_{B})$ is a $\mathbb{C}-$chain complex, we introduce the following generalized notion of quantum master equation.
\begin{defn}\label{GenQME}
 W say a $\mathbf{k}$-linear map ($\mathbf{k}=\mathbb{C}((\hbar))$)
$$
\langle-\rangle: C_{\bullet}\rightarrow B_{\bullet}\otimes_{\mathbb{C}}O_{\mathrm{BV}}((\hbar))
$$
satisfies QME if
$$
{(d_{C}-d_{B}+\hbar\Delta)\langle-\rangle=0}.
$$
\end{defn}

\begin{rem}\label{GenQMERem}
If we take $(C_{\bullet},d)=(\mathbf{k},0)$, the map $I(-)$
$$
I(-):\mathbf{k}\rightarrow O_{\mathrm{BV}}((\hbar)), \ I(c)=ce^{I/\hbar}
$$
satisfies QME if and only if $I\in\mathbb{C}[[\hbar]]$ itself satisfies QME.
\end{rem}
We recall the notion of BV integration.
\begin{defn}
  A BV integration is a linear map
  $$
  \int_{\mathrm{BV}}:O_{\mathrm{BV}}\rightarrow \mathbb{C}
  $$
  such that $\int_{\mathrm{BV}}\Delta \Phi=0$.
\end{defn}
A typical example of BV integration is the top fermion integration of the BV algebra $(\mathbb{C}[[x_i,\theta_i]]_{i=1,\dots,n}, \Delta_{\mathrm{BV}}=\sum_{i=1}^n\frac{\partial}{\partial x_i}\frac{\partial}{\partial\theta_i})$. Here $|x_i|=0, |\theta_i|=-1$ for $i=1,\dots,n$. The top fermion integration is defined by
$$
\int_{\mathrm{BV},\mathrm{top}}f:=(\frac{\partial}{\partial\theta_1}\cdots \frac{\partial}{\partial\theta_n}f)|_{x_1=\cdots=x_n=0},\quad f\in \mathbb{C}[[x_i,\theta_i]]_{i=1,\dots,n}.
$$

\begin{rem}
In \cite{BVQandindex},  the algebraic index is computed using this BV integral. In general,  a choice of super Lagrangian submanifold in the BV manifold gives rise to a BV integration. The top fermion is given by the odd Lagrangian $\mathbb{C}^n\shift\subset \mathbb{C}^n\oplus\mathbb{C}^n\shift$.
\end{rem}
Roughly speaking, the effective BV quantization theory $T$ on a smooth manifold $X$ contains
\begin{itemize}
 \item Local observable algebra (Factorization algebra) : ${\mathrm{Obs}_T}$, a $\mathbf{k}$-module equipped with a certain algebraic structure.
  \item Factorization homology (complex): ${(C_{\bullet}(\mathrm{Obs}_T),d_T)}$, a $\mathbf{k}$-chain complex.
  \item A BV algebra ${(O_{\mathrm{BV},T},\Delta)}$. This BV algebra corresponds to the algebra of the zero modes in effective BV quantization.
  \item A BV integration map
  $$
  \int_{\mathrm{BV}}:O_{\mathrm{BV},T}\rightarrow \mathbb{C}
  $$
  such that $\int_{\mathrm{BV}}\Delta \Phi=0$.
  \item A $\mathbf{k}$-linear map
  $$
 \mathrm{Tr}_T: C_{\bullet}(\mathrm{Obs}_T)\rightarrow (\BVk, -\hbar\Delta_{\mathrm{BV}}),
  $$
which satisfies QME, that is, we have
  $$
{(d_T+\hbar\Delta)\mathrm{Tr}_T(-)=0.}
  $$
\end{itemize}

Here we give an informal explanation of the above structures. Intuitively, the complex ${(C_{\bullet}(\mathrm{Obs}_T),d_T)}$ captures the information of the operator product expansions.  It is convenient to formulate this in the language of right $\mathcal{D}$-modules, as we see in the construction of the chiral chain complex. Naively, the last step corresponds to the Feynman integrals in physics. Since these integrals often diverge, we need to do the renormalization. Costello \cite{costello2011renormalization} gives a general mathematical framework of renormalization and effective theories using heat kernel cut-off and counterterms. In \cite{li2012feynman,Li:2016gcb}, the renormalization of a large class of 2d chiral quantum field theories is accomplished in Costello's framework.



\begin{rem}
The trace map  $\mathrm{tr}: \tilde{C}^{\mathrm{ch}}(X,\omega_X)_{\mathcal{Q}}\rightarrow
\mathbb{C}$ constructed by Beilinson and Drinfeld can be viewed as a homological renormalization procedure since we can use this trace map to integrate singular forms in the chiral chain complex. In \cite{li2020regularized}, the authors develop analytic tools for integrals arising from two-dimensional chiral quantum field theories. In particular, the regularized integral  is introduced for integrating differential forms on products of Riemann surfaces with arbitrary holomorphic poles
along the diagonals. From Theorem \ref{RegularizedANDTrace}, these two approaches are essentially the same.

\end{rem}
\subsection{BV formalism in 2d chiral quantum field theory}\label{BVchiral}
In this section, we review the 2d chiral quantum field theory following \cite{Li:2016gcb}. In \cite{Li:2016gcb}, the second named author of the present paper establishes an exact correspondence between renormalized quantum master equations for chiral deformations and Maurer-Cartan equations for modes Lie algebra of the free $\beta\gamma-bc$ vertex algebra. The relation between this result and our main theorem is discussed in Section \ref{TraceLieCohomology}.

 From now on, we assume that $X=\mathbb{C}/\mathbb{Z}+\tau\mathbb{Z}$ is an elliptic curve. Similar methods could be generalized to higher genus.  Here we focus on the explicit construction on an elliptic curve and leave higher genus situations for future study.

Let $\mathbf{L}=\oplus_{\alpha \in \mathbb{Q}} \mathbf{L}^{\alpha}$, where $\mathbf{L}^{\alpha}=\oplus_{i\in\mathbb{Z}}\mathbf{L}_i^{\alpha}$ is a $\mathbb{Z}$-graded vector space for each $\alpha$. Here
the $\mathbb{Q}$ -grading is the conformal weight and we assume for simplicity only finitely many weights appear in $\mathrm{L}$. The data of conformal weight will not be used for our main theorem in Section \ref{MainTheoremBV}.

Suppose that $\mathrm{\mathbf{L}}$ is equipped with a degree 0 symplectic pairing
$$
\langle-,-\rangle: \wedge^{2} \mathbf{L} \rightarrow \mathbb{C}
$$
which is of conformal weight $-1,$ that is, the only nontrivial pairing is
$$
\langle-,-\rangle: \mathbf{L}^{\alpha} \otimes \mathbf{L}^{1-\alpha} \rightarrow \mathbb{C}
$$
For each $a \in \mathbf{L}^{\alpha}$, we associate a field

\begin{equation}\label{ConformalWeight}
a(z)=\sum_{r \in \mathbb{Z}-\alpha} a_{r} z^{-r-\alpha}.
\end{equation}
We define their singular part of OPE by
$$
a(z) b(w) \sim\left(\frac{ i\hbar}{\pi}\right) \frac{\langle a, b\rangle}{(z-w)}, \quad \forall a, b \in \mathbf{L},
$$

\begin{rem}
  In (\ref{ConformalWeight}), we use the conformal weight notation. If we use the VOA notation, we have
  $$
  a(z)=\sum_{n \in \mathbb{Z}} a_{(n)} z^{-n-1},
  $$
  that is, $a_{(n)}=a_{n+1-\alpha}.$
\end{rem}

The vertex algebra ${V}^{\beta\gamma-bc}$ that realizes the above OPE relations is given by the Fock representation space over the base ring $\mathbf{k}=\mathbb{C}((\hbar))$. The vacuum vector satisfies
$$
a_{r}|0\rangle=0, \quad \forall a \in \mathbf{L}^{\alpha}, r+\alpha>0,
$$
and ${V}^{\beta\gamma-bc}$ is freely generated from the vacuum by the operators $\left\{a_{r}\right\}_{r+\alpha \leq 0}, a \in \dim\mathbf{L}^{\alpha} .$ For any $a \in \mathbf{L}$, $a(z)$ becomes a field acting naturally on the Fock space ${V}^{\beta\gamma-bc}$ . Note that ${V}^{\beta\gamma-bc}$ is isomorphic to $\mathbb{C}[[\partial^ka^s]]((\hbar))_{s=1,\dots,\mathrm{\mathbf{L}},k\geq 0}$ as a vector space if we choose a basis $\{a^s\}_{s=1,\dots,\dim \mathrm{\mathbf{L}}}$ of $\mathrm{\mathbf{L}}$. The identification is the following map
$$
\mathbb{C}[[\partial^ka^s]]((\hbar))_{s=1,\dots,\dim \mathrm{\mathbf{L}},k\geq 0}\cong{V}^{\beta\gamma-bc}
$$
$$
P(\partial^ka^s)\mapsto P(a^s_{-\alpha-k})|0\rangle.
$$

 Since ${V}^{\beta\gamma-bc}$ is conformal, we have a graded vertex algebra bundle $\mathcal{V}^{\beta\gamma-bc}\cong X\times {V}^{\beta\gamma-bc}$. We denote the corresponding graded chiral algebra by $$\mathcal{A}=\mathcal{A}^{\beta\gamma-bc}=\mathcal{V}^{\beta\gamma-bc}\otimes_{\mathcal{O}_X} \omega_X.$$

\begin{rem}
 Let $\mathcal{L}^{\vee}=X\times \mathbf{L}^{\vee}$ be the trivial holomorphic super vector bundle. The degree 0 pairing on $\mathbf{L}$  induces a (-1)-symplectic pairing on the space of field $\mathcal{E}=\Gamma(X,\mathcal{L}^{\vee}_\mathcal{Q})$
$$
(\varphi_1.\varphi_2):=\int_X dz\wedge \langle\varphi_1,\varphi_2\rangle, \ \varphi_1,\varphi_2\in \mathcal{E}=\Gamma(X,\mathcal{L}^{\vee}_\mathcal{Q})=\Omega^{0,\bullet}(X,\mathcal{L}^{\vee}).
$$
One can then use the BV formalism to construct the perturbative theory of this 2d chiral quantum field theory, here we just summarize the mathematical structures in the construction. For more details, see \cite{costello2021factorization,gorbounov2016chiral,Li:2016gcb}.
\end{rem}
\begin{rem}\label{twistedEnvelope}
  Our notion of free $\beta\gamma-bc$ system is essentially a special example of chiral Weyl algebras introduced in \cite[Section 3.8.1]{book:635773}. In fact, one set (assume that $\omega^{\alpha}_X$ formally exists)
  $$
  L=(\bigoplus_{\alpha\in \mathbb{Q}}\mathrm{\mathbf{L}}^{\alpha}\otimes_{\mathbb{C}}\omega_X^{-\alpha+1})\otimes_{\mathcal{O}_X}\mathcal{D}_X,
$$
  $$
  L^{\flat}=\omega_{X}\oplus (\bigoplus_{\alpha\in \mathbb{Q}}\mathrm{\mathbf{L}}^{\alpha}\otimes_{\mathbb{C}}\omega_X^{-\alpha+1})\otimes_{\mathcal{O}_X}\mathcal{D}_X.
  $$
  Then
$$
 \big((\mathrm{\mathbf{L}}^{\alpha}\otimes_{\mathbb{C}}\omega_X^{-\alpha+1})\otimes_{\mathcal{O}_X}\mathcal{D}_X\big) \boxtimes\big((\mathrm{\mathbf{L}}^{1-\alpha}\otimes_{\mathbb{C}}\omega_X^{\alpha})\otimes_{\mathcal{O}_X}\mathcal{D}_X\big) \xrightarrow{\langle-,-\rangle} \omega_X\otimes_{\mathcal{O}_X}\mathcal{D}_{X^2}
 $$
 $$
 \rightarrow \omega_X\otimes_{\mathcal{D}_X}\mathcal{D}_{X\rightarrow X^2}=\Delta_*\omega_X
$$
makes $L^{\flat}$ a Lie* algebra and one obtains a chiral algebra $U(L)^{\flat}$ called the twisted chiral enveloping algebra of $L^{\flat}$ (which is the chiral envelope $U(L^{\flat})$ modulo the ideal generated by $1-1^{\flat}$, here $1=\omega_X$ is the unit of the envelope and $1^{\flat}= \omega_X\subset L^{\flat}$ is the direct summand in $L^{\flat}$). Our free $\beta\gamma-bc$ system can be obtained in the same way by introducing a formal variable $\hbar$ and rescale the Lie* pairing to $\frac{i\hbar}{\pi}\langle-,-\rangle$. We will denote it by $U(L)^{\flat}_{\hbar}$. The corresponding vertex algebra bundle is generated by elements in $\mathrm{\mathbf{L}}^{\alpha}$ which transform as $\omega_X^{-\alpha}$.
   Since we focus on the elliptic curve which is flat, the conformal weight is not manifest in our later construction.
\end{rem}
Let $\{a^s\}_{s=1,\dots,\dim \mathbf{L}}$ (resp. $\{\asc\}_{s=1,\dots,\dim \mathbf{L}}$) be a basis of $\mathbf{L}$ (resp.  $\mathbf{L}\shift$). We introduce a commutative ring generate by  $\mathbf{L}\oplus \mathbf{L}\shift$ and their formal derivatives
$$
\OL:=\mathbb{C}[[\partial^ka^s,\partial^ka^{s\dagger}]]_{s=1,\dots,\dim \mathrm{\mathbf{L}}, k\geq 0},
$$
with
$$
\deg(\partial^k\as)=|a^s|, \  \deg(\partial^k\asc)=|a^s|-1, \quad k\geq 0.
$$
Here $\OL$ is considered as a commutative super vertex algebra and the translation operator $L_{-1}$ is a derivation on $\OL$ such that
$$
L_{-1}\partial^ka^s=\partial^{k+1}a^s,\quad L_{-1}\partial^ka^{\dagger}=\partial^{k+1}a^{s\dagger}, \quad s=1,\dots,\dim \mathrm{\mathbf{L}}, k\geq 0.
$$
Let $(L_{-1} \OL)$ be the ideal generated by $L_{-1}\OL$ in $\OL$. We define a BV algebra
$$
\BVL:=\mathbb{C}[[\as,\asc]]=\OL/(L_{-1}\OL),
$$
with
$$
\deg(\as)=|a^s|, \  \deg(\asc)=|a^s|-1.
$$

For simplicity of notation, we sometimes write $O, O_{\mathrm{BV}}$ instead. We also introduce $O_{\mathbf{k}}:=O\otimes_{\mathbb{C}}\mathbf{k}$, then $\BVk:=O_{\mathrm{BV}}\otimes_{\mathbb{C}}\mathbf{k}=O_{\mathbf{k}}/(L_{-1}O_{\mathbf{k}}).$ The quotient map $O_{\mathbf{k}}\rightarrow \BVk$ is denoted by $\mathbf{p}_{\mathbf{BV}}$.

 The BV differential on $O_{\mathrm{BV}}$ is
$$
\Delta_{\mathrm{BV}}=\frac{-i}{{\mathrm{Im}(\tau)}}\sum_{p,q}\omega_{pq}\partial_{\ap}\partial_{\aqc},
$$
where $\omega_{pq}$ is the even symplectic form $\langle-,-\rangle$ w.r.t. the basis $\{a^s\}_{s=1,\dots,\dim \mathrm{\mathbf{L}}}$.

Now we collect the data in effective BV formalism  of 2d chiral quantum field theory:
\begin{itemize}
 \item The local observable algebra (factorization algebra) ${\mathrm{Obs}_{2d-chiral}}$ is $\mathcal{V}^{\beta\gamma-bc}$.

   \item The factorization complex ${(C_{\bullet}(\mathrm{Obs}_{2d-chiral}),d)}$ is the chiral chain complex
       $$
{(C_{\bullet}(\mathrm{Obs}_{2d-chiral}),d)}:=       (\tilde{C}^{\mathrm{ch}}(X,\mathcal{A}^{\beta\gamma-bc}),d_{\mathrm{tot}}).
       $$
  \item The BV algebra is ${(O_{\mathrm{BV}},\Delta_{\mathrm{BV}})}$.
  \item A BV integration map
  $$
  \int_{\mathrm{BV}}:O_{\mathrm{BV}}\rightarrow \mathbb{C}.
  $$
  \item A $\mathbf{k}-$linear map
  $$
\mathrm{Tr}_{2d-chiral}: C_{\bullet}(\mathrm{Obs}_{2d-chiral})\rightarrow \BVk
  $$
satisfying QME  , which means that it is a chain map
  $$
{(d_{\mathrm{tot}}+\hbar\Delta)\mathrm{Tr}_{2d-chiral}=0.}
  $$
\end{itemize}
 We will construct the chain map {$\mathrm{Tr}_{2d-chiral}$ }using Feynman integrals. It will be the composition of the following sequence of maps
 $$
\mathrm{Tr}_{2d-chiral}: \tilde{C}^{\mathrm{ch}}(X,\mathcal{A}^{\beta\gamma-bc})_{\mathcal{Q}}\xrightarrow{\mathcal{W}}\tilde{C}^{\mathrm{ch}}(X,\omega_X)_{\mathcal{Q}}\otimes_{\mathbb{C}}\BVk\xrightarrow{\mathrm{tr}} \BVk.
 $$

In the rest of this section, we describe the $\mathcal{W}$ explicitly. This is a Feynman diagram construction and we first introduce the propagator. The BV propagator $P^{\mathrm{BV}}(z,w)\in \Gamma(X\times X, (\mathbf{L}^{\vee}\boxtimes\mathbf{L}^{\vee}(*\Delta))_\mathcal{Q})$ is a bisection defined by
$$
P^{\mathrm{BV}}(z,w)=P(z,w)\cdot \sum_{p,q}\omega_{pq}(\frac{\partial}{\partial \ap}\otimes \frac{\partial}{ \partial \aq})
$$
Its analytic part $P(z,w)$ is the Szeg\"{o} kernel.

Using the theta function $\vartheta_1(z;\tau)=-i\mathop{\sum}^{\infty}\limits_{n=-\infty}(-1)^nq^{(n+\frac{1}{2})^2}e^{(2n+1)\pi i z}$, we can write
$$
P(z,w)=\frac{i}{\pi}\partial_z(\log \vartheta_1(z-w;\tau))-\frac{2\mathrm{Im} (z-w)}{\mathrm{Im}(\tau)}.
$$
It is easy to see that
\begin{equation}\label{Zeromodes}
 P(z,w)= \frac{i}{\pi (z-w)}+P^{\mathrm{reg}}, \quad \bar{\partial}_zP(z,w)=-\frac{id\bar{z}}{\mathrm{Im}(\tau)},
\end{equation}
here $P^{\mathrm{reg}}$ is regular along the diagonal.
\begin{rem}
  Here $P(z,w)$ is the propagator for the free 2d chiral theory $\int_{X}\beta\bar{\partial}\gamma$. In other words, $P(z,w)$ is the Green's function $\bar{\partial}^{-1}$.
\end{rem}
We introduce some notations, let
$$
P_{ij}(k,l)=\partial^k_{z_i}{\partial^l}_{z_j}P(z_i,z_j), \quad Q(k,l)=\lim_{z_i\rightarrow z_j}\partial^k_{z_i}{\partial^l}_{z_j}(P(z_i,z_j)-\frac{i}{\pi (z_i-z_j)}).
$$

We now construct the map
$$
\mathcal{W}:\Gamma(X^n,(\mathcal{A}\shift)^{\boxtimes n}(*\Delta_{\{1,\dots,n\}})_\mathcal{Q}) \rightarrow\Gamma(X^n,(\omega_X\shift)^{\boxtimes n}(*\Delta_{\{1,\dots,n\}})_\mathcal{Q})\otimes_{\mathbb{C}}\BVk
$$
\begin{equation}\label{DefinitionOfW}
\mathcal{W}( v_1 \otimes\cdots\otimes v_n\cdot f(z_1,\dots,z_n))= \mathbf{p}_{\mathrm{BV}}\circ \mathbf{Mult}\big( e^{\hbar \mathfrak{P}+\hbar\mathfrak{Q}+D}( v_1\otimes\cdots\otimes v_n)\big)\cdot f(z_1,\dots,z_n).
\end{equation}

We explain the notations in the above expression. Here
$$
 v\cdot f= v_1\otimes\cdots\otimes v_n\cdot f(z_1,\dots,z_n)\in \Gamma\big(X^{n},(\mathcal{A}\shift)^{\boxtimes n}(*\Delta_{\{1,\dots,n\}})_\mathcal{Q}\big)
$$
where $f\in \Gamma\big(X^{n},\omega_{X^n}(*\Delta_{\{1,\dots,n\}})_{\mathcal{Q}}\big)$ and $v_1\otimes\cdots\otimes v_n\in V^{\otimes n}$. And we use the identification
$$
\Gamma\big(X^{n},(\mathcal{A}\shift)^{\boxtimes n}(*\Delta_{\{1,\dots,n\}})_\mathcal{Q}\big)\simeq V^{\otimes n}\otimes_{\mathbb{C}}\Gamma\big(X^{n},\omega_{X^n}(*\Delta_{\{1,\dots,n\}})_{\mathcal{Q}}\big)\text{\small{[n]}},
$$
and the identification
$$
V\simeq \mathbb{C}[[\partial^ka^s]]((\hbar))_{s=1,\dots,\dim \mathrm{\mathbf{L}}, k\geq 0}.
$$
The operator $D$ is defined as follows
$$
D(v_1\otimes\cdots\otimes v_n):=\sum_iv_1\otimes\cdots\otimes D_i(v_i)\otimes \cdots\otimes v_n,
$$
where $D_i$ is a derivation from $\mathbb{C}[[\partial^ka^s]]((\hbar))_{s=1,\dots,\dim \mathrm{\mathbf{L}}, k\geq 0}$ to $O_{\mathbf{k}}d\bar{z}_i$ defined by
$$
D_i(\partial^k\as)=(-1)^{p(\as)} \partial^k\asc \cdot d\bar{z}_i.
$$

And bi-differential operators $\mathfrak{P}$ and $\mathfrak{Q}$ are defined as follows
$$
\mathfrak{P}:=\sum_{i<j}\mathcal{P}_{ij}, \ \mathfrak{Q}:=\sum_{i}\mathcal{Q}_i
$$
$$
\mathcal{P}_{ij}:=\sum_{k,l}\sum_{p,q}P_{ij}(k,l)\omega_{pq}(\frac{\partial}{\partial (\partial^k \ap)})_i\otimes (\frac{\partial}{\partial (\partial^l \aq)})_j,\quad i<j,
$$
here
$$
(\frac{\partial}{\partial (\partial^k \ap)})_i\otimes (\frac{\partial}{\partial (\partial^l \aq)})_j(v_1\otimes\cdots\otimes v_n)=(-1)^{p(a^p)\cdot (p(v_1)+\cdots+p(v_{i-1}))+p(a^q)\cdot (p(v_1)+\cdots+p(v_{j-1}))}
$$
$$
\cdot v_1\otimes\cdots\otimes \frac{\partial}{\partial (\partial^k \ap)}(v_i) \otimes\cdots\otimes \frac{\partial}{\partial (\partial^l \aq)}(v_j)\otimes  \cdots\otimes v_n,
$$
and
$$
\mathcal{Q}_i=\frac{1}{2}\sum_{k,l}\sum_{p,q}Q(k,l)\omega_{pq}(\frac{\partial}{\partial (\partial^k \ap)}\cdot\frac{\partial}{\partial (\partial^l \aq)})_i,
$$
where
$$
\omega_{pq} (\frac{\partial}{\partial (\partial^k \ap)}\cdot\frac{\partial}{\partial (\partial^l \aq)})_i(v_1\otimes\cdots\otimes v_n)=\omega_{pq}v_1\otimes\cdots\otimes\frac{\partial}{\partial (\partial^k \ap)}\cdot\frac{\partial}{\partial (\partial^l \aq)}v_i\otimes \cdots\otimes v_n .
$$
See figure(\ref{WFeynman}). To summarize, we have a sequence of map
$$
V^{\otimes n}\otimes_{\mathbb{C}}\Gamma\big(X^{n},\omega_{X^n}(*\Delta_{\{1,\dots,n\}})_{\mathcal{Q}}\big)\text{\small{[n]}}\xrightarrow{e^{\hbar \mathfrak{P}+\hbar\mathfrak{Q}+D}}{O_{\mathbf{k}}}^{\otimes n}\otimes_{\mathbb{C}}\Gamma\big(X^{n},\omega_{X^n}(*\Delta_{\{1,\dots,n\}})_{\mathcal{Q}}\big)\text{\small{[n]}}
$$
$$
\xrightarrow{\mathbf{p}_{\mathrm{BV}}\circ \mathbf{Mult}}{\BVk}\otimes_{\mathbb{C}}\Gamma\big(X^{n},\omega_{X^n}(*\Delta_{\{1,\dots,n\}})_{\mathcal{Q}}\big)\text{\small{[n]}},
$$
here $\mathbf{Mult}:O_{\mathbf{k}}^{\otimes n}\rightarrow O_{\mathbf{k}}$ is the multiplication map and recall that $\mathbf{p}_{\mathrm{BV}}:O_{\mathbf{k}}\rightarrow\BVk=O_{\mathbf{k}}/(L_{-1}O_{\mathbf{k}})$ is the quotient map. The exponential $e^{\hbar \mathfrak{P}+\hbar\mathfrak{Q}+D}$ is well defined since these operators are commuting with each other.
\begin{rem}
  In the definition of operators $D$ and $\mathcal{Q}_i$, the Koszul signs are trivial since $D_i=\sum_p\sum_{k\geq 0}(\partial^k\apc\frac{\partial}{\partial(\partial^k\ap)})_id\bar{z}_i$  and $\omega_{pq}(\frac{\partial}{\partial (\partial^k \ap)}\cdot\frac{\partial}{\partial (\partial^l \aq)})_i$ are even operators. The parity of $(\partial^k\apc\frac{\partial}{\partial(\partial^k\ap)})_id\bar{z}_i$ is $p(\apc)+p(\ap)+p(d\bar{z}_i)=\bar{0}$ and the parity of $\omega_{pq}(\frac{\partial}{\partial (\partial^k \ap)}\cdot\frac{\partial}{\partial (\partial^l \aq)})_i$ is $p(\ap)+p(\aq)=\bar{0}$ (recall that the bilinear pairing $\omega$ is even).
\end{rem}
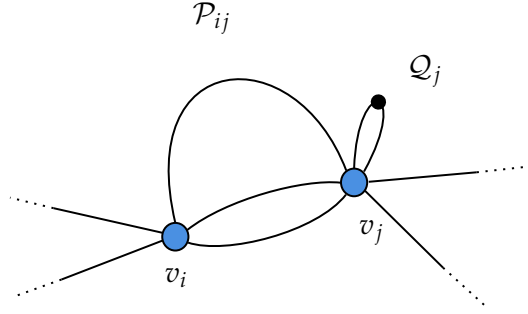
\begin{figure}
  \centering

\tikzset{every picture/.style={line width=0.75pt}} 

\begin{tikzpicture}[x=0.75pt,y=0.75pt,yscale=-1,xscale=1]

\draw  [fill={rgb, 255:red, 74; green, 144; blue, 226 }  ,fill opacity=1 ] (232.2,166.33) .. controls (232.2,162.47) and (235.16,159.35) .. (238.8,159.35) .. controls (242.45,159.35) and (245.41,162.47) .. (245.41,166.33) .. controls (245.41,170.18) and (242.45,173.31) .. (238.8,173.31) .. controls (235.16,173.31) and (232.2,170.18) .. (232.2,166.33) -- cycle ;
\draw  [fill={rgb, 255:red, 74; green, 144; blue, 226 }  ,fill opacity=1 ] (322.36,138.93) .. controls (322.36,135.08) and (325.32,131.95) .. (328.96,131.95) .. controls (332.61,131.95) and (335.57,135.08) .. (335.57,138.93) .. controls (335.57,142.79) and (332.61,145.91) .. (328.96,145.91) .. controls (325.32,145.91) and (322.36,142.79) .. (322.36,138.93) -- cycle ;
\draw [color={rgb, 255:red, 0; green, 0; blue, 0 }  ,draw opacity=1 ][line width=0.75]    (238.8,159.35) .. controls (218,70.9) and (299,64.9) .. (325,132.9) ;
\draw [color={rgb, 255:red, 0; green, 0; blue, 0 }  ,draw opacity=1 ]   (232.2,164.33) -- (177.4,151.93) ;
\draw [color={rgb, 255:red, 0; green, 0; blue, 0 }  ,draw opacity=1 ]   (232.2,168.33) -- (181.4,187.93) ;
\draw [color={rgb, 255:red, 0; green, 0; blue, 0 }  ,draw opacity=1 ]   (334.39,143.06) -- (373.8,181.93) ;
\draw [color={rgb, 255:red, 0; green, 0; blue, 0 }  ,draw opacity=1 ]   (244.81,162.93) .. controls (264.4,147.47) and (303.4,137.07) .. (322.2,139.07) ;
\draw  [dash pattern={on 0.84pt off 2.51pt}]  (177.4,151.93) -- (155.4,146.6) ;
\draw  [dash pattern={on 0.84pt off 2.51pt}]  (181.4,187.93) -- (158.65,196.11) ;
\draw  [dash pattern={on 0.84pt off 2.51pt}]  (373.8,181.93) -- (396.19,202.23) ;
\draw [color={rgb, 255:red, 0; green, 0; blue, 0 }  ,draw opacity=1 ]   (245.22,169.3) .. controls (264.31,175.68) and (310.22,163.97) .. (324.87,145.07) ;
\draw [color={rgb, 255:red, 0; green, 0; blue, 0 }  ,draw opacity=1 ]   (328.96,131.95) .. controls (330.67,83.72) and (357.96,89.95) .. (333.96,133.45) ;
\draw [shift={(341.22,98.33)}, rotate = 329.55] [color={rgb, 255:red, 0; green, 0; blue, 0 }  ,draw opacity=1 ][fill={rgb, 255:red, 0; green, 0; blue, 0 }  ,fill opacity=1 ][line width=0.75]      (0, 0) circle [x radius= 3.35, y radius= 3.35]   ;
\draw [color={rgb, 255:red, 0; green, 0; blue, 0 }  ,draw opacity=1 ]   (336.36,138.61) -- (390.78,133.81) ;
\draw  [dash pattern={on 0.84pt off 2.51pt}]  (390.78,133.81) -- (420.23,130.56) ;

\draw (246,46.4) node [anchor=north west][inner sep=0.75pt]    {$\mathcal{P}_{ij}$};
\draw (354,73.4) node [anchor=north west][inner sep=0.75pt]    {$\mathcal{Q}_{j}$};
\draw (232,180.4) node [anchor=north west][inner sep=0.75pt]    {$v_{i}$};
\draw (330,156.4) node [anchor=north west][inner sep=0.75pt]    {$v_{j}$};

\end{tikzpicture}
  \caption{The Feynman graph.}\label{WFeynman}
\end{figure}
\begin{rem}
This construction is a mathematical reformulation of the point-splitting method used by physicists.  Similar Feynman diagram constructions related to vertex operator algebra have been discussed in \cite{Tuite12,Mason03,Takhtajan05}.

\end{rem}

From the following lemma, we know that $\mathcal{W}$ is compatible with the $\mathcal{D}$-module structure and we can extend it to a map
$$
\mathcal{W}: \Gamma\big(X^I,\Delta_{\bullet}^{(I/I_0)}\mathrm{DR}\big(\mathop{\bigoplus}\limits_{T\in \equiQ(I_0)}\Delta_*^{(I_0/T)}((\mathcal{A}\shift)^{\boxtimes T}(*\Delta_T)_\mathcal{Q}\big)\big)
$$
$$\rightarrow {\BVk}\otimes_{\mathbb{C}}\Gamma\big(X^I,\Delta_{\bullet}^{(I/I_0)}\mathrm{DR}\big(\mathop{\bigoplus}\limits_{T\in \equiQ(I_0)}\Delta_*^{(I_0/T)}((\omega_X\shift)^{\boxtimes T}(*\Delta_T)_\mathcal{Q}\big)\big).
$$
\begin{lem}\label{WDmodule}
We have
  $$
  [(L_{-1})_i+\partial_{z_i},e^{\hbar \mathfrak{P}+\hbar \mathfrak{Q}+D}]=0, \ 1\leq i\leq n,
  $$
  here
  $$
  (L_{-1})_i(v_1\otimes \cdots\otimes v_n)=v_1\otimes\cdots \otimes L_{-1}v_i\otimes \cdots\otimes v_n.
  $$
\end{lem}
\begin{proof}
  We have
  \begin{align*}
    [(L_{-1})_i,\hbar \mathcal{P}_{ij}]  & =[\sum_{k\geq 0}\sum_r\partial^{k+1} a^r(\frac{\partial}{\partial (\partial^k a^r)})_i,\sum_{k,l\geq 0}\sum_{p,q}P_{ij}(k,l)\omega_{pq}(\frac{\partial}{\partial (\partial^k a^p)})_i\otimes (\frac{\partial}{\partial (\partial^l a^q)})_j]\\
     &=-\sum_{k,l\geq 0}\sum_{p,q}P_{ij}(k+1,l)\omega_{pq}(\frac{\partial}{\partial (\partial^k a^p)})_i\otimes (\frac{\partial}{\partial (\partial^l a^q)})_j\\
     &=-[\partial_{z_i}, \hbar \mathcal{P}_{ij}].
  \end{align*}
  Thus
  $$
  [(L_{-1})_i+\partial_{z_i},e^{\hbar \mathcal{P}_{ij}}]=0.
  $$
We have
\begin{align*}
  [(L_{-1})_i+\partial_{z_i}, Q_i] &=[(L_{-1})_i,Q_i]  \\
   & =[\sum_{k\geq 0}\sum_r(\partial^{k+1} a^r\frac{\partial}{\partial (\partial^k a^r)})_i,\frac{1}{2}\sum_{k,l\geq 0}\sum_{p,q}Q(k,l)\omega_{pq}(\frac{\partial}{\partial (\partial^k a^p)}\cdot\frac{\partial}{\partial (\partial^l a^q)})_i]\\
   &=-\frac{1}{2}\sum_{k,l\geq 0}\sum_{p,q}(Q(k+1,l)+Q(k,l+1))\omega_{pq}(\frac{\partial}{\partial (\partial^k a^p)}\cdot\frac{\partial}{\partial (\partial^l a^q)})_i\\
   &=0,
\end{align*}
  here we use $Q(k+1,l)+Q(k,l+1)=0$.
  And we have
  \begin{align*}
[(L_{-1})_i+\partial_{z_i}, D]&=[(L_{-1})_i, D_i]\\
& =   [\sum_{k\geq 0}\sum_r(\partial^{k+1} a^r\frac{\partial}{\partial (\partial^k a^r)})_i+\sum_{k\geq 0}\sum_r(\partial^{k+1} a^{r\dagger}\frac{\partial}{\partial (\partial^k a^{r\dagger})})_i, \sum_{l\geq 0}(\partial^{l} \apc\frac{\partial}{\partial (\partial^l a^p)})_id\bar{z}_i]  \\
     &= -\sum_{k\geq 0}(\partial^{k+1} \apc\frac{\partial}{\partial (\partial^k a^p)})_id\bar{z}_i+\sum_{k\geq 0}(\partial^{k+1} \apc\frac{\partial}{\partial (\partial^k a^p)})_id\bar{z}_i\\
     &=0.
  \end{align*}
\end{proof}

Since $\mathcal{W}$ is linear, it commutes with $d_{\mathrm{Norm}}$. In summary, we have a chain map
$$
\mathcal{W}:(\tilde{C}^{\mathrm{ch}}(X,\mathcal{A}^{\beta\gamma-bc})_{\mathcal{Q}},d_{\mathrm{Norm}}+d_{\mathrm{dR}})\rightarrow({\BVk}\otimes_{\mathbb{C}}\tilde{C}^{\mathrm{ch}}(X,\omega_X)_{\mathcal{Q}},d_{\mathrm{Norm}}+d_{dR}).
$$

\subsection{Feynman diagrams and the chiral differential}
In this section, we will prove that $\mathcal{W}$ intertwines chiral differentials.

We extend the BV operator to $O_{\mathbf{k}}^{\otimes k}$ by defining
$$
\Delta_{\mathrm{BV}}(O_1\otimes \cdots\otimes O_k)=\frac{-i}{\mathrm{Im}(\tau)}\sum_{p,q}\sum_{i,j}\omega^{pq}{(\frac{\partial}{\partial \ap})_i\otimes(\frac{\partial}{\partial \aqc})_j}(O_1\otimes \cdots\otimes O_k)
$$
$$
=\frac{-i}{\mathrm{Im}(\tau)}\sum_{p,q}\sum_{i,j}\pm\omega^{pq}O_1\otimes \cdots \otimes \frac{\partial}{\partial \ap}O_i\otimes\cdots\otimes\frac{\partial}{\partial \aqc}O_j\otimes \cdots\otimes O_k.
$$
Here $\pm$ are Koszul signs. More explicitly, we have
\begin{equation}\label{BVextend}
  \begin{split}
     \Delta_{\mathrm{BV}}(O_1\otimes \cdots\otimes O_k) & =\frac{-i}{\mathrm{Im}(\tau)}\sum_{p,q}\sum_{i<j}(-1)^{\bullet_1}\omega^{pq}O_1\otimes \cdots \otimes \frac{\partial}{\partial \ap}O_i\otimes\cdots\otimes\frac{\partial}{\partial \aqc}O_j\otimes \cdots\otimes O_k \\
       &+\frac{-i}{\mathrm{Im}(\tau)}\sum_{p,q}\sum_{i>j}(-1)^{\bullet_2}\omega^{pq}O_1\otimes \cdots \otimes \frac{\partial}{\partial \aqc}O_j\otimes\cdots\otimes\frac{\partial}{\partial \ap}O_i\otimes \cdots\otimes O_k,
  \end{split}
\end{equation}
where
$$
\bullet_1=p(a^p)\cdot (p(O_1)+\cdots+p(O_{i-1}))+(p(a^q)+1)\cdot (p(O_1)+\cdots+p(O_{j-1})),
$$
$$
\bullet_2=p(a^p)\cdot(p(a^q)+1)+\bullet_1=\bullet_1.
$$
Here we use the fact that $p(\ap)+p(\aq)=\bar{0}$ for nonvanishing $\omega^{pq}$. From the above definition, we have
$$
\Delta_{\mathrm{BV}}\circ \mathbf{Mult}=\mathbf{Mult}\circ \Delta_{\mathrm{BV}}.
$$
\begin{lem}
For $\eta\in \Gamma(X^n,(\mathcal{A}\shift)^{\boxtimes n}(*\Delta_{\{1,\dots,n\}})_\mathcal{Q})$, we have
  $$
  \bar{\partial}(\mathcal{W}(\eta))=-\hbar\Delta_{\mathrm{BV}}\mathcal{W}(\eta)+\mathcal{W}(\bar{\partial}\eta).
  $$
\end{lem}
\begin{proof}

  It is clear that $[\bar{\partial}, \hbar\mathfrak{Q}+D]=0$ .  We now compute $[\bar{\partial},\hbar\mathfrak{P}]$
  \begin{align*}
&[\bar{\partial},\hbar\mathfrak{P}]=\sum_i[\bar{\partial}_{z_i},\hbar\mathfrak{P}]\\
&=\sum_i[\bar{\partial}_{z_i}, \hbar\sum_{j,j>i}\mathcal{P}_{ij}+\hbar\sum_{k,k<i}\mathcal{P}_{ki}]\\
&= \hbar\sum_i [\bar{\partial}_{z_i}, \sum_{j,j>i}\sum_{p,q}P(z_i,z_j)\omega_{pq}(\frac{\partial}{\partial a^p})_i\otimes(\frac{\partial}{\partial a^q})_j+\sum_{k<i}\sum_{p,q}P(z_k,z_i)\omega_{pq}(\frac{\partial}{\partial a^p})_k\otimes(\frac{\partial}{\partial a^q})_i]   \\
   & =\hbar\sum_i\sum_{j,j>i}\sum_{p,q}\frac{-id\bar{z}_i}{\mathrm{Im}(\tau)}\omega_{pq}(\frac{\partial}{\partial a^p})_i\otimes(\frac{\partial}{\partial a^q})_j-\hbar\sum_i\sum_{k,k<i}\sum_{p,q}\frac{-id\bar{z}_i}{\mathrm{Im}(\tau)}\omega_{pq}(\frac{\partial}{\partial a^p})_k\otimes(\frac{\partial}{\partial a^q})_i.
\end{align*}
In the last step we use (\ref{Zeromodes}). Note that the minus sign of the second term comes from the fact that the function $P(z,w)=P(z-w)$ is odd and $z_i$ is in the second position in $P(z_k,z_i)$.   On the other hand, using the identification
$$
V\simeq \mathbb{C}[[\partial^ka^s]]((\hbar))_{s=1,\dots,\dim \mathrm{\mathbf{L}}, k\geq 0}\subset \mathbb{C}[[\partial^ka^s,\partial^ka^{s\dagger}]]((\hbar))_{s=1,\dots,\dim \mathrm{\mathbf{L}}, k\geq 0}\cong O_{\mathbf{k}} ,
$$
 we can view $\Delta_{\mathrm{BV}}, \hbar\mathfrak{P}$ and $\mathfrak{Q}$ as operators acting on both $V^{\otimes n}$ and $O_{\mathbf{k}}^{\otimes n}$. Since
 $$
 \Delta_{\mathrm{BV}}|_{V\simeq\mathbb{C}[[\partial^ka^s]]((\hbar))_{s=1,\dots,\dim \mathrm{\mathbf{L}}, k\geq 0}}=0,
 $$
  to get $\Delta_{\mathrm{BV}}\mathcal{W}(\eta)$ we only need to compute the commutator
\begin{align*}
&-[\Delta_{\mathrm{BV}}, \hbar\mathfrak{P}+\hbar\mathfrak{Q}+D]=-\sum_i[\Delta_{\mathrm{BV}},D_i]\\
& = \sum_i[\sum_{k,k<i}\sum_{p,q}\frac{i}{\mathrm{Im}(\tau)}\omega_{pq}(\frac{\partial}{\partial \ap})_k\otimes (\frac{\partial}{ \partial \aqc})_i+\sum_{j,j>i}\sum_{p,q}\frac{i}{\mathrm{Im}(\tau)}\omega_{pq}(\frac{\partial}{\partial \aqc})_i\otimes (\frac{\partial}{ \partial \ap})_j,\sum_r(a^{r\dagger}\frac{\partial}{\partial a^r})_id\bar{z}_i]   \\
&+\sum_i[\sum_{p,q}\frac{i}{\mathrm{Im}(\tau)}\omega_{pq}(\frac{\partial}{\partial \aqc})_i\cdot (\frac{\partial}{ \partial \ap})_i,\sum_r(a^{r\dagger}\frac{\partial}{\partial a^r})_id\bar{z}_i].
\end{align*}
Here the first and the second term correspond to the BV operator acting on different tensor components and the same tensor component, respectively. By direct computations, the above commutator equals to
\begin{align*}
   & \sum_{i}\sum_{k,k<i}\sum_{p,q}\frac{i}{\mathrm{Im}(\tau)}\omega_{pq}(\frac{\partial}{\partial \ap})_k\otimes (\frac{\partial}{ \partial \aq}d\bar{z}_i)_i+\sum_i\sum_{j,j>i}\sum_{p,q}\frac{i}{\mathrm{Im}(\tau)}\omega_{pq}(\frac{\partial}{\partial \aq}d\bar{z}_i)_i\otimes (\frac{\partial}{ \partial \ap})_j\\
   &+\underbrace{\sum_i\sum_{p,q}\frac{i}{\mathrm{Im}(\tau)}\omega_{pq}(\frac{\partial}{\partial \aq})_i\cdot (\frac{\partial}{ \partial \ap})_id\bar{z}_i}_{=0}\\
   &=\sum_{i}\sum_{k,k<i}\sum_{p,q}(-1)^{p(\ap)+p(\aq)}\frac{id\bar{z}_i}{\mathrm{Im}(\tau)}\omega_{pq}(\frac{\partial}{\partial \ap})_k\otimes (\frac{\partial}{ \partial \aq})_i+\sum_i\sum_{j,j>i}\sum_{p,q}(-1)^{p(\aq)}\frac{id\bar{z}_i}{\mathrm{Im}(\tau)}\omega_{pq}(\frac{\partial}{\partial \aq})_i\otimes (\frac{\partial}{ \partial \ap})_j\\
   &=\sum_{i}\sum_{k,k<i}\sum_{p,q}\frac{id\bar{z}_i}{\mathrm{Im}(\tau)}\omega_{pq}(\frac{\partial}{\partial \ap})_k\otimes (\frac{\partial}{ \partial \aq})_i-\sum_i\sum_{j,j>i}\sum_{p,q}(-1)^{p(\ap)p(\aq)+p(\ap)}\frac{id\bar{z}_i}{\mathrm{Im}(\tau)}\omega_{pq}(\frac{\partial}{\partial \ap})_i\otimes (\frac{\partial}{ \partial \aq})_j\\
   &=\sum_{i}\sum_{k,k<i}\sum_{p,q}\frac{id\bar{z}_i}{\mathrm{Im}(\tau)}\omega_{pq}(\frac{\partial}{\partial \ap})_k\otimes (\frac{\partial}{ \partial \aq})_i-\sum_i\sum_{j,j>i}\sum_{p,q}\frac{id\bar{z}_i}{\mathrm{Im}(\tau)}\omega_{pq}(\frac{\partial}{\partial \ap})_i\otimes (\frac{\partial}{ \partial \aq})_j.
\end{align*}
Here we use the fact that $\omega$ is symplectic, that is, $\omega_{pq}=-(-1)^{p(a^p)p(a^q)}\omega_{qp}=-(-1)^{p(a^p)}\omega_{qp}$. Now we conclude that $-[\hbar\Delta_{\mathrm{BV}}, \hbar\mathfrak{P}+\hbar\mathfrak{Q}+D]$ is equal to $[\bar{\partial}, \hbar\mathfrak{P}+\hbar\mathfrak{Q}+D]$. Then
$$
-[\hbar\Delta_{\mathrm{BV}},e^{\hbar\mathfrak{P}+\hbar\mathfrak{Q}+D}]=[\bar{\partial},e^{\hbar\mathfrak{P}+\hbar\mathfrak{Q}+D}],
$$
and the proof is complete by observing that $[\Delta_{\mathrm{BV}},\mathbf{p}_{\mathrm{BV}}\circ \mathbf{Mult}]=[\bar{\partial},\mathbf{p}_{\mathrm{BV}}\circ \mathbf{Mult}]=0$.
\end{proof}

\begin{rem}
  This lemma can be viewed as a two-dimensional analogue of \cite[Lemma 2.33]{BVQandindex}.
\end{rem}

\begin{prop}\label{WDChirall}
  For $T'\in \equiQ(T,|T|-1)$, we have following commutative diagram
  $$
\begin{tikzcd}
\Gamma(X^T,(\mathcal{A}\shift)^{\boxtimes T}(*\Delta_T)_\mathcal{Q}) \arrow[d, "\bar{\partial}+2\pi id_{\mathrm{ch},(T,T')}"] \arrow[r,"\mathcal{W}"] &  {\BVk}\otimes_{\mathbb{C}}\Gamma(X^T,(\omega_X\shift)^{\boxtimes T}(*\Delta_T)_\mathcal{Q})\arrow[d, "\bar{\partial}+2\pi id_{\mathrm{ch},(T,T')}-\hbar\Delta_{\mathrm{BV}}"] \\
\Gamma(X^T,\Delta_*^{(T/T')}((\mathcal{A}\shift)^{\boxtimes T'}(*\Delta_{T'})_\mathcal{Q})) \arrow[r," \mathcal{W}"]                    & {\BVk}\otimes_{\mathbb{C}}\Gamma(X^T,\Delta_*^{(T/T')}((\omega_X\shift)^{\boxtimes T'}(*\Delta_{T'})_\mathcal{Q}))
\end{tikzcd}
$$
\end{prop}

\begin{proof}
  We need to prove
\begin{equation}\label{FeynmanChiral}
 d_{\mathrm{ch},(T,T')}(\mathcal{W}(f\cdot v))=\mathcal{W}(d_{\mathrm{ch},(T,T')}(f\cdot v)).
\end{equation}
We first identify $T$ and $\{1,\dots,n\}$, where $n=|T|$. Suppose $T'\in \equiQ(T,n-1)$ is described by colliding $i$-th and $j$-th positions ($i<j$).  Recall
$$
\mathfrak{P}=\mathcal{P}_{ij}+\sum_{\bullet\neq i,j}\mathcal{P}_{i\bullet}+\sum_{\bullet\neq i,j}\mathcal{P}_{j\bullet}+\sum_{\substack{\blacktriangle\neq i,j\\\blacktriangledown\neq i,j}}\mathcal{P}_{\blacktriangle\blacktriangledown},
$$
$$
\mathfrak{Q}=\mathcal{Q}_{i}+\mathcal{Q}_{j}+\sum_{\bullet\neq i,j}\mathcal{Q}_{\bullet},
$$
We write
$$
\mathfrak{P}(i\rightarrow j):=\mathcal{P}_{ij}(i\rightarrow j)+\sum_{\bullet\neq i,j}\mathcal{P}_{i\bullet}(i\rightarrow j)+\sum_{\bullet\neq i,j}\mathcal{P}_{j\bullet}+\sum_{\substack{\blacktriangle\neq i,j\\\blacktriangledown\neq i,j}}\mathcal{P}_{\blacktriangle\blacktriangledown},
$$
here
$$
\mathcal{P}_{ij}(i\rightarrow j)=\mathcal{P}_{ij}^{\mathrm{Sing}}(i\rightarrow j)+\mathcal{Q}_{ij}(i\rightarrow j),
$$
$$
\mathcal{P}_{ij}^{\mathrm{Sing}}(i\rightarrow j):=\sum_{p,q}\sum_{k,l}\partial^k_{z_i}\partial^l_{z_j}(\frac{i}{\pi (z_i-z_j)})\cdot \omega_{pq}(\frac{\partial}{\partial (\partial^k a^p)})_i\otimes (\frac{\partial}{\partial (\partial^l a^q)})_j,
$$
$$
\mathcal{Q}_{ij}(i\rightarrow j):=\sum_{r\geq 0}\sum_{k,l}\sum_{p,q}Q(k+r,l)(z_i-z_j)^{(r)}\omega_{pq}(\frac{\partial}{\partial (\partial^k a^p)})_i\otimes (\frac{\partial}{\partial (\partial^l a^q)})_j,
$$
and
$$
\mathcal{P}_{i\bullet}(i\rightarrow j):=\sum_{r\geq 0}\sum_{k,l}\sum_{p,q}P_{j\bullet}(k+r,l)(z_i-z_j)^{(r)} \omega_{pq}(\frac{\partial}{\partial (\partial^k a^p)})_i\otimes (\frac{\partial}{\partial (\partial^l a^q)})_{\bullet},\quad \bullet\neq i,j.
$$
Here $(z_i-z_j)^{(r)}:=\frac{(z_i-z_j)^r}{r!}, r\geq 0$. In other words, the notation $(i\rightarrow j)$ just means the Taylor expansions of the regular part of $P_{ij}$ and $P_{i\bullet}$ at $z_i=z_j$ (here $\bullet\neq i,j$).

We also introduce
$$
\mathrm{\mathbf{Mult}}_{i\rightarrow j}(o_1\otimes \cdots\otimes o_i\otimes \cdots\otimes o_j\otimes \cdots\otimes o_n)=
$$
$$
o_1\otimes \cdots\otimes 1\otimes \cdots\otimes o_i\cdot o_j\otimes \cdots\otimes o_n,
$$
here $o_1\otimes \cdots\otimes o_n\in (O_{\mathbf{k}})^{\otimes n}$ .

  With above notations, (\ref{FeynmanChiral}) is equivalent to the following identity (recall the residue formula (\ref{OPECHIRAL}) of the chiral operation)
$$
  \mathrm{Res}_{z_i\rightarrow z_ j} f\cdot e^{(z_i-z_j)\otimes \vec{\partial}_{z_i}}\cdot \mathbf{p}_{\mathrm{BV}}\circ \mathbf{Mult}(e^{\hbar\mathfrak{P}(i\rightarrow j)+\hbar\mathfrak{Q}+{D}}v)
  $$

  $$
  =\mathbf{p}_{\mathrm{BV}}\circ \mathbf{Mult}\left(e^{\hbar\mathfrak{P}+\hbar\mathfrak{Q}+{D}}\mathrm{\mathbf{Mult}}_{i\rightarrow j}\mathrm{Res}_{z_i\rightarrow z_j}f\cdot e^{(z_i-z_j)\otimes \vec{\partial}_{z_i}}
  \cdot e^{(z_i-z_j)(L_{-1})_i} \cdot e^{\hbar \mathcal{P}^{\mathrm{Sing}}_{ij}(i\rightarrow j)}v\right).
  $$
  Here we use the Wick theorem, see Appendix \ref{Wichtheorem}.

  We begin by proving
  $$
  e^{D}\mathrm{\mathbf{Mult}}_{i\rightarrow j}\mathrm{Res}_{z_i\rightarrow z_j}=\mathrm{\mathbf{Mult}}_{i\rightarrow j}\mathrm{Res}_{z_i\rightarrow z_j}e^{D}.
  $$
  This follows from the definition of $\mathrm{Res}_{z_i\rightarrow z_j}$ and type reasons $$D_i(o_i)|_{z_i=z_j}\cdot D_j(o_j)=\cdots d\bar{z}_j\cdot d\bar{z}_j=0.$$
  Since $\mathbf{p}_{\mathrm{BV}}$ is the quotient map $O_{\mathbf{k}}/(L_{-1}O_{\mathbf{k}})\rightarrow \BVk$, we have
  $$
  \mathbf{p}_{\mathrm{BV}}\circ \mathbf{Mult}(-)=\mathbf{p}_{\mathrm{BV}}\circ \mathbf{Mult} ( e^{(z_i-z_j)(L_{-1})_i}\cdot -)
  $$

Now we move the operators $\hbar\mathfrak{P}+\hbar\mathfrak{Q}$ inside. We have
$$
\mathcal{P}_{i\bullet}(i\rightarrow j)|_{z_i=z_j}=\sum_{k,l}\sum_{p,q}P_{j\bullet }(k,l)\omega_{pq}(\frac{\partial}{\partial (\partial^k a^p)})_i\otimes (\frac{\partial}{\partial (\partial^l a^q)})_{\bullet},
$$
 $$
  \mathcal{Q}_{ij}(i\rightarrow j)|_{z_i=z_j}=\sum_{k,l}\sum_{p,q}Q(k,l)\omega_{pq}(\frac{\partial}{\partial (\partial^k a^p)})_i\otimes (\frac{\partial}{\partial (\partial^l a^q)})_j.
  $$

  Then we have following identities (\ref{PMultCommute})-(\ref{QLCommute}). The first two identities show how to move the operator $e^{\hbar\mathfrak{P}+\hbar\mathfrak{Q}}$ to the right of the multiplication symbol $\mathbf{\mathrm{Mult}}_{i\rightarrow j}$
  \begin{equation}\label{PMultCommute}
  \begin{split}
     e^{\hbar\mathcal{P}_{j\bullet}}\mathrm{\mathbf{Mult}}_{i\rightarrow j}& = \mathrm{\mathbf{Mult}}_{i\rightarrow j}e^{\hbar\mathcal{P}_{i\bullet}(i\rightarrow j)|_{z_i=z_j}+\hbar\mathcal{P}_{j\bullet}},
  \end{split}
  \end{equation}
   here $\bullet\neq i,j$.
  \begin{equation}\label{QMultCommute}
  e^{\hbar\mathcal{Q}_{j}}\mathrm{\mathbf{Mult}}_{i\rightarrow j}=\mathrm{\mathbf{Mult}}_{i\rightarrow j}e^{\hbar\mathcal{Q}_{j}+\hbar\mathcal{Q}_{i}+\hbar\mathcal{Q}_{ij}(i\rightarrow j)|_{z_i=z_j}}.
  \end{equation}
  The remaining four identities allow us to exchange the positions of $e^{\hbar \mathfrak{P}+\hbar\mathfrak{Q}}$ and $e^{(z_i-z_j)(L_{-1})_i}$
 \begin{equation}\label{QiCommute}
  e^{\hbar\mathcal{Q}_{i}}e^{(z_i-z_j)(L_{-1})_i} =e^{(z_i-z_j)(L_{-1})_i}  e^{\hbar\mathcal{Q}_{i}},
  \end{equation}
\begin{equation}\label{QjCommute}
  e^{\hbar\mathcal{Q}_{j}}e^{(z_i-z_j)(L_{-1})_i} =e^{(z_i-z_j)(L_{-1})_i}  e^{\hbar\mathcal{Q}_{j}},
\end{equation}

  \begin{equation}\label{PLCommute}
  e^{\hbar\mathcal{P}_{i\bullet}(i\rightarrow j)|_{z_i=z_j}}e^{(z_i-z_j)(L_{-1})_i} =e^{(z_i-z_j)(L_{-1})_i}e^{\hbar \mathcal{P}_{i\bullet}(i\rightarrow j)},
  \end{equation}
   here $\bullet\neq i,j$.
\begin{equation}\label{QLCommute}
    e^{\hbar\mathcal{Q}_{ij}(i\rightarrow j)|_{z_i=z_j}}e^{(z_i-z_j)(L_{-1})_i} =e^{(z_i-z_j)(L_{-1})_i}e^{\hbar\mathcal{Q}_{ij}(i\rightarrow j)}.
\end{equation}

The above identities will be proved in Appendix \ref{Identities}.  Combining these, we obtain
  $$
  \mathbf{p}_{\mathrm{BV}}\circ \mathbf{Mult}\left(e^{\hbar\mathfrak{P}+\hbar\mathfrak{Q}+D}\mathrm{\mathbf{Mult}}_{i\rightarrow j}\mathrm{Res}_{z_i\rightarrow z_j}f\cdot e^{(z_i-z_j)\otimes \vec{\partial}_{z_i}}
  \cdot e^{(z_i-z_j)(L_{-1})_i}\cdot e^{\hbar \mathcal{P}^{\mathrm{Sing}}_{ij}(i\rightarrow j)}v\right)
  $$
  \begin{align*}
     & =\mathbf{p}_{\mathrm{BV}}\circ \mathbf{Mult}(\mathrm{\mathbf{Mult}}_{i\rightarrow j}\mathrm{Res}_{z_i\rightarrow  z_j}e^{\hbar\sum_{\bullet\neq i,j}\mathcal{P}_{i\bullet}(i\rightarrow j)|_{z_i=z_j}+\hbar\sum_{\bullet\neq i,j}\mathcal{P}_{j\bullet}}\cdot e^{\hbar\mathcal{Q}_{j}+\hbar\mathcal{Q}_{i}+\hbar\mathcal{Q}_{ij}(i\rightarrow j)|_{z_i\rightarrow  z_j}}\cdot e^{\hbar\sum_{\tiny{\substack{\blacktriangle\neq i,j\\\blacktriangledown\neq i,j}}}\mathcal{P}_{\blacktriangle\blacktriangledown}+\hbar\sum_{\bullet\neq i,j}\mathcal{Q}_{\bullet}} \\
     &\cdot f\cdot e^{(z_i-z_j)\otimes \vec{\partial}_{z_i}}
  \cdot e^{(z_i-z_j)(L_{-1})_i}\cdot e^{\hbar \mathcal{P}^{\mathrm{Sing}}_{ij}(i\rightarrow j)}e^{D}v)
  \end{align*}
  $$
  =\mathbf{p}_{\mathrm{BV}}\circ \mathbf{Mult}\left(\mathrm{\mathbf{Mult}}_{i\rightarrow j}\mathrm{Res}_{z_i \rightarrow z_j}f\cdot e^{(z_i-z_j)\otimes \vec{\partial}_{z_i}}
  \cdot e^{(z_i-z_j)(L_{-1})_i}\cdot e^{\hbar\mathfrak{P}(i\rightarrow j)+\hbar\mathfrak{Q}+D}v\right)
  $$
    $$
  =\mathrm{Res}_{z_i\rightarrow  z_j}f\cdot e^{(z_i-z_j)\otimes \vec{\partial}_{z_i}}
  \cdot \mathbf{p}_{\mathrm{BV}}\circ \mathbf{Mult}( e^{\hbar\mathfrak{P}(i\rightarrow j)+\hbar\mathfrak{Q}+D}v)
  $$
  which proves the proposition.
\end{proof}
\subsection{Main Theorem}\label{MainTheoremBV}
Now we are ready to state our main theorem. We first recall some notations and definitions.
\begin{itemize}
  \item Let $\mathbf{L}$ be a $\mathbb{Z}$-graded vector space equipped with a symplectic form of degree 0.

      \item Denote $V=V^{\beta\gamma-bc}$ to be the associated free $\beta\gamma-bc$ graded vertex operator algebra. The $\mathbb{Z}$-grading of $V$ comes from that of $\mathbf{L}$.

      \item Let $\mathcal{A}$ be the corresponding graded chiral algebra on an elliptic curve $X$.

  \item The chiral chain complex is denoted by $\tilde{C}^{\mathrm{ch}}(X,\mathcal{A})_{Q}$ equipped with the differential $$
      d_{\mathrm{tot}}=d_{\mathrm{DR}}+\bar{\partial}+2\pi i d_{\mathrm{ch}}+d_{\mathrm{Norm}}.
      $$
          \item Let $O_{\mathrm{BV},\mathbf{k}}(\mathbf{L})=O_{\mathbf{k}}(\mathbf{L})/(L_{-1}O_{\mathbf{k}}(\mathbf{L}))$ be the corresponding BV algebra. We introduce the homological grading on $O_{\mathrm{BV},\mathbf{k}}(\mathbf{L})$ by setting
$$
\mathrm{hodeg}(a)=-\mathrm{deg}(a),\quad a\in O_{\mathrm{BV},\mathbf{k}}(\mathbf{L}).
$$
Now the BV algebra $O_{\mathrm{BV},\mathbf{k}}(\mathbf{L})$ becomes a chain complex with differential $-\hbar\Delta_{\mathrm{BV}}$.

\end{itemize}

We have the following result.

\begin{thm}\label{mainThm}

(1)
The map
$$
\mathcal{W}:(\tilde{C}^{\mathrm{ch}}(X,\mathcal{A})_{Q},d_{\mathrm{tot}})\rightarrow(\tilde{C}^{\mathrm{ch}}(X,\omega_X)_{Q}\otimes_{\mathbb{C}}\BVk,d_{\mathrm{tot}}-\hbar\Delta_{\mathrm{BV}}).
$$
constructed in Section \ref{BVchiral}  is a chain map.

(2) The composition $\mathrm{\mathbf{Tr}}:=\mathrm{tr}\circ \mathcal{W}$ satisfies the quantum master equation. That is,
  $$
  \mathrm{\mathbf{Tr}}: (\tilde{C}^{\mathrm{ch}}(X,\mathcal{A})_{Q},d_{\mathrm{tot}})\rightarrow (\BVLk,-\hbar\Delta_{\mathrm{BV}})
  $$
  is a chain map.

 (3) Both $\mathcal{W}$ and $\mathrm{\mathbf{Tr}}$ are quasi-isomorphisms.
\end{thm}
\begin{proof}[Proof of (1) and (2)]
From the definition of the map $\mathcal{W}=e^{\hbar \mathfrak{P}+\hbar\mathfrak{Q}+D}$ (see (\ref{DefinitionOfW})), we see that $\mathcal{W}$ preserves the homological grading and this implies that $\Trace$ also preserves the homological grading.

From  Lemma \ref{WDmodule} and Proposition \ref{WDChirall}, we know that $\mathcal{W}$ is a chain map. Together with the canonical trace map $\mathrm{tr}$ on $\tilde{C}^{\mathrm{ch}}(X,\omega_X)_{Q}$, we get the second statement.
\end{proof}

In the rest of this section, we will prove that the trace map $\mathrm{\mathbf{Tr}}$ is a quasi-isomorphism. We first compute the homology of the BV complex $(\BVLk,-\hbar\Delta_{\mathrm{BV}})$.

Since the pairing $\langle-,-\rangle$ is nondegenerate, we can find a basis $\{\phi_{2k,i},\phi_{2k,i}^{\vee},\psi_{2k+1,j},\psi_{2k+1,j}^{\vee}\}$ of $\mathrm{\mathbf{L}}$ such that
$$
\mathrm{span}\{\phi_{0,i},\phi_{0,i}^{\vee}\}_{i=1,\dots,\frac{\dim \mathrm{\mathbf{L}}_0}{2}}=\mathrm{\mathbf{L}}_0,\quad \langle \phi_{0,i},\phi_{0,j}^{\vee}\rangle=\delta_{ij},
$$
$$
\mathrm{span}\{\phi_{2k,i}\}_{i=1,\dots,\dim\mathrm{\mathbf{L}}_{2k}}\oplus \mathrm{span}\{\phi_{2k,i}^{\vee}\}_{i=1,\dots,\dim\mathrm{\mathbf{L}}_{-2k}=\dim\mathrm{\mathbf{L}}_{2k}}=\mathrm{\mathbf{L}}_{2k}\oplus\mathrm{\mathbf{L}}_{-2k},k>0, \quad
$$
$$
\langle \phi_{2k,i},\phi_{2k,j}^{\vee}\rangle=\delta_{ij},
$$
and
$$
\mathrm{span}\{\psi_{2k-1,i}\}_{i=1,\dots,\dim\mathrm{\mathbf{L}}_{2k-1}}\oplus \mathrm{span}\{\psi_{2k-1,i}^{\vee}\}_{i=1,\dots,\dim\mathrm{\mathbf{L}}_{-2k+1}=\dim\mathrm{\mathbf{L}}_{2k-1}}=\mathrm{\mathbf{L}}_{2k-1}\oplus\mathrm{\mathbf{L}}_{-2k+1},k>0,
$$
$$
\langle \psi_{2k-1,i},\psi_{2k-1,j}^{\vee}\rangle=\delta_{ij}.
$$
Our BV algebra is
$$
\mathbb{C}[[\phi_{2k,i},\phi_{2k,j}^{\vee},\psi_{2k-1,i},\psi_{2k-1,j}^{\vee},\phi_{2k,i}^{\dagger},\phi_{2k,j}^{\vee\dagger},\psi_{2k-1,i}^{\dagger},\psi_{2k-1,j}^{\vee\dagger}]].
$$
We will also use the notation
$$
x_{2k,i}:=\phi_{2k,i},x_{2k,i}^{\vee}:=\phi_{2k,j}^{\vee},x_{2k-1,i}:=\psi_{2k-1,i}^{\dagger},x_{2k-1,j}^{\vee}:=\psi_{2k-1,j}^{\vee\dagger},
$$
here the grading is
$$
\deg x_{2k,i}=2k,\deg x_{2k,i}^{\vee}=-2k,\deg x_{2k-1,i}=2k-2,\deg x_{2k-1,j}^{\vee}=-2k,
$$
and similarly
$$
\theta_{2k,i}:=\phi_{2k,i}^{\dagger},\theta_{2k,j}^{\vee}:=\phi_{2k,j}^{\vee\dagger},{\theta_{2k-1,i}}:=\psi_{2k-1,i},{\theta_{2k-1,j}^{\vee}}:=\psi_{2k-1,j}^{\vee},
$$
$$
\deg\theta_{2k,i}=2k-1,\deg\theta_{2k,j}^{\vee}=-2k-1,\deg{\theta_{2k-1,i}}=2k-1,\deg{\theta_{2k-1,j}^{\vee}}=-2k+1.
$$

Now the BV operator is
$$
\Delta_{\mathrm{BV}}=\frac{-i}{\mathrm{Im}(\tau)}(\sum \frac{\partial }{\partial x_{k,i}}\frac{\partial}{\partial\theta^{\vee}_{k,i}}+\frac{\partial }{\partial x^{\vee}_{k,i}}\frac{\partial}{\partial\theta_{k,i}}).
$$
To simplify the notation, we sometimes omit the subscript. The homological grading on $O_{\mathrm{BV}}$ is
$$
\mathrm{hodeg}(a)=-\mathrm{deg}(a),\quad a\in O_{\mathrm{BV}}.
$$
In particular
$$
\mathrm{hodeg}(\phi)=\mathrm{hodeg}(\psi)=0,\mathrm{hodeg}(\phi^{\dagger})=\mathrm{hodeg}(\psi^{\dagger})=1.
$$

Now it is easy to see that our BV algebra $O_{\mathrm{BV}}(\mathbf{L})=\mathbb{C}[[x,\theta^{\vee},x^{\vee},\theta]]$ is isomorphic to the BV algebra of formal polyvector fields $\mathbb{C}[[x,\partial_x,x^{\vee},\partial_{x^{\vee}}]]$ if we set
$$
\theta^{\vee}=\partial_x,\theta=\partial_{x^{\vee}}.
$$

Now consider the formal de Rham algebra
$$
\hat{\Omega}:=\mathbb{C}[[x,dx,x^{\vee},dx^{\vee}]]
$$
and $\hat{d}$ is the de Rham differential. We choose a constant top form
$$
\Omega_{\mathrm{top}}:=\prod dx\cdot \prod dx^{\vee},
$$
and its degree is
$$
\mathrm{hodeg} (\Omega_{\mathrm{top}})=-\dim \mathrm{\mathbf{L}}_{\bar{0}}.
$$
Denote by $\dashv \Omega_{\mathrm{top}}$ the contraction with $\Omega_{\mathrm{top}}$, we have an isomorphism of chain complexes
$$
\dashv \Omega_{\mathrm{top}}:(O_{\mathrm{BV}}(\mathbf{L}),\Delta_{\mathrm{BV}})\rightarrow (\hat{\Omega}[-\dim\mathrm{\mathbf{L}}_{\bar{0}}],\frac{-i}{\mathrm{Im}(\tau)}\hat{d}).
$$
Recall that the top fermion integration  $\int_{\mathrm{BV},\mathrm{top}}: (O_{\mathrm{BV}}(\mathbf{L}),\Delta_{\mathrm{BV}})\rightarrow (\mathbb{C}[-\dim \mathrm{\mathbf{L}}_{\bar{0}}],0)$ can be written as the composition of $\dashv \Omega_{\mathrm{top}}$ and the quasi-isomorphism
$$
(\hat{\Omega}[-\dim\mathrm{\mathbf{L}}_{\bar{0}}],\frac{-i}{\mathrm{Im}(\tau)}\hat{d})\rightarrow (\mathbb{C}[-\dim \mathrm{\mathbf{L}}_{\bar{0}}],0).
$$

 We can extend it linearly to a quasi-isomorphism $$\int_{\mathrm{BV},\mathrm{top}}: (O_{\mathrm{BV}}(\mathbf{L})\otimes_{\mathbb{C}}\mathbf{k},
\hbar\Delta_{\mathrm{BV}})\rightarrow (\mathbf{k}[-\dim \mathrm{\mathbf{L}}_{\bar{0}}],0).$$

We are ready to prove that the trace map is a quasi-isomorphism. We first state a direct corollary of our main theorem.
\begin{cor}\label{mainCor}
 We can define a chain map

  $$
\TraceBV:=\int_{\mathrm{BV},\mathrm{top}}\circ\mathrm{\mathbf{Tr}}:(\tilde{C}^{\mathrm{ch}}(X,\mathcal{A})_{Q},d_{\mathrm{tot}})\rightarrow   (\mathbf{k}[-\dim \mathrm{\mathbf{L}}_{\bar{0}}],0).
  $$
Furthermore, $\TraceBV$  is a quasi-isomorphism.
\end{cor}

\begin{proof}[Proof of (3)]
  We only need to prove that $\Trace$ is a quasi-isomorphism. We first recall the definition of  another two chiral chain complexes $\tilde{C}^{\mathrm{ch}}(X,\mathcal{A}_{\mathcal{P}})_{\mathcal{Q}}$ and $C^{\mathrm{ch}}(X,\mathcal{A})_{\mathcal{PQ}}$ in \cite[pp.306]{book:635773}.

  For a right $\mathcal{D}$-module $M$ set
  $$
h(M):=M\otimes_{\mathcal{D}_X}\mathcal{O}_X.
$$
Consider the (non-unital) commutative $\mathcal{D}_X$-algebra resolution of $\mathcal{O}_X$
$$
0\rightarrow \mathcal{P}^{-1}\rightarrow\mathcal{P}^0\xrightarrow{\varepsilon_{\mathcal{P}}}\mathcal{O}_X,
$$
where
$$
\mathcal{P}^0=\mathrm{Sym}^{>0}_{\mathcal{O}_X}\mathcal{D}_X,\quad
\mathcal{P}^{-1}:=\mathrm{Ker}(\varepsilon_{\mathcal{P}}),
$$
and for any $k>0$
$$
\varepsilon_{\mathcal{P}}(\prod_{i=1}^kP_i(\partial_{z,i}))=\prod_{i=1}^k(P_i(\partial_{z,i})\cdot 1)\in \mathcal{O}_X,\quad \prod_{i=1}^kP_i(\partial_{z,i})\in \mathrm{Sym}^{k}_{\mathcal{O}_X}\mathcal{D}_X.
$$

We now have the chiral chain complex $\tilde{C}^{\mathrm{ch}}(X,\mathcal{A}_{\mathcal{P}})_{\mathcal{Q}}$ associated to the chiral algebra $\mathcal{A}_{\mathcal{P}}=\mathcal{A}\otimes \mathcal{P}$ and define
$$
{C}^{\mathrm{ch}}(X,\mathcal{A})_{\mathcal{PQ}}:=\Gamma(X^{\mathcal{S}},h(C(\mathcal{A}_{\mathcal{P}})_{\mathcal{Q}}))=\lim_{\rightarrow}(I\in X^\mathcal{S},\Gamma(X^I,h(C(\mathcal{A}_{\mathcal{P}})_{\mathcal{Q},X^I}))).
$$

By \cite[pp.307,(4.2.12.6)]{book:635773}, we have quasi-isomorphisms
$$
\tilde{C}^{\mathrm{ch}}(X,\mathcal{A})_{\mathcal{Q}}\xleftarrow{\sim}\tilde{C}^{\mathrm{ch}}(X,\mathcal{A}_{\mathcal{P}})_{\mathcal{Q}}\xrightarrow{\sim} C^{\mathrm{ch}}(X,\mathcal{A})_{\mathcal{PQ}}.
$$
The first quasi-isomorphism is induced by $\varepsilon_{\mathcal{P}}$ and we get the induced trace map on $\tilde{C}^{\mathrm{ch}}(X,\mathcal{A}_{\mathcal{P}})_{\mathcal{Q}}$
$$
\Trace_{\mathcal{P}}:\tilde{C}^{\mathrm{ch}}(X,\mathcal{A}_{\mathcal{P}})_{\mathcal{Q}}\rightarrow O_{\mathrm{BV},\mathbf{k}}(\mathbf{L}).
$$
The second quasi-isomorphism is given by the projection $\mathrm{DR}(-)\rightarrow h(-)$ and we denote it by $p$.

By Remark \ref{twistedEnvelope}, we have
$$
\mathcal{A}=U(L)^{\flat}_{\hbar}.
$$
Recall there is a canonical trace map
$$
\mathrm{tr}:\Gamma(X,h((\omega_X\otimes\mathcal{P})_{\mathcal{Q}}))\rightarrow \mathbb{C}[-1].
$$
Define $L_{\mathcal{PQ}}:=L\otimes\mathcal{P}\otimes\mathcal{Q},L^{\flat}_{\mathcal{PQ}}:=L^{\flat}\otimes\mathcal{P}\otimes\mathcal{Q}$. We can construct a central extension $\Gamma(X,h(L_{\mathcal{PQ}}))^{\flat}$ of $\Gamma(X,h(L_{\mathcal{PQ}}))$ by

$$
\begin{tikzcd}
0 \arrow[r] & {\Gamma(X,h(\omega_{\mathcal{PQ}}))} \arrow[d, "\mathrm{tr}"'] \arrow[r] & {\Gamma(X,h(L^{\flat}_{\mathcal{PQ}}))} \arrow[d] \arrow[r] & {\Gamma(X,h(L_{\mathcal{PQ}}))} \arrow[d, equal] \arrow[r] & 0 \\
0 \arrow[r] & {\mathbb{C}[-1]} \arrow[r]                                                        & {\Gamma(X,h(L_{\mathcal{PQ}}))^{\flat}} \arrow[r]           & {\Gamma(X,h(L_{\mathcal{PQ}}))} \arrow[r]                                & 0
\end{tikzcd}
$$
Suppose we have a central extension of a super Lie algebra $W$

$$
0\rightarrow \mathbb{C}[-1]\rightarrow W^{\flat}\rightarrow
 W\rightarrow 0.
$$
The $\flat$-twisted Chevalley complex of $W$ is defined by
$$
C(W)^{\flat}:=\frac{\mathrm{Sym}_{\mathbb{C}}W^{\flat}\shift}{\langle 1-(1^{\flat})^k|k\geq 1\rangle},
$$
where $1\in \mathrm{Sym}^0_{\mathbb{C}}W^{\flat}\shift=\mathbb{C}$ and $(1^{\flat})^k\in \mathrm{Sym}^k_{\mathbb{C}}W^{\flat}\shift$ is the $k$-th ($k>0$) symmetric product of $1^{\flat}\in \mathbb{C}\hookrightarrow W^{\flat}\shift$. The differential is denoted by $d_{\mathrm{Chevalley}}$.

By \cite[pp.355,(4.8.3.1)]{book:635773}, there is a quasi-isomorphism
$$
C:(C(\Gamma(X,h(L_{\mathcal{PQ}})))^{\flat}\otimes_{\mathbb{C}}\mathbf{k},\frac{i\hbar}{\pi}d_{\mathrm{Chevalley}})\xrightarrow{\sim} C^{\mathrm{ch}}(X,U(L)^{\flat}_{\hbar})_{\mathcal{PQ}}=C^{\mathrm{ch}}(X,\mathcal{A})_{\mathcal{PQ}}.
$$

We now construct an explicit chain map from $O_{\mathrm{BV},\mathbf{k}}$ to $C(\Gamma(X,h(L_{\mathcal{PQ}})))^{\flat}\otimes_{\mathbb{C}}\mathbf{k}$ by defining
$$
s^{\flat}:O_{\mathrm{BV},\mathbf{k}}(\mathbf{L})\rightarrow C(\Gamma(X,h(L_{\mathcal{PQ}})))^{\flat}\otimes_{\mathbb{C}}\mathbf{k}.
$$
$$
\deg=0:1\mapsto 1,
$$
$$
\deg=|a_i|: a_i\mapsto -\frac{\pi}{\mathrm{Im}(\tau)}\cdot a_idzd\bar{z}\shift
$$
$$
\deg=|a_i|-1:a_i^{\dagger}\mapsto -\frac{\pi}{\mathrm{Im}(\tau)}\cdot a_idz\shift.
$$
This is chain map, since
$$
s^{\flat}(\hbar\Delta_{\mathrm{BV}}(a_i\cdot a_{j}^{\dagger}))=\frac{i\hbar\cdot \omega^{ij}}{\pi}\cdot \frac{-\pi}{\mathrm{Im}(\tau)}=\int_X(-\frac{\pi}{\mathrm{Im}(\tau)})^2\mu[1](a_idz\shift\cdot  a_jdw\shift d\bar{w})
$$
$$
=\mathrm{tr}\left((-\frac{\pi}{\mathrm{Im}(\tau)})^2\mu[1](a_idz\shift\cdot  a_jdw\shift d\bar{w})\right)=\frac{i\hbar}{\pi}d_{\mathrm{Chevalley}}\left((-\frac{\pi}{\mathrm{Im}(\tau)})^2(a_idz\shift\cdot  a_jdw\shift d\bar{w})\right).
$$
This chain map is a quasi-isomorphism since $C(\Gamma(X,h(L_{\mathcal{PQ}})))^{\flat}\otimes_{\mathbb{C}}\mathbf{k}$ is quasi-isomorphic to
$$
C(R\Gamma_{\mathrm{DR}}(X,L))^{\flat}\otimes_{\mathbb{C}}\mathbf{k},
$$
where
$$
R\Gamma_{\mathrm{DR}}(X,L)\simeq R\Gamma(X,\mathcal{O}_X)\otimes_{\mathbb{C}}\mathbf{L}\cong \mathrm{span}\{a_i,a_i^{\dagger}\}[-1].
$$
 In summary, we have the following diagram.
$$
\begin{tikzcd}
{\tilde{C}^{\mathrm{ch}}(X,\mathcal{A})_{\mathcal{Q}}} \arrow[d,"\Trace"']                                                                                       & {\tilde{C}^{\mathrm{ch}}(X,\mathcal{A}_{\mathcal{P}})_{\mathcal{Q}}} \arrow[l, "\sim"'] \arrow[r, "\sim","p"'] \arrow[ld, phantom] \arrow[ld,"\Trace_{\mathcal{P}}"',shift left] & {C^{\mathrm{ch}}(X,\mathcal{A})_{\mathcal{PQ}}}                                                           \\
O_{\mathrm{BV},\mathbf{k}} \arrow[rr, "\sim"',"s^{\flat}"] \arrow[rd, "\sim" description]  &                                                                                                                                            & {C(\Gamma(X,h(L_{\mathcal{PQ}})))^{\flat}\otimes_{\mathbb{C}}\mathbf{k}} \arrow[u, "\sim"',"C"] \arrow[ld, "\sim" description] \\
                                                                                                                                             & {C(R\Gamma_{\mathrm{DR}}(X,L))^{\flat}\otimes_{\mathbb{C}}\mathbf{k}}                                                                                                    &
\end{tikzcd}
$$

We only need to prove that $\Trace_{\mathcal{P}}$ is not trivial on the homology. Consider the element (here we use the basis $\{\phi,\phi^{\vee},\psi,\psi^{\vee}\}$ of $\mathrm{\mathbf{L}}$)
$$
\phi_{2k,i}dz_{(2k,i)}\shift \cdot\phi^{\vee}_{2k,i}dw_{(2k,i)}\shift\in \tilde{C}^{\mathrm{ch}}(X,\mathcal{A}_{\mathcal{P}})_{\mathcal{Q}}.
$$
We have
$$
d_{\mathrm{tot}}(\phi_{2k,i}dz_{(2k,i)}\shift \cdot\phi^{\vee}_{2k,i}dw_{(2k,i)}\shift)=\frac{i\hbar}{\pi}\cdot dw_{(2k,i)}\shift.
$$
By \cite[pp.315,Proposition.4.3.3(i)]{book:635773}, we have
$$
\tilde{C}^{\mathrm{ch}}(X,\omega_{\mathcal{P}})_{\mathcal{Q}}\xrightarrow{\sim}\mathbb{C}.
$$
Here $\omega$ is the unit chiral algebra. Notice that
$$
\tilde{C}^{\mathrm{ch}}(X,\omega_{\mathcal{P}})_{\mathcal{Q}}\subset \tilde{C}^{\mathrm{ch}}(X,\mathcal{A}_{\mathcal{P}})_{\mathcal{Q}}.
$$
Since $d_{\mathrm{tot}}1\cdot dw_{(2k,i)}\shift=0$ and $\mathrm{deg}(1\cdot dw_{(2k,i)}\shift)=-1$, we can find $\alpha_{(2k,i)}\in \tilde{C}^{\mathrm{ch}}(X,\omega_{\mathcal{P}})_{\mathcal{Q}}\subset \tilde{C}^{\mathrm{ch}}(X,\mathcal{A}_{\mathcal{P}})_{\mathcal{Q}}$ such that
$$
d_{\mathrm{tot}}\alpha_{(2k,i)}=-1\cdot dw_{(2k,i)}\shift,\quad \mathrm{deg}(\alpha_{(2k,i)})=-2.
$$
Then
$$
d_{\mathrm{tot}}(\phi_{2k,i}dz_{(2k,i)}\shift \cdot\phi^{\vee}_{2k,i}dw_{(2k,i)}\shift+\frac{i\hbar}{\pi}\alpha_{(2k,i)})=0.
$$
We consider the element $\Psi\cdot \Phi$ in $\tilde{C}^{\mathrm{ch}}(X,\mathcal{A}_{\mathcal{P}})_{\mathcal{Q}}$, where
$$
\Psi=\prod_{k,i}\psi_{2k-1,j}d{z_{(2k-1,j)}}\shift d\bar{z}_{(2k-1,j)}\cdot \prod_{k,i}\psi^{\vee}_{2k-1,j}d{w_{(2k-1,j)}}\shift d\bar{w}_{(2k-1,j)},
$$
$$
\Phi=\prod_{k,i}(\phi_{2k,j}d{z_{(2k,j)}}\shift\cdot \phi^{\vee}_{2k,j}d{w_{(2k,j)}}\shift+\frac{i\hbar}{\pi}\alpha_{(2k,i)}).
$$
Since $\alpha_{(k,i)}$ is central, we have
$$
d_{\mathrm{tot}}(\Psi\cdot \Phi)=0.
$$
By the same Feynman diagram argument as in Appendix \ref{FeynmanCompute}, only one-vertex tree diagram contributes and we have
$$
\Trace_{\mathcal{P}}(\Psi\cdot \Phi)=
$$
$$
\prod_{k,i}\Trace_{\mathcal{P}}(\psi_{2k-1,j}d{z_{(2k-1,j)}}\shift d\bar{z}_{(2k-1,j)})\cdot\Trace_{\mathcal{P}}( \prod_{k,i}\psi^{\vee}_{2k-1,j}d{w_{(2k-1,j)}}\shift d\bar{w}_{(2k-1,j)})
$$
$$
\cdot \prod_{k,i}\left(\Trace_{\mathcal{P}}(\phi_{2k,j}d{z_{(2k,j)}}\shift\cdot \phi^{\vee}_{2k,j}d{w_{(2k,j)}}\shift)+\Trace_{\mathcal{P}}(\frac{i\hbar}{\pi}\alpha_{(2k,i)})\right).
$$
By type reasons, we have
$$
\Trace_{\mathcal{P}}(\frac{i\hbar}{\pi}\alpha_{(2k,i)})=0.
$$
Finally, we get
$$
\Trace_{\mathcal{P}}(\Psi\cdot \Phi)=(-\frac{\mathrm{Im}(\tau)}{\pi})^{\dim \mathrm{\mathbf{L}}}\cdot \prod_{k,i}\psi_{2k-1,j}\cdot\psi^{\vee}_{2k-1,j}\cdot \prod_{k,i}\phi_{2k,j}^{\dagger}\cdot \phi^{\vee\dagger}_{2k,j}
$$
which is a scalar multiple of the top fermion in the BV algebra $O_{\mathrm{BV},\mathbf{k}}$. This proves the result.
\end{proof}
\begin{rem}
  The chiral chain complex $C^{\mathrm{ch}}(X,\mathcal{A})_{\mathcal{PQ}}$ has the advantage of being reasonably small and it is naturally a BV algebra (see \cite[pp.314,Proposition.4.3.1(i)]{book:635773}). However, the direct construction of the trace map on   $C^{\mathrm{ch}}(X,\mathcal{A})_{\mathcal{PQ}}$ will involve more complicated notations because of the quotient sheaf $h(C(\mathcal{A}_{\mathcal{P}})_{\mathcal{Q},X^I})$ and we will not do it here.
\end{rem}
\begin{rem}
  When $\mathrm{\mathbf{L}}=\mathrm{\mathbf{L}}_{\bar{1}}$ is purely odd, we get the trace map for the bc-system. It is well known that the space of conformal block (the 0-th chiral homology) of the bc-system is one-dimensional and can be identified with the fiber of the determinant line bundle $(\det R\Gamma_{\mathrm{DR}}(X,L_{>0}) )^{-1}$ (see \cite[pp35,Proposition 11.1]{1999Notes}), where
 $$
  L_{>0}=(\bigoplus_{\alpha\in \mathbb{Q}}\mathrm{\mathbf{L}}_{>0}^{\alpha}\otimes_{\mathbb{C}}\omega_X^{-\alpha+1})\otimes_{\mathcal{O}_X}\mathcal{D}_X,
$$
and $\mathrm{\mathbf{L}}_{>0}:=\bigoplus_{k>0}\mathrm{\mathbf{L}}_{2k-1}$. When $X$ is an elliptic curve, we see that our trace map gives an explicit quasi-isomorphism from the chiral chain complex to the fiber of the determinant line bundle.
\end{rem}
\section{Applications}
In this section, we discuss some applications of our main theorem. For any Lie subalgebra of the inner derivations of $V^{\beta\gamma-bc}$, we construct a cocycle in the Lie algebra cochain complex valued in the linear dual of the elliptic chiral chain complex. This construction can be viewed as a chiral algebra analogue of the construction in the algebraic index theory (see \cite{bressler2002riemann,feigin2005hochschild,gui2021geometry,nest1995algebraic} and reference therein),  where a cocycle in the Lie algebra cochain complex valued in the linear dual of the Hochschild complex of the Weyl algebra was constructed. This cocycle computes the algebraic index.  The chiral algebra version of this cocycle constructed here will be of the same importance to understanding the two-dimensional chiral algebraic index.

The second application is the construction of the trace map on the coset models. Coset models are realized in the theory of vertex operator algebra as the commutant algebras. Many important vertex operator algebras can be constructed in this way, like $\mathcal{W}-$algebras.  The commutant is a subalgebra inside a given vertex operator algebra and has no singular OPEs with another fixed subalgebra. Using the free field realization and our main theorem, one can construct a family of trace maps on the cosets.

\subsection{Trace map valued in Lie algebra cohomology}\label{TraceLieCohomology}

We first recall the definition of derivations of a vertex algebra $V$ following \cite{book1415599}.

\begin{defn}
  A derivation $d$ of parity $\gamma\in \mathbb{Z}/2\mathbb{Z}$ of a vertex algebra $V$ is an endomorphism of the space $V$ such that $dV_{\alpha}\subset V_{\alpha+\gamma}$ and
  $$
  d(a_{(n)}b)=(da)_{(n)}b+(-1)^{\alpha\cdot \gamma}a_{(n)}(db)\ \ \text{for all}\ a\in V_{\alpha},\ b\in V.
  $$
  We will denote by $\mathrm{Der}(V)$ the space of derivations of $V$. It is clear that $\mathrm{Der}(V)$ is a Lie superalgebra.

  By Borcherds identity, each homogeneous element $a\in V_{p(a)}$ gives rise to a derivation $a_{(0)}(-)$ of parity $p(a)$ of $V$. We denote this linear map by $\varphi: V\rightarrow \mathrm{Der}(V)$
  $$
  \varphi(a)(b)=a_{(0)}b.
  $$

 \end{defn}
By definition the map $\varphi: V\rightarrow\mathrm{Der}(V)$ factors through $V/L_{-1}V$
$$
V\rightarrow V/L_{-1}V\rightarrow \mathrm{Der}(V)
$$
and the second morphism is a Lie superalgebra morphism.

In the rest of this subsection, $V$ (resp. $\mathcal{A}$) means the free $\beta\gamma-bc$ vertex operator algebra (resp. chiral algebra) for a fixed symplectic super vector space $\mathbf{L}$. Throughout this subsection, we rescale the regularized integral such that
$$
\dashint_{X}dz\wedge d\bar{z}:=\frac{1}{2\pi i}\int_{X}dz\wedge d\bar{z}=\frac{1}{2\pi i}\int_{X}(-2i)dxdy=-\frac{\mathrm{Im}(\tau)}{\pi}.
$$
Then from Remark \ref{Rescaling}, the differential of the chiral chain complex becomes
$$
  d_{\mathrm{tot}}=d_{\mathrm{DR}}+\bar{\partial}+d_{\mathrm{ch}}+d_{\mathrm{Norm}}.
$$
We also introduce the notation $\underline{\mathrm{\mathbf{1}}}$ for the normalized constant section

\begin{equation}\label{ContantUnit}
\underline{\mathrm{\mathbf{1}}}:=-\frac{\pi}{\mathrm{Im}(\tau)}dz\cdot d\bar{z}\in \Gamma(X,\mathcal{A}_{\mathcal{Q}}).
\end{equation}
Then $\underline{\mathrm{\mathbf{1}}}$ can be viewed as a chiral chain in $\tilde{C}^{\mathrm{ch}}(X,\mathcal{A})_{\mathcal{Q}}$ and
$$
\mathrm{\mathbf{Tr}}(\underline{\mathrm{\mathbf{1}}})=\dashint_X \underline{\mathrm{\mathbf{1}}}=1.
$$

 Suppose $\mathfrak{g}$ is Lie superalgebra and $\rho:\mathfrak{g}\rightarrow V$  is a linear map such that the composition
  $$
   \mathfrak{g}\xrightarrow{\rho}V\xrightarrow{\varphi}\mathrm{Der}(V)
  $$
 $\varphi\circ\rho: \mathfrak{g} \rightarrow \mathrm{Der}(V)$ is a Lie superalgebra morphism.

  Now we introduce the trace map valued in Lie superalgebra cohomology. For detailed discussion of Lie superalgebra cohomology, see \cite{1975Cohomologies}.

\begin{defn}
 We view  $\xi=\xi_1\otimes\cdots \otimes\xi_m\in \mathfrak{g}^{\otimes n}$ as a constant section $$\xi=\rho(\xi_1)dz_1\shift\boxtimes\cdots\boxtimes\rho(\xi_m)dz_M\shift\in \Gamma(X^m,(\mathcal{A}\shift)^{\boxtimes m}(*\Delta_{\{1,\dots,m\}})).$$
  We will also use the identification
  $$
\rho(\xi_1)dz_1\shift\boxtimes\cdots\boxtimes\rho(\xi_m)dz_M\shift=  (-1)^{\bullet_\xi}\rho(\xi)d^mz[m],
  $$
  where
  $$
  \rho(\xi)=\rho(\xi_1)\otimes\cdots\otimes\rho(\xi_m),
  $$
  and
  $$
  \bullet_{\xi}=\sum_{i=1}^m(i-1)p(\xi_i).
  $$
  For any $a\in \Gamma(X^I,(\mathcal{A}\shift)^{\boxtimes I}(*\Delta_I)_\mathcal{Q})$, we define
  $$
\Trace_{\mathfrak{g}}(a)\{\xi_1\otimes\cdots \otimes\xi_m\}:=  \Trace(a\boxtimes \xi)\in \BVk.
  $$
\end{defn}
\begin{thm}\label{LieCohomologyChiral}
  The element $\Trace_{\mathfrak{g}}(-)\{-\}\in C_{\mathrm{Lie}}\big(\mathfrak{g},\mathrm{Hom}_{\mathbf{k}}(\tilde{C}^{\mathrm{ch}}(X,\mathcal{A})_{\mathrm{Q}},\mathbf{k})\big)\otimes_{\mathbf{k}}\BVk$ is a cocycle.
\end{thm}
\begin{proof}
  From the main theorem (Theorem \ref{mainThm}), we have
  \begin{align*}
   &\Trace((\bar{\partial}+d_{\mathrm{ch}})(a\boxtimes \xi))=-\hbar\Delta_{\mathrm{BV}}\Trace((a\boxtimes \xi)).
  \end{align*}
  Note that $\bar{\partial}(a\boxtimes \rho(\xi))=\bar{\partial}a\boxtimes \rho(\xi)$ and
  $$
\sum_{1\leq k\leq m}\sum_{i\in I}d_{\mathrm{ch},(I\sqcup \{1,\dots,m\}, I\sqcup_{i\sim k} \{1,\dots,m\})}(a\boxtimes \xi)=
  $$
  $$
\sum_{1\leq k\leq m}\sum_{i\in I}  (-1)^{p(\xi_i)\cdot (\sum\limits^{i-1}_{l=1}p(\xi_{l})+p(a)_i)+i-1}\big(\rho(\xi_k)_{(0)}\big)_ia\boxtimes \rho(\xi_1)\wedge\cdots\wedge\widehat{\rho(\xi_i)}\wedge\cdots\wedge\rho(\xi_m)\cdot(-1)^{\bullet_\xi}d^mz[m]
$$
$$
+d_\mathrm{DR}(-),
  $$
  here $\big((\xi_k)_{(0)}\big)_i$ means the action on the i-th component without the Koszul sign and the parity $p(a)_i=\mathop{\sum}^{|I|}\limits_{t=i}p(a_t)$ if we write $a$ as $a_1\boxtimes\cdots \boxtimes a_{|I|}$. This is exactly the Lie algebra action on the chiral chain complex. And we have
\begin{equation}\label{LieBracket}
  \begin{split}
      \sum_{1\leq k<l \leq m }  d_{\mathrm{ch},(I\sqcup \{1,\dots,m\}, I\sqcup \{1,\dots,m\}/k\sim l)}(a\boxtimes \xi)&=  \\
        \sum_{1\leq k<l\leq m}(-1)^{\bullet_{k,l}}a\boxtimes
\rho(\xi_k)_{(0)}\rho(\xi_l)\wedge\rho(\xi_1)\wedge\cdots \wedge \widehat{\rho(\xi_k)}\wedge\cdots\wedge &\widehat{\rho(\xi_l)}\wedge\cdots\wedge\rho(\xi_m)\cdot(-1)^{\bullet_\xi}d^mz[m]+d_{\mathrm{DR}}(-),
  \end{split}
\end{equation}
where
$$
\bullet_{k,l}=(p(\xi_k)+p(\xi_l))\cdot(p(\xi_1)+\cdots+p(\xi_{k-1}))+p(\xi_l)\cdot (p(\xi_{k+1}+\cdots+p(\xi_{l-1})))+k+l.
$$
Now by our assumption that $\varphi\circ \rho:\mathfrak{g}\rightarrow\mathrm{Der}(V)$ is a Lie superalgebra morphism and Borcherds identity, we have
$$
\varphi(\rho(\xi_k)_{(0)}\rho(\xi_l))=(\rho(\xi_k)_{(0)}\rho(\xi_l))_{(0)}=[\rho(\xi_k)_{(0)},\rho(\xi_l)_{(0)}]=[\varphi\circ \rho(\xi_k),\varphi\circ \rho(\xi_l)]=\varphi\circ\rho([\xi_k,\xi_l]).
$$
Let
$$
x=\rho(\xi_k)_{(0)}\rho(\xi_l)-\rho([\xi_k,\xi_l]).
$$

Then $x_{(0)}=0\in\mathrm{Der}(V)$ implies that $x=\mathrm{const}\cdot |0\rangle$ or $x\in L_{-1}V$ (see \cite[Lemma 5.3]{1999Chiral} and the references therein). We can replace $\rho(\xi_k)_{(0)}\rho(\xi_l)$ in (\ref{LieBracket}) by $\rho([\xi_k,\xi_l])$. To see this, suppose $x=\mathrm{const}\cdot |0\rangle$.  By the type reason, it will not contribute to the regularized integral. If $x\in L_{-1}V$ , from the fact that the regularized integral annihilates holomorphic total derivatives we get the same conclusion.

Using Lemma \ref{VanishOnDR}, we have
$$
\Trace_{\mathfrak{g}}((\bar{\partial}+d_{\mathrm{ch}})a)\{\xi\}+\partial_{\mathrm{Lie}}\Trace_{\mathfrak{g}}(a)\{\xi\}=\Trace((\bar{\partial}+d_{\mathrm{ch}})(a\boxtimes \xi))=-\hbar\Delta_{\mathrm{BV}}\Trace_{\mathfrak{g}}(a)\{ \xi\}.
$$
The proof is complete.
\end{proof}

\begin{rem}
  In particular, we can take $\mathfrak{g}=V/L_{-1}V$  (choose an arbitrary lift $\rho:\mathfrak{g}\rightarrow V$) and obtain a cocycle $\Trace_{V/L_{-1}V}(-)\{-\}\in C_{\mathrm{Lie}}\big(V/L_{-1}V,\mathrm{Hom}_{\mathbf{k}}(\tilde{C}^{\mathrm{ch}}(X,\mathcal{A})_{\mathrm{Q}},\mathbf{k})\big)\otimes_{\mathbf{k}}\BVk$.
Note that in general $\mathfrak{g}\rightarrow V/L_{-1}V$ is not a Lie superalgebra morphism as we only assume that $\varphi\circ \rho :\mathfrak{g}\rightarrow \mathrm{Der}(V)$ is. That is why our proof is slightly more complicated. In the later discussion on the Witten genus, we will only use the case when $\mathfrak{g}=V/L_{-1}V$.
\end{rem}
Now we want to compare our constructions to the result in \cite{Li:2016gcb}. We take $\mathfrak{g}=\mathbb{C}\text{\small{[-1]}}$ to be the one-dimensional abelian Lie superalgebra. Suppose we have an element $S\in V_{\bar{1}}$ satisfying $(S_{(0)})^2v=0, \forall v\in V$, then we have a morphism of Lie superalgebras
$$
\mathfrak{g}\rightarrow \mathrm{Der}(V),
$$
$$
c\cdot\text{\small{[-1]}}\mapsto c\cdot S_{(0)}.
$$

Also we have a complex $(V,S_{(0)})$ and its 0-th cohomology $H_S(V)$ is a vertex super algebra(see \cite[pp.96]{book:1415117}). We have
\begin{align*}
  \hbar\Delta_{\mathrm{BV}}\Trace_{\mathfrak{g}}(a)\{P^{\otimes}(S)\}& = \Trace_{\mathfrak{g}}(d_{\mathrm{tot}}a)\{P^{\otimes}(S)\},\\
   &
\end{align*}
for any $a\in \Gamma(X^I,(\mathcal{A}\shift)^{\boxtimes I}(*\Delta_I)_\mathcal{Q})$. Here $P(-)$ is a formal power series and $P^{\otimes}(-)$ means that we replace $\otimes$ with the commutative product
$$
P(x)=\sum_{k\geq 0}p_kx^k, \ P^{\otimes}(S)=\sum_{k\geq 0}p_kS^{\otimes n}
$$
 In summary, we have following proposition.

\begin{prop}\label{PS}
  For each power series $P(-)$, we have a chain map
  $$
  \Trace_{\mathfrak{g}}(-)\{P^{\otimes}(S)\}: (\tilde{C}^{\mathrm{ch}}(X,H_S(V)^r)_{Q},d_{\mathrm{tot}})\rightarrow (\BVk,\hbar\Delta_{\mathrm{BV}}).
  $$
\end{prop}
Now we take $a=\underline{\mathrm{\mathbf{1}}}$ (recall $\underline{\mathrm{\mathbf{1}}}$ is defined by (\ref{ContantUnit})) and $P(x)=e^{x/\hbar}$. In order to obtain a well-defined Laurent series, we assume that
$$
\Trace_{\mathfrak{g}}(\underline{\mathrm{\mathbf{1}}})\{S_i^{\otimes N}\}=0, \quad i\in \{0,1\},  \quad\text{for N large enough},
$$
here $S=\sum_{i\geq 0}S_ih^i$.

From Proposition \ref{PS}, we can get the following result in \cite{Li:2016gcb}.

\begin{prop}
The linear map $S(-): \mathbf{k}\rightarrow O_{\mathrm{BV},\mathbf{k}}$ defined by
$$
S(c)=c\cdot \Trace_{\mathfrak{g}}(\underline{\mathrm{\mathbf{1}}})\{e^{S/\hbar}\}
$$
satisfies the quantum master equation (see Definition \ref{GenQME} and Remark \ref{GenQMERem}). In other words,
$$
\hbar\Delta_{\mathrm{BV}}\Trace_{\mathfrak{g}}(\underline{\mathrm{\mathbf{1}}})\{e^{S/\hbar}\}=0.
$$
\end{prop}

\subsection{Trace map and the formal Witten genus }\label{TraceWittenGenus}
We now focus on the free $\beta\gamma$-system which is a special case of our previous discussion and prove that the formal Witten genus appears naturally in our construction. We first briefly recollect the data in the free $\beta\gamma$-system, following the presentations in \cite{gorbounov2016chiral,1999Chiral}.

Write $\mathrm{\mathbf{L}}=\mathrm{\mathbf{L}}_{\bar{0}}=\mathrm{\mathbf{L}}^0_{\bar{0}}\oplus\mathrm{\mathbf{L}}^1_{\bar{0}}$, here recall that the subscript $\bar{0}$ means that $\mathrm{\mathbf{L}}$ is purely even and the superscripts $0$ and $1$ are conformal weights. We can choose a basis of $\mathrm{\mathbf{L}}$ such that
$$
\mathrm{\mathbf{L}}=\mathrm{\mathbf{L}}^0_{\bar{0}}\oplus\mathrm{\mathbf{L}}^1_{\bar{0}}\cong \mathbb{C}^N\oplus \mathbb{C}^N=\mathrm{span}_{\mathbb{C}}\{\beta^i\}_{i=1,\dots,N}\oplus\mathrm{span}_{\mathbb{C}}\{\gamma^i\}_{i=1,\dots,N},
$$
and the even symplectic pairing is given by
$$
\langle (u_1,v_1),(u_2,v_2)\rangle=u_1v_2^t-u_2v_1^t,\quad (u_1,v_1),(u_2,v_2)\in \mathbb{C}^N\oplus \mathbb{C}^N,\quad v^t \ \text{is the transpose of }\ v.
$$
The corresponding vertex algebra (resp. chiral algebra) is still denoted by $V=V^{\beta\gamma}$(resp. $\mathcal{A}=\mathcal{A}^{\beta\gamma}$). Let $\hat{\mathbb{O}}_N$ denote the formal power series ring with $N$ variables
$$
\hat{\mathbb{O}}_N:=\mathbb{C}[[y_1,\cdots,y_N]],
$$
and write $\Omega^{\bullet}_{\fO}:=\oplus_{\bullet\geq 0}\Omega^{\bullet}_{\hat{\mathbb{O}}_N}$ for the formal differential forms
$$
\Omega^{i}_{\hat{\mathbb{O}}_N}:=\mathrm{span}_{\mathbb{C}}\{f(y)\cdot dy_{k_1}\wedge\cdots\wedge dy_{k_i}|1\leq k_1<\cdots< k_i\leq N, f(y)\in \hat{\mathbb{O}}_N\}.
$$
We introduce the Lie algebra $W_N$ of formal vector fields as follows
$$
W_N=\{\sum^N_{i=1}f_i\frac{\partial}{\partial y^i}: f_i\in \hat{\mathbb{O}}_N\}.
$$
The finite dimensional Lie algebra $\mathfrak{gl}_N$ can be viewed as a Lie subalgebra of $W_N$
$$
\mathfrak{gl}_N\ni \{a_{ij}\}_{1\leq i,j\leq N}\mapsto a_{ij}y^i\frac{\partial}{\partial y^j}\in W_N.
$$

We consider an extension of the Lie algebra $W_N$ \cite{gorbounov2016chiral,1999Chiral}.

\begin{defn}
  The 2-cocycle $c:W_N\times W_N \rightarrow  \Omega^1_{\fO}/d\fO$
  $$
  c(f\frac{\partial}{\partial y^i},g\frac{\partial}{\partial y^j}):=[d(\frac{\partial f}{\partial y^j})\cdot \frac{\partial g}{\partial y^i}]\in \Omega^1_{\fO}/d\fO,
  $$
  where $[-]$ is the equivalence class in $\Omega^1_{\fO}/d\fO$, defines an extension
  $$
  0\rightarrow \Omega^1_{\fO}/d\fO\rightarrow \widetilde{W}_N\rightarrow W_N\rightarrow 0.
  $$
\end{defn}
Next we define a Lie algebra morphism from $\widetilde{W}_N$ to $V/L_{-1}V$.   Write $\widetilde{W}_N=\Omega^1_{\fO}/d\fO\oplus W_N$ as a vector space. Define a linear map
  $$
  \rho_{W_N}: W_N\rightarrow V\cong \mathbb{C}[[\partial^k\gamma^i,\partial^k\beta^i]]_{k\geq 0, i=1,\dots,N}
  $$
  $$
  \rho_{W_N}(f(y)\cdot \frac{\partial}{\partial y^i})=\frac{\pi}{i\hbar}f(\gamma_0)\beta^i_{-1}|0\rangle=\frac{\pi}{i\hbar}f(\gamma)\beta^i,
  $$
  here $f(\gamma_0)$ means that we replace $y^i$ by $\gamma^i_0$. The induced map $\varphi\circ \rho_{W_N}: W_N\rightarrow \mathrm{Der}(V)$ is not a Lie algebra homomorphism, since
  \begin{align*}
&   [\varphi\circ \rho_{W_N}(f\frac{\partial}{\partial y^i}),\varphi\circ \rho_{W_N}(g\frac{\partial}{\partial y^j})]  \\
   &= [(\frac{\pi}{i\hbar}f(\gamma)\beta^i)_{(0)},(\frac{\pi}{i\hbar}g(\gamma)\beta^j)_{(0)}] \\
     & =\frac{\pi}{i\hbar}(f(\gamma)\frac{\partial}{\partial y^i} g(\gamma)\beta^j-g(\gamma)\frac{\partial}{\partial y^j} f(\gamma)\beta^i)_{(0)}+(\sum^N_{k=1}\frac{\partial}{{\partial}{\gamma^k}}(\frac{\partial f(\gamma)}{\partial \gamma^j})\partial\gamma^k\cdot \frac{\partial g(\gamma)}{\partial \gamma^i})_{(0)}\\
     &=\varphi\circ \rho_{W_N}([f\frac{\partial}{\partial y^i}, g\frac{\partial}{\partial y^j}])+\varphi\circ   \rho_{\Omega^1_{\fO}}(d(\frac{\partial f}{\partial y^j})\cdot \frac{\partial g}{\partial y^i}).
  \end{align*}

  Here $ \rho_{\Omega^1_{\fO}}$ is defined by
  $$
  \rho_{\Omega^1_{\fO}}:\Omega^1_{\fO}\rightarrow V  \cong  \mathbb{C}[[\partial^k\gamma^i,\partial^k\beta^i]]_{k\geq 0, i=1,\dots,N}
  $$
  $$
  \rho_{\Omega^1_{\fO}}(f(y)dy^i)=f(\gamma_0)\gamma^i_{-1}|0\rangle=f(\gamma)\partial\gamma^i.
  $$
  Note that the map $\varphi\circ\rho_{\Omega^1_{\fO}}:\Omega^1_{\fO}\rightarrow \mathrm{Der}(V)$ factors through a sequence of maps
   $$
  \Omega^1_{\fO}\rightarrow \Omega^1_{\fO}/d\fO\rightarrow V/L_{-1}V\rightarrow \mathrm{Der}(V),
   $$
    since
  $$
  \varphi\circ\rho_{\Omega^1_{\fO}}(df(y))=L_{-1}(f(\gamma_0)|0\rangle).
  $$
We denote by $\overline{\rho}_{W_N}$ the map $W_N\rightarrow V/L_{-1}V$ induced by $\rho_{W_N}$. Similarly, set $\overline{\rho}_{\Omega^1_{\fO}/d\fO}$ the map $\Omega^1_{\fO}/d\fO\rightarrow V/L_{-1}V$ induced by $  \rho_{\Omega^1_{\fO}}:\Omega^1_{\fO}\rightarrow V$.

In summary, we have the following definition.

\begin{defn}
Define the Lie algebra morphism $\overline{\rho}_{\widetilde{W}_{N}}:\widetilde{W}_N\rightarrow V/L_{-1}V$ by
$$
\overline{\rho}_{\widetilde{W}_{N}}:=\overline{\rho}_{\Omega^1_{\fO}/d\fO}\oplus\overline{\rho}_{W_N}: \Omega^1_{\fO}/d\fO\oplus W_N\cong\widetilde{W}_N\rightarrow V/L_{-1}V.
$$
We choose an arbitrary lift ${\rho}_{\widetilde{W}_{N}}:\widetilde{W}_N\rightarrow V$ of $\overline{\rho}_{\widetilde{W}_{N}}$. It automatically satisfies that $\varphi\circ {\rho}_{\widetilde{W}_{N}}:\widetilde{W}_N\rightarrow \mathrm{Der}(V)$ is a Lie algebra morphism.
\end{defn}

Applying Theorem \ref{LieCohomologyChiral}, we get the following proposition.
\begin{prop}\label{BVFeynman}
  We identify the subalgebra $\mathbb{C}[[\gamma^i,\gamma^{i\dagger}]]\subset O_{\mathrm{BV}}(\mathrm{\mathbf{L}})$ with $\Omega_{\fO}=\oplus_{\bullet\geq 0}\Omega^{\bullet}_{\fO}$ by
  $$
  \gamma^i\mapsto y^i,\quad \gamma^{i\dagger}\mapsto dy^i.
  $$
  Write
  $$
  \Trace_{\widetilde{W}_N}(\underline{\mathrm{\mathbf{1}}})\{-\}=\sum_{k\geq 0, l\in \mathbb{Z}}  \Trace_{(k),l}\{-\}\hbar^l,\quad \Trace_{(k),l}\in   C^k_{\mathrm{Lie}}\big(\widetilde{W}_N,\mathbb{C}\big)\otimes_{\mathbb{C}}O_{\mathrm{BV}}(\mathrm{\mathbf{L}}).
  $$
  Here $\mathrm{\mathbf{1}}$ is defined in (\ref{ContantUnit}). Then $  \Trace_{(k),l}=0$ for $l>0, k\geq 0$ and the chain
  $$
  \Trace_{0}:=\sum_{k\geq 0}  \Trace_{(k),0}\{-\}\in C_{\mathrm{Lie}}\big(\widetilde{W}_N,\mathbb{C}\big)\otimes_{\mathbb{C}}\Omega^{\bullet}_{\fO}
  $$
  is a cocycle in $C_{\mathrm{Lie}}\big(\widetilde{W}_N,\Omega^{\bullet}_{\fO})$ if we equip $\Omega^{\bullet}_{\fO}$ with the action of $\widetilde{W}_N$ induced by the $W_N$-action as Lie derivative.

\end{prop}
\begin{proof}
To simplify the notation, we write $\rho$ instead of $\rho_{\widetilde{W}_{N}}$. By definition
$$
\sum_{l\in\mathbb{Z}}\hbar^l\Trace_{(k),l}\{\xi_1\wedge\cdots\wedge\xi_k\}=\dashint_{X^k} e^{\hbar \mathfrak{P}+\hbar\mathfrak{Q}+D}(-1)^{\bullet_\xi}\rho(\xi_1\wedge\cdots\wedge\xi_k)d^mz[m],
$$

We can assume that
$$
\xi_i=f_i(y)\frac{\partial}{\partial y^{l_i}}+[g_i(y)dy^{m_i}],\quad f_i(y)\frac{\partial}{\partial y^{l_i}}\in W_N,\ [g_i(y)dy^{m_i}]\in \Omega^1_{\fO}/d\fO \quad i=1,\dots,k.
$$
Thus
$$
\rho(\xi_i)=\frac{\pi}{i\hbar}f_i(\gamma)\beta^{l_i}+g_i(\gamma)\partial\gamma^{m_i}\quad \mod (L_{-1}V).
$$
We simply take $\rho(\xi_i)=\frac{\pi}{i\hbar}f_i(\gamma)\beta^{l_i}+g_i(\gamma)\partial\gamma^{m_i}$ since the element in $L_{-1}V$ will not affect the result. And the term $g_i(\gamma)\partial\gamma^{m_i}$ will also not contribute to the result by Feynman diagram arguments (see Appendix \ref{FeynmanCompute}).

Now the fact that $\rho(\xi_i)$ is linear in $\beta$ implies that
$$
  \Trace_{(k),l}=0\quad\text{for}\  l>0, k\geq 0.
$$
We have (QME)
$$
\Delta_{\mathrm{BV}}\Trace_{(k+1),-1}+\partial_{\mathrm{Lie}}\Trace_{(k),0}=0.
$$
One can use Feynman diagrams (see Appendix \ref{FeynmanCompute}) to show that
$$
  \Trace_{0}\in C_{\mathrm{Lie}}\big(\widetilde{W}_N,\mathbb{C}\big)\otimes_{\mathbb{C}}\Omega^{\bullet}_{\fO}
  $$
  and
$$
\Trace_{(k+1),-1}=\sum_{j=1}^{k+1} (-1)^{\bullet_j} \dashint_{X}dz_j\frac{\pi}{i}(\sum_{s=1}^{N}\frac{\partial f_j(\gamma)}{\partial \gamma^s}\gamma^{s\dagger}\beta^{l_j}d\bar{z}_j+f_j(\gamma)\beta^{l_j\dagger}d\bar{z}_j)\wedge \mathrm{\mathbf{Tr}}_{(k),0}\{\xi_1\wedge\cdots\wedge \widehat{\xi}_j\wedge\cdots\wedge \xi_{k+1}\},
$$
where
$$
\bullet_j= {p(\xi_j)\cdot (p(\xi_{j-1})+\cdots+p(\xi_1))+j-1}=j-1.
$$
Then the left-hand side of the QME is equal to
  $$
\sum_{j=1}^{k+1} (-1)^{\bullet_j} \Delta_{\mathrm{BV}}(\dashint_{X}dz_j\frac{\pi}{i}(\sum_{s=1}^{N}\frac{\partial f_j(\gamma)}{\partial \gamma^s}\gamma^{s\dagger}\beta^{l_j}d\bar{z}_j+f_j(\gamma)\beta^{l_j\dagger}d\bar{z}_j)\wedge \mathrm{\mathbf{Tr}}_{(k),0}\{\xi_1\wedge\cdots\wedge \widehat{\xi}_j\wedge\cdots\wedge \xi_{k+1}\}).
  $$

  Write
  $$
\Omega_j=\mathrm{\mathbf{Tr}}_{(k),0}(\xi_1\wedge\cdots\wedge \widehat{\xi}_j\wedge\cdots\wedge \xi_{k+1})).
  $$
  Then
  \begin{align*}
&  \Delta_{\mathrm{BV}}(\dashint_{X}dz_j\frac{\pi}{i}(\sum_{s=1}^{N}\frac{\partial f_j(\gamma)}{\partial \gamma^s}\gamma^{s\dagger}\beta^{l_j}d\bar{z}_j+f_j(\gamma)\beta^{l_j\dagger}d\bar{z}_j)\wedge\Omega_j)\\
      &=\Delta_{\mathrm{BV}}(\big(-\dashint_{X}dz_j\wedge d\bar{z}_j\frac{\pi}{i}(\sum_{s=1}^N\frac{\partial f_j(\gamma)}{\partial \gamma^s}\gamma^{s\dagger}\beta^{l_j}+f_j(\gamma)\beta^{l_j\dagger})\big)\wedge \Omega_j) \\
      &=\frac{i}{\mathrm{Im}(\tau)}(\sum_{t=1}^{N}\frac{\partial}{\partial\beta^t}\frac{\partial}{\partial\gamma^{t\dagger}}-\frac{\partial}{\partial\gamma^t}\frac{\partial}{\partial\beta^{t\dagger}})(\frac{-\mathrm{Im}(\tau)}{\pi}\cdot\frac{\pi}{i}(\sum_{s=1}^N\frac{\partial f_j(\gamma)}{\partial \gamma^s}\gamma^{s\dagger}\beta^{l_j}+f_j(\gamma)\beta^{l_j\dagger})\wedge \Omega_j)\\
      &=-(\sum_{t=1}^{N}\frac{\partial}{\partial\beta^t}(\sum_{s=1}^N\frac{\partial f_j(\gamma)}{\partial \gamma^s}\gamma^{s\dagger}\beta^{l_j})\wedge (-1)^{p(\gamma^{t\dagger})p(\gamma^{s\dagger})}\frac{\partial}{\partial\gamma^{t\dagger}} \Omega_j-\sum_{t=1}^{N}\frac{\partial}{\partial\beta^{t\dagger}}(\sum_{s=1}^Nf_j(\gamma)\beta^{l_j\dagger})\wedge (-1)^{p(\beta^{t\dagger})p(\gamma^{t})} \frac{\partial}{\partial\gamma^t}  \Omega_j)\\
      &= (df_j(\gamma)\cdot \iota_{\frac{\partial}{\partial y^{l_j}}}+f_j(\gamma)\mathcal{L}_{\frac{\partial}{\partial y^{l_j}}}))\Omega_j \quad (\text{we use}\ p(\gamma^{t\dagger})=\bar{1},p(\gamma^{t})=\bar{0})\\
     & = d(f_j(\gamma)\cdot \iota_{\frac{\partial}{\partial y^{l_j}}}\Omega_j)+f_j(\gamma)\mathcal{L}_{\frac{\partial}{\partial y^{l_j}}}\Omega_j-f_j(\gamma)\cdot d(\iota_{\frac{\partial}{\partial y^{l_j}}}\Omega_j)\\
     & = d(f_j(\gamma)\cdot \iota_{\frac{\partial}{\partial y^{l_j}}}\Omega_j)+f_j(\gamma)\cdot \iota_{\frac{\partial}{\partial y^{l_j}}}d\Omega_j\\
     &=\mathcal{L}_{f_j(y)\frac{\partial}{\partial y^{l_j}}}\Omega_j.
  \end{align*}
  Here we use the fact that $\frac{\partial}{\partial\beta^t}\Omega_j=\frac{\partial}{\partial\beta^{t\dagger}}\Omega_j=0$, since $\Omega_j\in \Omega_{\fO}\cong \mathbb{C}[[\gamma^t,\gamma^{t\dagger}]]_{t=1,\dots,N}$.

  Finally, we get
  $$
  \partial_{\mathrm{Lie}}\Trace_{(k),0}\{-\}+\sum_{j=1}^{k+1}(-1)^{\bullet_j}\mathcal{L}_{f_j(y)\frac{\partial}{\partial y^{l_j}}}\Omega_j=0.
  $$
  The proof is complete.
\end{proof}
\begin{rem}
  The cocycle property also can be derived by direct computation (see the following proposition). Here we emphasize that the cocycle property is related to the quantum master equation and can be generalized to more complicated Lie algebra other than $\tilde{W}_N$.
\end{rem}

We can compute the cocycle $\mathrm{\mathbf{Tr}}_{0}$ explicitly. In fact, it is equal to the formal Witten genus times a simple factor. Let us briefly review the formal Witten genus in Lie algebra cohomology.  We first define the formal Atiyah class as follows
$$
\mathrm{At}\in\mathrm{C}_{\mathrm{Lie}}^1(W_N, \mathrm{GL}_N;\Omega^1_{\fO}\otimes_{\mathbb{C}}\mathrm{End}(\mathbb{C}^N))
$$
$$
\mathrm{At}(f(y)\frac{\partial}{\partial y^i})=d(\partial_kf(y))\otimes (\frac{\partial}{\partial y^i}\otimes y^k)\in\Omega^1_{\fO}\otimes_{\mathbb{C}}\mathrm{End}(\mathbb{C}^N).
$$

The formal Chern character is defined by
$$
\mathrm{ch}_{k}:=\frac{1}{(-2\pi i)^kk!}\mathrm{tr}(\mathrm{At}^k)\in \mathrm{C}_{\mathrm{Lie}}^k(W_N, \mathrm{GL}_N;\Omega^k_{\fO}), \quad k\geq 0.
$$
The cochain $\mathrm{ch}_{k}$ is indeed a cocylce.
\begin{defn}
  The formal Witten genus is defined by
  $$
  \log \mathrm{Wit}_{N}(\tau):=\sum_{k\geq 2}\frac{(2k-1)!}{(2\pi i)^{2k}}E_{2k}(\tau)\mathrm{ch}_{2k}.
  $$

\end{defn}
Here $E_{2k}$ are the Eisenstein series that are defined by
$$
E_{2k}(\tau)=\frac{1}{2\zeta(2k)}\sum_{(m,n)\in\mathbb{Z}^2-\{(0,0)\}}\frac{1}{(m\tau+n)^{2k}},\quad k\geq 2,
$$
$$
E_2(\tau)=\frac{1}{2\zeta(2)}(\sum_{n\neq 0}\frac{1}{n^2}\sum_{m\neq 0}\sum_{n\in \mathbb{Z}}\frac{1}{(m\tau+n)^2}),
$$
where $\zeta(2k),k\geq 1$ are the zeta-values.

We also use $\widehat{E}_2(\tau)=E_2(\tau)-\frac{3}{\pi}\frac{1}{\mathrm{Im}(\tau)}$ to denote the modular completion of the 2nd Eisenstein series.

We introduce
$$
\theta\in C^1_{\mathrm{Lie}}\big(\widetilde{W}_N,\mathbb{C}\big)\otimes_{\mathbb{C}}O_{\mathrm{BV}}(\mathrm{\mathbf{L}})
$$
$$
\theta (f(y)\frac{\partial}{\partial y^l}+[g(y)dy^m])=\sum_{k=1}^{N}\frac{\partial f(\gamma)}{\partial \gamma^k}\gamma^{k\dagger}\beta^l+f(\gamma)\beta^{l\dagger}.
$$
With this notation, we have the following theorem.
\begin{thm}\label{WittenGenus}
We have
$$
  \Trace_{\widetilde{W}_N}(\underline{\mathrm{\mathbf{1}}})\{-\}=e^{\frac{\pi}{i\hbar}\theta}\cdot e^{p^*\log \mathrm{Wit}_N(\tau)+\frac{1}{32\pi^4}\widehat{E}_2\cdot p^*\mathrm{tr}(\mathrm{At}^2)}\in C^k_{\mathrm{Lie}}\big(\widetilde{W}_N,\mathbb{C}\big)\otimes_{\mathbb{C}}\BVk.
$$
In particular, $\Trace_{0}$ is cohomologous to $p^*\mathrm{Wit}_N(\tau)$ when we viewed it as an element in the relative cochain complex $\mathrm{C}_{\mathrm{Lie}}^{\bullet}(\widetilde{W}_N, \mathrm{GL}_N;\Omega_{\fO})$, where
$$
p^*:\mathrm{C}_{\mathrm{Lie}}^k(W_N, \mathrm{GL}_N;\Omega^k_{\fO})\rightarrow\mathrm{C}_{\mathrm{Lie}}^k(\widetilde{W}_N, \mathrm{GL}_N;\Omega^k_{\fO})\subset C^k_{\mathrm{Lie}}\big(\widetilde{W}_N,\mathbb{C}\big)\otimes_{\mathbb{C}}\BVk
$$
 is the pull back map induced by $\widetilde{W}_N\rightarrow W_N$.
\end{thm}
\begin{proof}
  We only need to notice that in this computation, the singular differential forms that we need to integrate are of the following form
  $$
d^nzd^n\bar{z}\cdot  P(z_1,z_2)P(z_2,z_3)\cdots P(z_{n-1},z_n)P(z_n,z_1).
   $$
This can be written as a logarithmic form
$$
P(z_1,z_2)(dz_1-dz_2)\cdot P(z_2,z_3)(dz_2-dz_3)\cdots P(z_n,z_1)(dz_n-dz_1)\cdot d^n\bar{z},
$$
up to a scalar factor.

 The above differential form is integrable and the regularized integral is equal to the usual integral (see Appendix \ref{ReInt} for the precise definition of the regularized integral).  Then the proof is the same as in \cite{2011WittenGenus,gorbounov2016chiral}. See Appendix \ref{FeynmanCompute} for the relevant Feynman diagram computations.
\end{proof}
\begin{rem}
  Our construction is motivated by the Gelfand-Fuks construction on Lie algebra cohomology of formal vector fields \cite{Gel_fand_1970}. This type of construction appears in mathematical formulations of non-linear sigma models, see \cite{feigin1989riemann,kapranov1999rozansky,kontsevich1999rozansky} and later developments \cite{gorbounov2016chiral,gui2021geometry}. In particular, \cite{gorbounov2016chiral} shows that the one-loop part of the partition function is the formal Witten genus (see \cite{2011WittenGenus,costello2010geometric} for the original geometric approach) in the relative Lie algebra cohomology of formal vector fields.
\end{rem}

\begin{rem}
  In \cite{BVQandindex}, the authors study the topological quantum mechanics in BV formalism and reduce the algebraic index theorem to the one-loop computation in Feynman diagrams.   In \cite{gui2021geometry}, the universal algebraic index is computed using the Gelfand-Fuks construction and the same $S^1$-equivariant localization method used in \cite{BVQandindex}. We expect that our construction here provides a two-dimensional analogue of \cite{gui2021geometry} and sheds some new light on the 2d chiral algebraic index (which can be thought of as a certain algebraic index on the loop space).
\end{rem}

\begin{rem}
  The computation in the above Proposition is automatically one-loop because $\mathfrak{g}=W_N$ is a very special Lie algebra. If we take a general Lie subalgebra of the inner derivation of $V$, the computations may involve higher loops and can be in principle very complicated. In other words, the formal Witten genus can be viewed as a special example of general (unknown) chiral indices.
\end{rem}

\subsection{Trace map for coset models}
In this subsection, we do not rescale the regularized integral. The differential of the chiral chain complex is the usual one
$$
  d_{\mathrm{tot}}=d_{\mathrm{DR}}+\bar{\partial}+2\pi id_{\mathrm{ch}}+d_{\mathrm{Norm}}.
$$

Let $V$ be a vertex algebra, a vertex subalgebra of $V$ is a subspace $A$ of $V$ containing $|0\rangle$ such that
$$
a_{(n)}A\subset A\ \text{for all}\ a\in A\ \text{and}\ n\in\mathbb{Z}.
$$

Now let $A$ be a vertex subalgebra of a vertex algebra $V$. The commutant of $A$ in $V$, denoted by $\mathrm{Com}(A,V)$, is the vertex subalgebra consisting of all elements $v\in V$ such that $a_{(n)}v=0$ for $n\geq 0$ and all $a\in A$. This construction is standard in the theory of vertex algebra, see \cite{book:1415117}.

Suppose we have a free field realization of $V$, that is, there is a morphism of chiral algebras
$$
\rho:\mathcal{V}^r\rightarrow \mathcal{A}^{\beta\gamma-bc}.
$$
$\rho$ induces a chain map
$$
\rho^{\mathrm{ch}}: \tilde{C}^{\mathrm{ch}}(X,\mathcal{V}^r)_{\mathcal{Q}}\rightarrow \tilde{C}^{\mathrm{ch}}(X,\mathcal{A}^{\beta\gamma-bc})_{\mathcal{Q}}.
$$
Then we get the induced trace map $\mathrm{\mathbf{Tr}}^{\mathrm{BV}}_{\rho^{\mathrm{ch}}}=\mathrm{\mathbf{Tr}}^{\mathrm{BV}}_{\mathcal{A}^{\beta\gamma-bc}}\circ\rho^{\mathrm{ch}}$ for chiral algebra $\mathcal{V}^r$.  Let $\alpha\in \Gamma(X^I,(\mathcal{V}^r\shift)^{\boxtimes T}(*\Delta_T)_\mathcal{Q})\subset \Gamma(X^I, \mathrm{Norm}(\equiQ(I),F)^{0})$  such that
$$
(\bar{\partial}+2\pi i d_{\mathrm{ch}})\alpha\in \mathrm{Im}(d_{\mathrm{DR}}).
$$
Then we can construct a trace map for chiral algebra $\mathrm{Com}(A,V)$ associated to $\alpha$ as follows. Define
$$
\mathrm{\mathbf{Tr}}_{\alpha}(\eta):=\mathrm{\mathbf{Tr}}^{\mathrm{BV}}_{\rho^{\mathrm{ch}}}(\alpha\boxtimes \eta), \eta\in \tilde{C}^{\mathrm{ch}}(X,(\mathrm{Com}(A,V))^r)_{\mathcal{Q}}.
$$
In the rest of this subsection, we denote the differential of the chiral complex $\tilde{C}^{\mathrm{ch}}(X,(\mathrm{Com}(A,V))^r)_{\mathcal{Q}}$ (resp. $\tilde{C}^{\mathrm{ch}}(X,\mathcal{V}^r)_{\mathcal{Q}}$) by
$$
d^{\mathrm{Com}(A,V)}_{\mathrm{tot}}=\bar{\partial}+d_{\mathrm{DR}}+2\pi id_{\mathrm{ch}}^{\mathrm{Com}(A,V)}+d_{\mathrm{Norm}}\quad (\text{resp.}\ d_{\mathrm{tot}}^{V}=\bar{\partial}+d_{\mathrm{DR}}+2\pi id_{\mathrm{ch}}^{V}+d_{\mathrm{Norm}}).
$$
\begin{thm}
  $\mathrm{\mathbf{Tr}}_{\alpha}$ is a trace map on $\tilde{C}^{\mathrm{ch}}(X,(\mathrm{Com}(A,V))^r)_{\mathcal{Q}}$, that is,
  $$
  \mathrm{\mathbf{Tr}}_{\alpha} : (\tilde{C}^{\mathrm{ch}}(X,(\mathrm{Com}(A,V))^r)_{\mathcal{Q}},d^{\mathrm{Com}(A,V)}_{\mathrm{tot}})\rightarrow (\mathbf{k},0)
$$
is a chain map. Here we omit the degree shift in $ (\mathbf{k},0)$.
\end{thm}
\begin{proof}
We will prove that
  $$
  \mathrm{\mathbf{Tr}}^{\mathrm{BV}}_{\rho^{\mathrm{ch}}}(\alpha\boxtimes d_{\mathrm{tot}}^{\mathrm{Com}(A,V)}\eta)=  \mathrm{\mathbf{Tr}}^{\mathrm{BV}}_{\rho^{\mathrm{ch}}}(d_{\mathrm{tot}}^{V}(\alpha\boxtimes \eta))=0.
  $$
  And we have following obvious identities
  $$
  \bar{\partial}(\alpha\boxtimes \eta)=\bar{\partial}\alpha\boxtimes \eta+\alpha\boxtimes \bar{\partial} \eta,
  $$
  $$
  \mathrm{\mathbf{Tr}}^{\mathrm{BV}}_{\rho^{\mathrm{ch}}}(d_{\mathrm{DR}}(\alpha\boxtimes \eta))= \mathrm{\mathbf{Tr}}^{\mathrm{BV}}_{\rho^{\mathrm{ch}}}(d_{\mathrm{DR}}\alpha\boxtimes \eta)= \mathrm{\mathbf{Tr}}^{\mathrm{BV}}_{\rho^{\mathrm{ch}}}(\alpha\boxtimes d_{\mathrm{DR}}\eta)=0,
  $$
  $$
    \mathrm{\mathbf{Tr}}^{\mathrm{BV}}_{\rho^{\mathrm{ch}}}(d_{\mathrm{Norm}}(\alpha\boxtimes \eta))=\mathrm{\mathbf{Tr}}^{\mathrm{BV}}_{\rho^{\mathrm{ch}}}(\alpha\boxtimes d_{\mathrm{Norm}}\eta).
  $$
  Note that the defining property of the commutant implies
  $$
  d_{\mathrm{ch}}^V(\alpha\boxtimes \eta)= (d_{\mathrm{ch}}^{V}\alpha)\boxtimes \eta+\alpha\boxtimes d_{\mathrm{ch}}^{\mathrm{Com}(A,V)}\eta.
  $$

Then
\begin{align*}
    \mathrm{\mathbf{Tr}}^{\mathrm{BV}}_{\rho^{\mathrm{ch}}}(\alpha\boxtimes d_{\mathrm{tot}}^{\mathrm{Com}(A,V)}\eta) &=  \mathrm{\mathbf{Tr}}^{\mathrm{BV}}_{\rho^{\mathrm{ch}}}(\alpha\boxtimes (\bar{\partial}+d_{\mathrm{DR}}+2\pi i d_{\mathrm{ch}}^{\mathrm{Com}(A,V)}+d_{\mathrm{Norm}})\eta)  \\
   & = \mathrm{\mathbf{Tr}}^{\mathrm{BV}}_{\rho^{\mathrm{ch}}}(d_{\mathrm{tot}}^{V}(\alpha\boxtimes \eta))-\mathrm{\mathbf{Tr}}^{\mathrm{BV}}_{\rho^{\mathrm{ch}}}((\bar{\partial}+2\pi id_{\mathrm{ch}}^V)\alpha\boxtimes \eta)\\
   &=-\mathrm{\mathbf{Tr}}^{\mathrm{BV}}_{\rho^{\mathrm{ch}}}( (d_{\mathrm{DR}}(-))\boxtimes\eta)=0,
\end{align*}
in the last step we use the assumption that $(\bar{\partial}+2\pi id_{\mathrm{ch}}^{V})\alpha\in \mathrm{Im}(d_{\mathrm{DR}}).$
\end{proof}

\begin{rem}
  Since $\mathcal{W}$-algebras can be constructed as coset models, our construction provides a trace map on the elliptic chiral homology of $\mathcal{W}$-algebras.
\end{rem}

\section{Appendix}
\subsection{The Wick theorem}\label{Wichtheorem}
In this appendix, we review the Wick theorem and reformulate it to fit our context.  For a field $A(z)=\sum_{n\in\mathbb{Z}}A_{(n)}z^{-n-1}$ in a vertex algebra , we write
$$
  A(z)_-=\sum_{n\geq 0}A_{(n)}z^{-n-1}, A(z)_+=\sum_{n<0}A_{(n)}z^{-n-1}.
$$

Here we recall a version of Wick theorem in \cite{book1415599}.

\begin{thm}(Wick theorem)
   Let $A^{1}(z), \ldots, A^{M}(z)$ and $B^{1}(z), \ldots, B^{N}(z)$ be
two collections of fields such that the following properties hold:

(i) $\left[\left[A^{i}(z)_{-}, B^{j}(w)\right], C^{k}(z)_{\pm}\right]=0$ for all $i, j, k,$ and $C=A$ or $B$

(ii) $\left[A^{i}(z)_{\pm}, B^{j}(w)_{\pm}\right]=0$ for all $i$ and $j$.

Let $\left[A^{i} B^{j}\right]=\left[A^{i}(z)_{-}, B^{j}(w)\right]$ denote the "contraction" of $A^{i}(z)$ and $B^{j}(w) .$ Then one
has:
$$
\begin{array}{ll}
: & A^{1}(z) \cdots A^{M}(z):: B^{1}(w) \cdots B^{N}(w):=\sum\limits_{s=0}^{\min (M, N)} \sum\limits_{i_{1}<\cdots<i_{s} \atop j_{1} \neq \cdots \neq j_{s}}^{N} \\
& \left(\pm\left[A^{i_{1}} B^{j_{1}}\right] \cdots\left[A^{i_{s}} B^{j_{s}}\right]: A^{1}(z) \cdots A^{M}(z) B^{1}(w) \cdots B^{N}(w):\left(i_{1}, \ldots, i_{s} ; j_{1}, \ldots, j_{s}\right)\right)
\end{array}
$$
where the subscript $\left(i_{1} \cdots i_{s} ; j_{1} \cdots j_{s}\right)$ means that the fields $A^{i_{1}}(z), \ldots, A^{i_{s}}(z),$
$B^{j_{1}}(w), \ldots, B^{j_{s}}(w)$ are removed, and the sign $\pm$ is obtained by the usual Koszul sign rule: each permutation of the adjacent odd fields changes the sign.
\end{thm}

Since we have the identification $V^{\beta\gamma-bc}\simeq \mathbb{C}[\partial^ka^s]$, any two elements $v_1=\partial^{k_1}a^{s_1}$ and $v_2=\partial^{k_2}a^{s_2}$ satisfy the conditions in the Wick theorem. We can reformulate the Wick theorem using the chiral operation $\mu$ as follows.
\begin{thm}\label{WickThm}
For $V^{\beta\gamma-bc}$, we have the Wick theorem
  $$
\mu(F(z_1,z_2)v_1dz_1\otimes v_2dz_2)
$$
$$
=\mathrm{\mathbf{Mult}}_{1\rightarrow 2}  \mathrm{Res}_{z_1\rightarrow z_2}F(z_1,z_2)\cdot e^{(z_1-z_2)\otimes \vec{\partial}_{z_1}}
  \cdot e^{(z_1-z_2)(L_{-1})_1} \cdot e^{\hbar \mathcal{P}^{\mathrm{Sing}}_{12}(1\rightarrow 2)}v_1dz_1\otimes v_2dz_2,
  $$
  for any local section $F(z_1,z_2)dz_1\otimes dz_2$ of $\omega_{X^2}(*\Delta_{\{1,2\}})_{\mathcal{Q}}$. For the definition of this generalized residue operation, see Definition \ref{genResidue}.
\end{thm}
\subsection{Proof of the identities (3.3)-(3.8)}\label{Identities}
Here we prove the identities appear in the proof of Proposition \ref{WDChirall}. To simplify notation, we use $D^p_k$ for $\frac{\partial}{\partial(\partial^k a^p)}$.

Note that (\ref{PMultCommute}), (\ref{QMultCommute}) and (\ref{QjCommute}) follow directly from the definition and (\ref{QiCommute}) is already proved in Lemma \ref{WDmodule}.

For (\ref{PLCommute}), we use a version of the Baker-Campbell-Hausdorff  formula
  $$
  e^{X}e^{Y}=e^{Y+[X,Y]+\frac{1}{2!}[X,[X,Y]]+\frac{1}{3!}[X,[X,[X,Y]]]+\cdots}e^{X}.
  $$
  It can be proved using
  $$
  e^XYe^{-X}=\sum_{n=0}^{\infty}\frac{[X,\cdots[X,[X,Y]]\cdots]}{n!}, e^{e^XYe^{-X}}=e^Xe^Ye^{-X}.
  $$
  To simplify notation, we write
  $$
  (z_i-z_j)^{(r)}=\frac{1}{r!}(z_i-z_j)^r,
  $$
  $$
  [[X]^s,Y]=\underbrace{[X,\cdots[X,[X}_{s},Y]]\cdots],
  $$
    $$
  [[X]^{(s)},Y]=\frac{1}{s!}\underbrace{[X,\cdots[X,[X}_{s},Y]]\cdots].
  $$
  Then we have

\begin{align*}
&[(z_i-z_j)(L_{-1})_i,\mathcal{P}_{i\bullet}(i\rightarrow j)]=[(z_i-z_j)(L_{-1})_i,\sum_{r\geq 0}\sum_{k,l\geq 0}\sum_{p,q}P_{j\bullet}(k+r,l)(z_i-z_j)^{(r)} \omega_{pq}(D^p_k)_i\otimes (D^q_l)_{\bullet}]\\
&=[(z_i-z_j)\sum_t\sum_{m} \partial^{t+1}a^m\frac{\partial}{\partial(\partial^t a^m)},\sum_{r\geq 0}\sum_{k,l\geq 0}\sum_{p,q}P_{j\bullet}(k+r,l)(z_i-z_j)^{(r)} \omega_{pq}(\frac{\partial}{\partial(\partial^ka^p)})_i\otimes (D^q_l)_{\bullet}]\\
&=\sum_{r\geq 0}\sum_{k,l\geq 0}\sum_{p,q}P_{j\bullet}(k+r,l)(z_i-z_j)^{(r)}\cdot (z_j-z_i) \omega_{pq}(D^p_{k-1})_i\otimes (D^q_l)_{\bullet}.
\end{align*}
By induction we get
\begin{align*}
  &[[(z_i-z_j)(L_{-1})_i]^{(s)},\mathcal{P}_{i\bullet}(i\rightarrow j)]\\
  &=[[(z_i-z_j)(L_{-1})_i]^{(s)},\sum_{r\geq 0}\sum_{k,l\geq 0}\sum_{p,q}P_{j\bullet}(k+r,l)(z_i-z_j)^{(r)} \omega_{pq}(D^p_k)_i\otimes (D^q_l)_{\bullet}]   \\
   & =\sum_{r\geq 0}\sum_{k,l\geq 0}\sum_{p,q}P_{j\bullet}(k+r,l)(z_i-z_j)^{(r)} (z_j-z_i)^{(s)} \omega_{pq}(D^p_{k-s})_i\otimes (D^q_l)_{\bullet}.
\end{align*}
Then we sum over $s$
\begin{align*}
  &\sum_{s\geq 0}  [[(z_i-z_j)(L_{-1})_i]^{(s)},\mathcal{P}_{i\bullet}(i\rightarrow j)]   \\ &=\sum_{s\geq 0}\sum_{r\geq 0}\sum_{k,l\geq 0}\sum_{p,q}P_{j\bullet}(k+r,l)(z_i-z_j)^{(r)} (z_j-z_i)^{(s)} \omega_{pq}(D^p_{k-s})_i\otimes (D^q_l)_{\bullet}\\
&=\sum_{n\geq 0}\sum_{m\geq 0}\sum_{k,l\geq 0}\sum_{p,q}P_{j\bullet}(m,l)(z_i-z_j)^{(m-k)} (z_j-z_i)^{(k-n)} \omega_{pq}(D^p_{n})_i\otimes (D^q_l)_{\bullet}\\
  &=\sum_{n\geq 0}\sum_{m,l\geq 0}\sum_{p,q}P_{j\bullet}(m,l) \omega_{pq}(D^p_{m})_i\otimes (D^q_l)_{\bullet},
\end{align*}
  here we set $D^p_{n}=(z_i-z_j)^{(n)}=0$, if $n<0$. We get the identity (\ref{PLCommute})
  $$
e^{(z_i-z_j)(L_{-1})_i}e^{\hbar \mathcal{P}_{i\bullet}(i\rightarrow j)}e^{-(z_i-z_j)(L_{-1})_i}=  e^{\hbar\mathcal{P}_{i\bullet}(i\rightarrow j)|_{z_i=z_j}} .
  $$
Using the same argument, we can derive the identity (\ref{QLCommute}).
\subsection{Regularized integrals}\label{ReInt}
In this appendix, we review the definition and some properties of regularized integrals introduced in \cite{li2020regularized}. We always assume that $X$ is a Riemann surface. Most of the materials in this appendix are stated without proof, we refer to \cite{li2020regularized} for details.

\begin{lem}\label{Decompo}
  Any $\eta\in \Omega^{0,1}(X,\omega_X(*D))$ can be written as
  $$
  \eta=\eta_{\log}+\partial\eta_1,
  $$
  where
  $$\eta_{\log}\in \Omega^{0,1}(X,\omega_X(\log D)), \eta_1\in \Omega^{0,1}(X,\mathcal{O}_X(*D)).
  $$
\end{lem}
Using this decomposition, we can define the regularized integral.
\begin{defn}
  We define the regularized integral
  $$
  \dashint_X: \Omega^{0,1}(X,\omega_X(*D))\rightarrow \mathbb{C}
  $$
  by
  $$
  \dashint_X\eta=\int_X\eta_{\log},
  $$
  here $\eta=\eta_{\log}+\partial\eta_1$ as in Lemma \ref{Decompo}.
\end{defn}

\begin{thm}
  The regularized integral is well defined and factors through the quotient
  $$
  \dashint_X:\frac{\Omega^{0,1}(X,\omega_X(*D))}{\partial \Omega^{0,1}(X,\mathcal{O}_X(*D))}\rightarrow \mathbb{C}.
  $$
\end{thm}
Now we describe  an extension of the residue operation. We first recall a lemma from \cite{li2020regularized}.
\begin{lem}\label{Log}
  Let $f$ be a smooth function around the origin $0\in \mathbb{C}$. Let $n$ be a positive integer. Then
  $$
  \lim_{\varepsilon\rightarrow 0}\int_{|z|=\varepsilon}\frac{fdz}{z^n}=\frac{2\pi i}{(n-1)!}\partial^{n-1}_zf(0).
  $$
  Here the integration contour is counter-clockwise oriented.
\end{lem}
\begin{defn}
  Let $\eta\in \Omega^{0,0}(X,\omega_X(*D))$. Let $p\in D$ and $z$ be a local coordinate around $p$ such that $z(p)=0$. Then the following limit
  $$
  \lim_{\varepsilon\rightarrow 0}\frac{1}{2\pi i}\int_{|z|=\varepsilon}\eta
  $$
  exists and does not depend on the choice of the local coordinate $z$. Here the integration contour is counter-clockwise oriented. We will denote this limit by $\mathrm{Res}_{z\rightarrow p}\eta.$
\end{defn}

The next theorem describes a version of the Stokes formula for a regularized integral.
\begin{thm}
  Let $\eta\in \Omega^{0,0}(X,\omega_X(*D))$. Then we have
  $$
  \dashint_X\bar{\partial}\eta=-2\pi i\sum_{p\in D}\mathrm{Res}_{z\rightarrow p}(\eta).
  $$
\end{thm}

Now we  turn to the definition of the regularized integral over the product of Riemann surfaces. One can define a regularized integral using a method similar to Lemma \ref{Log}
$$
\dashint_{X^I}:\Omega^{0,|I|}(X^I,\omega_{X^I}(*\Delta_I)).
$$
In fact, we can write $\eta\in   \Omega^{0,|I|}(X^I,\omega_{X^I}(*\Delta_I))$

$$
  \eta=\eta_{\log}+d_{\mathrm{DR}}\eta_1,
  $$
  where
  $$
  \eta_{\log}\in \Omega^{0,|I|}(X^I,\omega_{X^I}(\log \Delta_I)), \eta_1\in \Omega^{0,|I|}(X^I,\omega_{X^I}(*\Delta_I)\otimes\Theta_{X^I}).
  $$
  Define
  $$
  \dashint_{X^I}\eta:=\int_{X^I}\eta_{\log}.
  $$

And it is equal to the iterated regularized integral over factors of $X^I$:
$$
\dashint_{X^I}=\dashint_{X_1}\cdots\dashint_{X_{|I|}}, X^I=X_1\times\cdots\times X_{|I|}.
$$
And the definition of the residue can be also generalized as follows.
\begin{defn}\label{genResidue}
  Choose a Riemann metric on $X$. Then we have following diagram
$$
\begin{tikzcd}
{X_{ij,\varepsilon}} \arrow[d, "\pi_j", bend left] \arrow[d, "\pi_i"', bend right] & {=\{|p_i-p_j|=\epsilon\}\subset X^I} \\
X^{I'}                                                                          &
\end{tikzcd}
$$
here $\epsilon>0$ is a sufficiently small positive real number and $\pi_{i,j}$ is the projection by forgetting the i-th or j-th factor. Projections $\pi_i$ and $\pi_j$ are circle bundles and give rise  to two ways to identify$X_{ij,\varepsilon}$  with $ST_{ij}(X^{I'})$. Here $ST_{ij}(X^{I'})$ is the unit tangent circle bundle on the factor whose preimage is $\{i,j\}$. Denote these two isomorphisms by $\mathrm{Iso}_i, \mathrm{Iso}_j$ and the inclusion by $\iota:X_{ij,\varepsilon}\hookrightarrow X^I$, we define the residue as follows
$$
\mathrm{Res}_{z_i\rightarrow z_j}\eta= \lim_{\varepsilon\rightarrow 0}\frac{1}{2\pi i}\int_{S^1} \mathrm{Iso}_i^*\iota^*\eta,\ \mathrm{Res}_{z_j\rightarrow z_i}\eta= \lim_{\varepsilon\rightarrow 0}\frac{1}{2\pi i}\int_{S^1} \mathrm{Iso}_j^*\iota^*\eta.
$$
\end{defn}
In in terms of local coordinates, we have the expression (\ref{Local}).
Now we give a proof of Proposition \ref{Stokes}.
\begin{proof}
The first statement follows from the part (a) of a lemma in  \cite[2.1.1]{Beilinson1988tb}.

Now we prove the second statement using the Stokes theorem.  By type reasons, we can assume that $\eta=\sum_{t\in I}\eta^t$, where
   $$
  \eta^{t}=f\cdot dz_1\wedge d\bar{z}_1\wedge\cdots \wedge dz_{t}\wedge\widehat{d\bar{z}}_{t}\wedge\cdots\wedge dz_{n}\wedge d\bar{z}_{n}.
  $$
  By the definition of residue, we have
  $$
  \mathrm{Res}_{z_{\alpha'}\rightarrow z_{\alpha''}}\eta= \mathrm{Res}_{z_{\alpha'}\rightarrow z_{\alpha''}}(\eta^{\alpha'}+\eta^{\alpha''}).
  $$
  Note that $\mathrm{Res}_{z_{\alpha'}\rightarrow z_{\alpha''}}\eta^{\alpha''}$ is equal to $\mathrm{Res}_{z_{\alpha''}\rightarrow z_{\alpha'}}\eta^{\alpha''}$ modulo holomorphic total derivatives, thus they have the same regularized integral.
  Now we have
  \begin{align*}
     \dashint_{X^{I}}\bar{\partial}\eta&=\dashint_{X^{I}}(\sum_{t}\bar{\partial}_{z_t}\eta^{t})  \\
     &=-2\pi i\sum_t \dashint_{X^{I-\{t\}}}(\sum_{s\in I-\{t\}}\mathrm{Res}_{z_t\rightarrow z_s}\eta^{t})\\
     &=-2\pi i\sum_{I'\in \equiQ(I,|I|-1)}\dashint_{X^{I'}}(\mathrm{Res}_{z_{\alpha'(I')}\rightarrow z_{\alpha''(I')}}\eta^{\alpha'(I')}+\mathrm{Res}_{z_{\alpha''(I')}\rightarrow z_{\alpha'(I')}}\eta^{\alpha''(I')})\\
     &=-2\pi i\sum_{I'\in \equiQ(I,|I|-1)}\dashint_{X^{I'}}(\mathrm{Res}_{z_{\alpha'(I')}\rightarrow z_{\alpha''(I')}}\eta^{\alpha'(I')}+\mathrm{Res}_{z_{\alpha'(I')}\rightarrow z_{\alpha''(I')}}\eta^{\alpha''(I')})\\
     &=-2\pi i\sum_{I'\in \equiQ(I,|I|-1)}\dashint_{X^{I'}}\mathrm{Res}_{z_{\alpha'(I')}\rightarrow z_{\alpha''(I')}}\eta=-2\pi i\sum_{I'\in \equiQ(I,|I|-1)}\dashint_{X^{I'}}\mathrm{Res}_{I,I'}\eta.
  \end{align*}

\end{proof}

\subsection{Feynman diagram computations}\label{FeynmanCompute}

  Here we provide some details of the Feynman diagram argument used in Proposition \ref{BVFeynman}. This discussion is similar to \cite{2011WittenGenus,gorbounov2016chiral,Btwisted}.

By the standard technique of Feynman diagrams (see e.g. \cite{costello2011renormalization}), $\Trace_{\widetilde{W}_N}(\underline{\mathrm{\mathbf{1}}})$ can be expressed as
$$
\Trace_{\widetilde{W}_N}(\underline{\mathrm{\mathbf{1}}})=\exp(\frac{1}{\hbar}\mathop{\sum}\limits_{\Gamma}h^{g(\Gamma)}\frac{W_{\Gamma}}{|\mathrm{Aut}(\Gamma)|}),
$$
where the sum is over the connected graph $\Gamma$. The Feynman weights $W_{\Gamma}$ and the automorphism group $\mathrm{Aut}(\Gamma)$ of the graph are defined in the usual way.   The genus of the graph $g(\Gamma)$ is defined by
$$
g(\Gamma)=b_1(\Gamma)+\sum_{v\in V(\Gamma)}\ell(v),
$$
 here $b_1(\Gamma)$ is the first Betti number of $\Gamma$ and $V(\Gamma)$ is the set of vertices of $\Gamma$. Since our inputs are $\frac{\pi}{i\hbar}f_i(\gamma)\beta^{l_i}+g_i(\gamma)\partial\gamma^{m_i}, i=1,\dots,k+1$, we only have the following two types of vertices
 \begin{itemize}
   \item Tree vertices $v$ labeled by $\frac{\pi}{i\hbar}f_i(\gamma)\beta^{l_i},i=1,\dots,k+1.$ We set $\ell(v)=0.$
   \item One-loop vertices $v$ labeled by $g_i(\gamma)\partial\gamma^{m_i},i=1,\dots,k+1.$ We set $\ell(v)=1.$
 \end{itemize}
 The Feynman weight $W_{\Gamma}$ of a connected graph $\Gamma$ is a product of propagators $\hbar P(z,w)$ connecting $\beta$ and $\gamma$. The explicit formula for the Feynman weight is not needed for our purpose in this paper.

There are seven types of connected Feynman diagrams in the graph expansion of $\Trace_{\widetilde{W}_N}(\underline{\mathrm{\mathbf{1}}})$. The first three types of graphs involve the one-loop vertex $g_j(\gamma)\partial\gamma^{m_j}, j=1,\dots,k+1$. We will prove that these will have no contribution to the final result. The type $IV$ and $V$ graphs have tails labelled by $f_j(\gamma)\frac{\partial}{\partial\gamma^{l_j}}, j=1,\dots,k+1$. The weights of these graphs are also zero. In summary, only the last two types of graphs contribute.

\begin{itemize}
  \item Let $\mathfrak{G}_I$ be the set of graphs with a tail labeled by $g_j(\gamma)\partial\gamma^{m_j}, j=1,\dots,k+1$ and a propagator connecting the edge labeled by $\partial\gamma^{m_j}$. Then the weight $W_{\Gamma_I}=0$ for $\Gamma\in \mathfrak{G}_I$, since
      $$
      \dashint_{X}d^2z\partial_zP(z,w)=0.
      $$
      See Fig \ref{Type123}.
  \item Let $\mathfrak{G}_{II}$ be the set of graphs such that each element contains an internal vertex labeled by $g_j(\gamma)\partial\gamma^{m_j}, j=1,\dots,k+1$. Then each $\Gamma_{II}\in \mathfrak{G}_{II}$ must have a tail labeled by $f_{j'}(\gamma)\frac{\partial}{\partial\gamma^{l_{j'}}},j'=1,\dots,k+1$. Thus the weight $W_{\Gamma_{II}}=0$ for $\Gamma\in \mathfrak{G}_{II}$, since
      $$
      \dashint_{X}d^2zP(z,w)=0.
      $$
      See Fig \ref{Type123}.
      \item Let $\mathfrak{G}_{III}$ be the set of graphs that have a vertex labeled by $g_j(\gamma)\partial\gamma^{m_j}, j=1,\dots,k+1$, and no propagator connecting the edge labeled by $\partial\gamma^{m_j}$. Then the weight $W_{\Gamma_{III}}=0$ for $\Gamma_{III}\in \mathfrak{G}_{III}$, since by definition $\partial \gamma^{m_j}=0\in O_{\mathrm{BV}}$. See Fig \ref{Type123}.
        \item Let $\mathfrak{G}_{IV}$ be the set of graphs such that each element is a wheel diagram with a tail labeled by $g_j(\gamma)\partial\gamma^{m_j}, j=1,\dots,k+1$. Then the weight $W_{\Gamma_{IV}}=0$ for $\Gamma\in \mathfrak{G}_{IV}$ by the same reason discussed in the case of $\mathfrak{G}_{II}$. See Fig \ref{Type45}.
            \item Let $\mathfrak{G}_{V}$ be the set of graphs that are tree-like diagrams (may have one-loop vertex) with a tail labeled by $g_j(\gamma)\partial\gamma^{m_j}, j=1,\dots,k+1$. Then the weight $W_{\Gamma_{V}}=0$ for $\Gamma\in \mathfrak{G}_{V}$ by the same reason discussed in the case of $\mathfrak{G}_{II}$. See Fig \ref{Type45}.
           \item Let $\mathfrak{G}_{1}$ be the set of graphs that are wheel diagrams without tails. See Fig \ref{Type1}.
                  \item Let $\mathfrak{G}_{2}$ be the set of graphs such that each element consists of only a single vertex labeled by $f_{j}(\gamma)\frac{\partial}{\partial\gamma^{l_{j}}},j=1,\dots,k+1$ and its edges. See Fig \ref{Type2}.
\end{itemize}
\begin{figure}[h]
  \centering

\tikzset{every picture/.style={line width=0.75pt}} 

\begin{tikzpicture}[x=0.75pt,y=0.75pt,yscale=-1,xscale=1]

\draw  [dash pattern={on 0.84pt off 2.51pt}] (178,62) .. controls (178,48.19) and (189.19,37) .. (203,37) .. controls (216.81,37) and (228,48.19) .. (228,62) .. controls (228,75.81) and (216.81,87) .. (203,87) .. controls (189.19,87) and (178,75.81) .. (178,62) -- cycle ;
\draw    (299.25,64.24) -- (281.83,102.66) ;
\draw [shift={(281.83,102.66)}, rotate = 159.39] [color={rgb, 255:red, 0; green, 0; blue, 0 }  ][line width=0.75]    (-5.59,0) -- (5.59,0)(0,5.59) -- (0,-5.59)   ;
\draw [shift={(299.25,64.24)}, rotate = 114.39] [color={rgb, 255:red, 0; green, 0; blue, 0 }  ][fill={rgb, 255:red, 0; green, 0; blue, 0 }  ][line width=0.75]      (0, 0) circle [x radius= 3.35, y radius= 3.35]   ;
\draw    (299.25,64.24) -- (262.64,78.75) ;
\draw [shift={(260.46,79.61)}, rotate = 158.38] [color={rgb, 255:red, 0; green, 0; blue, 0 }  ][line width=0.75]      (0, 0) circle [x radius= 3.35, y radius= 3.35]   ;
\draw [shift={(299.25,64.24)}, rotate = 158.38] [color={rgb, 255:red, 0; green, 0; blue, 0 }  ][fill={rgb, 255:red, 0; green, 0; blue, 0 }  ][line width=0.75]      (0, 0) circle [x radius= 3.35, y radius= 3.35]   ;
\draw    (299.25,64.24) -- (267.82,38.48) ;
\draw [shift={(266,36.99)}, rotate = 219.34] [color={rgb, 255:red, 0; green, 0; blue, 0 }  ][line width=0.75]      (0, 0) circle [x radius= 3.35, y radius= 3.35]   ;
\draw [shift={(299.25,64.24)}, rotate = 219.34] [color={rgb, 255:red, 0; green, 0; blue, 0 }  ][fill={rgb, 255:red, 0; green, 0; blue, 0 }  ][line width=0.75]      (0, 0) circle [x radius= 3.35, y radius= 3.35]   ;
\draw    (299.25,64.24) -- (321.34,42.14) ;
\draw [shift={(323,40.48)}, rotate = 314.99] [color={rgb, 255:red, 0; green, 0; blue, 0 }  ][line width=0.75]      (0, 0) circle [x radius= 3.35, y radius= 3.35]   ;
\draw [shift={(299.25,64.24)}, rotate = 314.99] [color={rgb, 255:red, 0; green, 0; blue, 0 }  ][fill={rgb, 255:red, 0; green, 0; blue, 0 }  ][line width=0.75]      (0, 0) circle [x radius= 3.35, y radius= 3.35]   ;
\draw  [dash pattern={on 0.84pt off 2.51pt}]  (203,62) -- (246,61.99) ;
\draw [color={rgb, 255:red, 74; green, 144; blue, 226 }  ,draw opacity=1 ]   (246,61.99) -- (296.9,64.14) ;
\draw [shift={(299.25,64.24)}, rotate = 2.42] [color={rgb, 255:red, 74; green, 144; blue, 226 }  ,draw opacity=1 ][line width=0.75]      (0, 0) circle [x radius= 3.35, y radius= 3.35]   ;
\draw  [dash pattern={on 0.84pt off 2.51pt}] (180,188) .. controls (180,174.19) and (191.19,163) .. (205,163) .. controls (218.81,163) and (230,174.19) .. (230,188) .. controls (230,201.81) and (218.81,213) .. (205,213) .. controls (191.19,213) and (180,201.81) .. (180,188) -- cycle ;
\draw    (301.25,190.24) -- (283.83,228.66) ;
\draw [shift={(283.83,228.66)}, rotate = 159.39] [color={rgb, 255:red, 0; green, 0; blue, 0 }  ][line width=0.75]    (-5.59,0) -- (5.59,0)(0,5.59) -- (0,-5.59)   ;
\draw [shift={(301.25,190.24)}, rotate = 114.39] [color={rgb, 255:red, 0; green, 0; blue, 0 }  ][fill={rgb, 255:red, 0; green, 0; blue, 0 }  ][line width=0.75]      (0, 0) circle [x radius= 3.35, y radius= 3.35]   ;
\draw    (301.25,190.24) -- (264.64,204.75) ;
\draw [shift={(262.46,205.61)}, rotate = 158.38] [color={rgb, 255:red, 0; green, 0; blue, 0 }  ][line width=0.75]      (0, 0) circle [x radius= 3.35, y radius= 3.35]   ;
\draw [shift={(301.25,190.24)}, rotate = 158.38] [color={rgb, 255:red, 0; green, 0; blue, 0 }  ][fill={rgb, 255:red, 0; green, 0; blue, 0 }  ][line width=0.75]      (0, 0) circle [x radius= 3.35, y radius= 3.35]   ;
\draw    (301.25,190.24) -- (269.82,164.48) ;
\draw [shift={(268,162.99)}, rotate = 219.34] [color={rgb, 255:red, 0; green, 0; blue, 0 }  ][line width=0.75]      (0, 0) circle [x radius= 3.35, y radius= 3.35]   ;
\draw [shift={(301.25,190.24)}, rotate = 219.34] [color={rgb, 255:red, 0; green, 0; blue, 0 }  ][fill={rgb, 255:red, 0; green, 0; blue, 0 }  ][line width=0.75]      (0, 0) circle [x radius= 3.35, y radius= 3.35]   ;
\draw    (301.25,190.24) -- (323.34,168.14) ;
\draw [shift={(325,166.48)}, rotate = 314.99] [color={rgb, 255:red, 0; green, 0; blue, 0 }  ][line width=0.75]      (0, 0) circle [x radius= 3.35, y radius= 3.35]   ;
\draw [shift={(301.25,190.24)}, rotate = 314.99] [color={rgb, 255:red, 0; green, 0; blue, 0 }  ][fill={rgb, 255:red, 0; green, 0; blue, 0 }  ][line width=0.75]      (0, 0) circle [x radius= 3.35, y radius= 3.35]   ;
\draw  [dash pattern={on 0.84pt off 2.51pt}]  (205,188) -- (248,187.99) ;
\draw [color={rgb, 255:red, 74; green, 144; blue, 226 }  ,draw opacity=1 ]   (248,187.99) -- (298.9,190.14) ;
\draw [shift={(301.25,190.24)}, rotate = 2.42] [color={rgb, 255:red, 74; green, 144; blue, 226 }  ,draw opacity=1 ][line width=0.75]      (0, 0) circle [x radius= 3.35, y radius= 3.35]   ;
\draw  [dash pattern={on 0.84pt off 2.51pt}] (329.41,190.96) .. controls (329.41,172.18) and (344.64,156.96) .. (363.41,156.96) .. controls (382.19,156.96) and (397.41,172.18) .. (397.41,190.96) .. controls (397.41,209.73) and (382.19,224.96) .. (363.41,224.96) .. controls (344.64,224.96) and (329.41,209.73) .. (329.41,190.96) -- cycle ;
\draw  [dash pattern={on 0.84pt off 2.51pt}]  (320,192.32) -- (342,191.32) ;
\draw    (390.41,192.96) -- (418,198.32) ;
\draw [shift={(418,198.32)}, rotate = 56.01] [color={rgb, 255:red, 0; green, 0; blue, 0 }  ][line width=0.75]    (-5.59,0) -- (5.59,0)(0,5.59) -- (0,-5.59)   ;
\draw [shift={(390.41,192.96)}, rotate = 11.01] [color={rgb, 255:red, 0; green, 0; blue, 0 }  ][fill={rgb, 255:red, 0; green, 0; blue, 0 }  ][line width=0.75]      (0, 0) circle [x radius= 3.35, y radius= 3.35]   ;
\draw [color={rgb, 255:red, 74; green, 144; blue, 226 }  ,draw opacity=1 ][line width=0.75]    (390.41,192.96) -- (374,192.32) ;
\draw    (400.92,213.24) -- (390.41,192.96) ;
\draw [shift={(390.41,192.96)}, rotate = 242.62] [color={rgb, 255:red, 0; green, 0; blue, 0 }  ][fill={rgb, 255:red, 0; green, 0; blue, 0 }  ][line width=0.75]      (0, 0) circle [x radius= 3.35, y radius= 3.35]   ;
\draw [shift={(402,215.32)}, rotate = 242.62] [color={rgb, 255:red, 0; green, 0; blue, 0 }  ][line width=0.75]      (0, 0) circle [x radius= 3.35, y radius= 3.35]   ;
\draw [color={rgb, 255:red, 74; green, 144; blue, 226 }  ,draw opacity=1 ][line width=0.75]  [dash pattern={on 0.84pt off 2.51pt}]  (374,192.32) -- (357.59,191.69) ;
\draw    (340.08,301.06) -- (322.95,337.57) ;
\draw [shift={(322.95,337.57)}, rotate = 160.13] [color={rgb, 255:red, 0; green, 0; blue, 0 }  ][line width=0.75]    (-5.59,0) -- (5.59,0)(0,5.59) -- (0,-5.59)   ;
\draw [shift={(340.08,301.06)}, rotate = 115.13] [color={rgb, 255:red, 0; green, 0; blue, 0 }  ][fill={rgb, 255:red, 0; green, 0; blue, 0 }  ][line width=0.75]      (0, 0) circle [x radius= 3.35, y radius= 3.35]   ;
\draw    (340.08,301.06) -- (304.13,314.83) ;
\draw [shift={(301.93,315.67)}, rotate = 159.04] [color={rgb, 255:red, 0; green, 0; blue, 0 }  ][line width=0.75]      (0, 0) circle [x radius= 3.35, y radius= 3.35]   ;
\draw [shift={(340.08,301.06)}, rotate = 159.04] [color={rgb, 255:red, 0; green, 0; blue, 0 }  ][fill={rgb, 255:red, 0; green, 0; blue, 0 }  ][line width=0.75]      (0, 0) circle [x radius= 3.35, y radius= 3.35]   ;
\draw    (340.08,301.06) -- (299.44,284.69) ;
\draw [shift={(297.26,283.81)}, rotate = 201.95] [color={rgb, 255:red, 0; green, 0; blue, 0 }  ][line width=0.75]      (0, 0) circle [x radius= 3.35, y radius= 3.35]   ;
\draw [shift={(340.08,301.06)}, rotate = 201.95] [color={rgb, 255:red, 0; green, 0; blue, 0 }  ][fill={rgb, 255:red, 0; green, 0; blue, 0 }  ][line width=0.75]      (0, 0) circle [x radius= 3.35, y radius= 3.35]   ;
\draw    (340.08,301.06) -- (376.68,295.26) ;
\draw [shift={(379,294.89)}, rotate = 350.99] [color={rgb, 255:red, 0; green, 0; blue, 0 }  ][line width=0.75]      (0, 0) circle [x radius= 3.35, y radius= 3.35]   ;
\draw [shift={(340.08,301.06)}, rotate = 350.99] [color={rgb, 255:red, 0; green, 0; blue, 0 }  ][fill={rgb, 255:red, 0; green, 0; blue, 0 }  ][line width=0.75]      (0, 0) circle [x radius= 3.35, y radius= 3.35]   ;
\draw  [dash pattern={on 0.84pt off 2.51pt}]  (243.83,298.82) -- (286.83,298.81) ;
\draw [color={rgb, 255:red, 74; green, 144; blue, 226 }  ,draw opacity=1 ]   (286.83,298.81) -- (337.73,300.96) ;
\draw [shift={(340.08,301.06)}, rotate = 2.42] [color={rgb, 255:red, 74; green, 144; blue, 226 }  ,draw opacity=1 ][line width=0.75]      (0, 0) circle [x radius= 3.35, y radius= 3.35]   ;
\draw  [dash pattern={on 0.84pt off 2.51pt}]  (435.83,311.05) -- (392.97,307.59) ;
\draw [color={rgb, 255:red, 74; green, 144; blue, 226 }  ,draw opacity=1 ]   (392.97,307.59) -- (342.41,301.35) ;
\draw [shift={(340.08,301.06)}, rotate = 187.04] [color={rgb, 255:red, 74; green, 144; blue, 226 }  ,draw opacity=1 ][line width=0.75]      (0, 0) circle [x radius= 3.35, y radius= 3.35]   ;
\draw [color={rgb, 255:red, 74; green, 144; blue, 226 }  ,draw opacity=1 ]   (320,192.32) -- (303.59,190.5) ;
\draw [shift={(301.25,190.24)}, rotate = 186.34] [color={rgb, 255:red, 74; green, 144; blue, 226 }  ,draw opacity=1 ][line width=0.75]      (0, 0) circle [x radius= 3.35, y radius= 3.35]   ;
\draw [color={rgb, 255:red, 74; green, 144; blue, 226 }  ,draw opacity=1 ][line width=0.75]    (504,56.64) -- (570,56.64) ;
\draw    (531,97.64) -- (579,97.29) ;
\draw [shift={(579,97.29)}, rotate = 404.59] [color={rgb, 255:red, 0; green, 0; blue, 0 }  ][line width=0.75]    (-5.59,0) -- (5.59,0)(0,5.59) -- (0,-5.59)   ;
\draw    (570.65,143.64) -- (527,143.64) ;
\draw [shift={(573,143.64)}, rotate = 180] [color={rgb, 255:red, 0; green, 0; blue, 0 }  ][line width=0.75]      (0, 0) circle [x radius= 3.35, y radius= 3.35]   ;
\draw  [dash pattern={on 0.84pt off 2.51pt}] (291.37,135.49) .. controls (298.05,135.11) and (304.3,139.76) .. (305.31,145.87) .. controls (306.33,151.98) and (301.73,157.24) .. (295.04,157.62) .. controls (288.35,158) and (282.11,153.35) .. (281.09,147.24) .. controls (280.08,141.13) and (284.68,135.86) .. (291.37,135.49) -- cycle ;
\draw  [dash pattern={on 0.84pt off 2.51pt}]  (292.26,147.56) -- (295.43,166.6) ;
\draw [color={rgb, 255:red, 74; green, 144; blue, 226 }  ,draw opacity=1 ]   (295.43,166.6) -- (300.69,187.96) ;
\draw [shift={(301.25,190.24)}, rotate = 76.16] [color={rgb, 255:red, 74; green, 144; blue, 226 }  ,draw opacity=1 ][line width=0.75]      (0, 0) circle [x radius= 3.35, y radius= 3.35]   ;

\draw (261,50.4) node [anchor=north west][inner sep=0.75pt]  [font=\tiny]  {$\partial \gamma ^{m}$};
\draw (96,53.4) node [anchor=north west][inner sep=0.75pt]    {$\Gamma _{I} :$};
\draw (366,66.4) node [anchor=north west][inner sep=0.75pt]  [font=\tiny]  {$\int \partial _{z} P( z,w) =0$};
\draw (292.5,108.64) node [anchor=north west][inner sep=0.75pt]  [font=\tiny]  {$dy=\gamma ^{\dagger }$};
\draw (98,179.4) node [anchor=north west][inner sep=0.75pt]    {$\Gamma _{II} :$};
\draw (437,190.4) node [anchor=north west][inner sep=0.75pt]  [font=\tiny]  {$\int P( z,w) =0$};
\draw (294.5,234.64) node [anchor=north west][inner sep=0.75pt]  [font=\tiny]  {$dy=\gamma ^{\dagger }$};
\draw (350,129.4) node [anchor=north west][inner sep=0.75pt]    {$\Gamma _{II} '$};
\draw (293.94,343.4) node [anchor=north west][inner sep=0.75pt]  [font=\tiny]  {$dy=\gamma ^{\dagger }$};
\draw (243.46,271.55) node [anchor=north west][inner sep=0.75pt]  [font=\tiny]  {$\partial \gamma ^{m} =0\in O_{BV}$};
\draw (108,293.4) node [anchor=north west][inner sep=0.75pt]    {$\Gamma _{III} :$};
\draw (521,38.4) node [anchor=north west][inner sep=0.75pt]  [font=\tiny]  {$P( z,w)$};
\draw (492,52.4) node [anchor=north west][inner sep=0.75pt]  [font=\tiny]  {$\beta $};
\draw (574,52.4) node [anchor=north west][inner sep=0.75pt]  [font=\tiny]  {$\gamma $};
\draw (589,93.4) node [anchor=north west][inner sep=0.75pt]  [font=\tiny]  {$dy=\gamma ^{\dagger }$};
\draw (589,139.4) node [anchor=north west][inner sep=0.75pt]  [font=\tiny]  {$y=\gamma $};

\end{tikzpicture}
  \caption{}\label{Type123}
\end{figure}
\begin{figure}
  \centering

\tikzset{every picture/.style={line width=0.75pt}} 

\begin{tikzpicture}[x=0.75pt,y=0.75pt,yscale=-1,xscale=1]

\draw [color={rgb, 255:red, 74; green, 144; blue, 226 }  ,draw opacity=1 ][line width=0.75]    (118,154.32) -- (136.33,127.79) ;
\draw [shift={(136.33,127.79)}, rotate = 304.65] [color={rgb, 255:red, 74; green, 144; blue, 226 }  ,draw opacity=1 ][fill={rgb, 255:red, 74; green, 144; blue, 226 }  ,fill opacity=1 ][line width=0.75]      (0, 0) circle [x radius= 3.35, y radius= 3.35]   ;
\draw [shift={(118,154.32)}, rotate = 304.65] [color={rgb, 255:red, 74; green, 144; blue, 226 }  ,draw opacity=1 ][fill={rgb, 255:red, 74; green, 144; blue, 226 }  ,fill opacity=1 ][line width=0.75]      (0, 0) circle [x radius= 3.35, y radius= 3.35]   ;
\draw [color={rgb, 255:red, 74; green, 144; blue, 226 }  ,draw opacity=1 ][fill={rgb, 255:red, 74; green, 144; blue, 226 }  ,fill opacity=1 ][line width=0.75]    (118,154.32) -- (149.67,183.8) ;
\draw [color={rgb, 255:red, 74; green, 144; blue, 226 }  ,draw opacity=1 ][line width=0.75]    (149.67,183.8) -- (183.83,180.11) ;
\draw [shift={(183.83,180.11)}, rotate = 353.85] [color={rgb, 255:red, 74; green, 144; blue, 226 }  ,draw opacity=1 ][fill={rgb, 255:red, 74; green, 144; blue, 226 }  ,fill opacity=1 ][line width=0.75]      (0, 0) circle [x radius= 3.35, y radius= 3.35]   ;
\draw [shift={(149.67,183.8)}, rotate = 353.85] [color={rgb, 255:red, 74; green, 144; blue, 226 }  ,draw opacity=1 ][fill={rgb, 255:red, 74; green, 144; blue, 226 }  ,fill opacity=1 ][line width=0.75]      (0, 0) circle [x radius= 3.35, y radius= 3.35]   ;
\draw [color={rgb, 255:red, 74; green, 144; blue, 226 }  ,draw opacity=1 ][line width=0.75]    (136.33,127.79) -- (176.33,124.85) ;
\draw [color={rgb, 255:red, 74; green, 144; blue, 226 }  ,draw opacity=1 ][line width=0.75]    (176.33,124.85) -- (191.33,151.37) ;
\draw [shift={(191.33,151.37)}, rotate = 60.51] [color={rgb, 255:red, 74; green, 144; blue, 226 }  ,draw opacity=1 ][fill={rgb, 255:red, 74; green, 144; blue, 226 }  ,fill opacity=1 ][line width=0.75]      (0, 0) circle [x radius= 3.35, y radius= 3.35]   ;
\draw [shift={(176.33,124.85)}, rotate = 60.51] [color={rgb, 255:red, 74; green, 144; blue, 226 }  ,draw opacity=1 ][fill={rgb, 255:red, 74; green, 144; blue, 226 }  ,fill opacity=1 ][line width=0.75]      (0, 0) circle [x radius= 3.35, y radius= 3.35]   ;
\draw [color={rgb, 255:red, 74; green, 144; blue, 226 }  ,draw opacity=1 ][line width=0.75]    (191.33,151.37) -- (183.83,180.11) ;
\draw    (108.73,96.65) -- (136.33,127.79) ;
\draw [shift={(136.33,127.79)}, rotate = 48.45] [color={rgb, 255:red, 0; green, 0; blue, 0 }  ][fill={rgb, 255:red, 0; green, 0; blue, 0 }  ][line width=0.75]      (0, 0) circle [x radius= 3.35, y radius= 3.35]   ;
\draw [shift={(107.17,94.89)}, rotate = 48.45] [color={rgb, 255:red, 0; green, 0; blue, 0 }  ][line width=0.75]      (0, 0) circle [x radius= 3.35, y radius= 3.35]   ;
\draw    (124.67,82.85) -- (136.33,127.79) ;
\draw [shift={(124.67,82.85)}, rotate = 120.45] [color={rgb, 255:red, 0; green, 0; blue, 0 }  ][line width=0.75]    (-5.59,0) -- (5.59,0)(0,5.59) -- (0,-5.59)   ;
\draw    (149.67,183.8) -- (131.33,224.32) ;
\draw [shift={(131.33,224.32)}, rotate = 159.34] [color={rgb, 255:red, 0; green, 0; blue, 0 }  ][line width=0.75]    (-5.59,0) -- (5.59,0)(0,5.59) -- (0,-5.59)   ;
\draw [shift={(149.67,183.8)}, rotate = 114.34] [color={rgb, 255:red, 0; green, 0; blue, 0 }  ][fill={rgb, 255:red, 0; green, 0; blue, 0 }  ][line width=0.75]      (0, 0) circle [x radius= 3.35, y radius= 3.35]   ;
\draw    (183.83,180.11) ;
\draw [shift={(183.83,180.11)}, rotate = 0] [color={rgb, 255:red, 0; green, 0; blue, 0 }  ][fill={rgb, 255:red, 0; green, 0; blue, 0 }  ][line width=0.75]      (0, 0) circle [x radius= 3.35, y radius= 3.35]   ;
\draw [shift={(183.83,180.11)}, rotate = 0] [color={rgb, 255:red, 0; green, 0; blue, 0 }  ][fill={rgb, 255:red, 0; green, 0; blue, 0 }  ][line width=0.75]      (0, 0) circle [x radius= 3.35, y radius= 3.35]   ;
\draw    (191.33,151.37) -- (228,141.8) ;
\draw [shift={(228,141.8)}, rotate = 390.36] [color={rgb, 255:red, 0; green, 0; blue, 0 }  ][line width=0.75]    (-5.59,0) -- (5.59,0)(0,5.59) -- (0,-5.59)   ;
\draw [shift={(191.33,151.37)}, rotate = 345.36] [color={rgb, 255:red, 0; green, 0; blue, 0 }  ][fill={rgb, 255:red, 0; green, 0; blue, 0 }  ][line width=0.75]      (0, 0) circle [x radius= 3.35, y radius= 3.35]   ;
\draw    (176.33,124.85) -- (190.5,95.37) ;
\draw [shift={(190.5,95.37)}, rotate = 340.67] [color={rgb, 255:red, 0; green, 0; blue, 0 }  ][line width=0.75]    (-5.59,0) -- (5.59,0)(0,5.59) -- (0,-5.59)   ;
\draw [shift={(176.33,124.85)}, rotate = 295.67] [color={rgb, 255:red, 0; green, 0; blue, 0 }  ][fill={rgb, 255:red, 0; green, 0; blue, 0 }  ][line width=0.75]      (0, 0) circle [x radius= 3.35, y radius= 3.35]   ;
\draw    (79.67,158.01) -- (118,154.32) ;
\draw [shift={(118,154.32)}, rotate = 354.51] [color={rgb, 255:red, 0; green, 0; blue, 0 }  ][fill={rgb, 255:red, 0; green, 0; blue, 0 }  ][line width=0.75]      (0, 0) circle [x radius= 3.35, y radius= 3.35]   ;
\draw [shift={(79.67,158.01)}, rotate = 399.51] [color={rgb, 255:red, 0; green, 0; blue, 0 }  ][line width=0.75]    (-5.59,0) -- (5.59,0)(0,5.59) -- (0,-5.59)   ;
\draw    (88.42,169.7) -- (118,154.32) ;
\draw [shift={(118,154.32)}, rotate = 332.53] [color={rgb, 255:red, 0; green, 0; blue, 0 }  ][fill={rgb, 255:red, 0; green, 0; blue, 0 }  ][line width=0.75]      (0, 0) circle [x radius= 3.35, y radius= 3.35]   ;
\draw [shift={(86.33,170.79)}, rotate = 332.53] [color={rgb, 255:red, 0; green, 0; blue, 0 }  ][line width=0.75]      (0, 0) circle [x radius= 3.35, y radius= 3.35]   ;
\draw    (100.08,108.52) -- (136.33,127.79) ;
\draw [shift={(136.33,127.79)}, rotate = 28] [color={rgb, 255:red, 0; green, 0; blue, 0 }  ][fill={rgb, 255:red, 0; green, 0; blue, 0 }  ][line width=0.75]      (0, 0) circle [x radius= 3.35, y radius= 3.35]   ;
\draw [shift={(98,107.42)}, rotate = 28] [color={rgb, 255:red, 0; green, 0; blue, 0 }  ][line width=0.75]      (0, 0) circle [x radius= 3.35, y radius= 3.35]   ;
\draw [color={rgb, 255:red, 74; green, 144; blue, 226 }  ,draw opacity=1 ][line width=0.75]    (263.67,73.52) -- (318.67,73.52) ;
\draw [color={rgb, 255:red, 74; green, 144; blue, 226 }  ,draw opacity=1 ][line width=0.75]    (185.07,182.11) -- (197.81,202.6) ;
\draw [shift={(183.83,180.11)}, rotate = 58.13] [color={rgb, 255:red, 74; green, 144; blue, 226 }  ,draw opacity=1 ][line width=0.75]      (0, 0) circle [x radius= 3.35, y radius= 3.35]   ;
\draw    (197.81,202.6) -- (199.35,228.32) ;
\draw [shift={(199.35,228.32)}, rotate = 131.57] [color={rgb, 255:red, 0; green, 0; blue, 0 }  ][line width=0.75]    (-5.59,0) -- (5.59,0)(0,5.59) -- (0,-5.59)   ;
\draw [shift={(197.81,202.6)}, rotate = 86.57] [color={rgb, 255:red, 0; green, 0; blue, 0 }  ][fill={rgb, 255:red, 0; green, 0; blue, 0 }  ][line width=0.75]      (0, 0) circle [x radius= 3.35, y radius= 3.35]   ;
\draw [color={rgb, 255:red, 74; green, 144; blue, 226 }  ,draw opacity=1 ][line width=0.75]    (197.81,202.6) -- (218.91,201.23) ;
\draw [shift={(197.81,202.6)}, rotate = 356.3] [color={rgb, 255:red, 74; green, 144; blue, 226 }  ,draw opacity=1 ][fill={rgb, 255:red, 74; green, 144; blue, 226 }  ,fill opacity=1 ][line width=0.75]      (0, 0) circle [x radius= 3.35, y radius= 3.35]   ;
\draw [color={rgb, 255:red, 74; green, 144; blue, 226 }  ,draw opacity=1 ][line width=0.75]  [dash pattern={on 0.84pt off 2.51pt}]  (218.91,201.23) -- (240,199.87) ;
\draw    (276.41,210.96) -- (304,216.32) ;
\draw [shift={(304,216.32)}, rotate = 56.01] [color={rgb, 255:red, 0; green, 0; blue, 0 }  ][line width=0.75]    (-5.59,0) -- (5.59,0)(0,5.59) -- (0,-5.59)   ;
\draw [shift={(276.41,210.96)}, rotate = 11.01] [color={rgb, 255:red, 0; green, 0; blue, 0 }  ][fill={rgb, 255:red, 0; green, 0; blue, 0 }  ][line width=0.75]      (0, 0) circle [x radius= 3.35, y radius= 3.35]   ;
\draw [color={rgb, 255:red, 74; green, 144; blue, 226 }  ,draw opacity=1 ][line width=0.75]    (276.41,210.96) -- (256.32,204.4) ;
\draw [shift={(276.41,210.96)}, rotate = 198.08] [color={rgb, 255:red, 74; green, 144; blue, 226 }  ,draw opacity=1 ][fill={rgb, 255:red, 74; green, 144; blue, 226 }  ,fill opacity=1 ][line width=0.75]      (0, 0) circle [x radius= 3.35, y radius= 3.35]   ;
\draw [color={rgb, 255:red, 74; green, 144; blue, 226 }  ,draw opacity=1 ][line width=0.75]  [dash pattern={on 0.84pt off 2.51pt}]  (256.32,204.4) -- (240,199.87) ;
\draw    (418.06,156) -- (374,167.32) ;
\draw [shift={(374,167.32)}, rotate = 210.58] [color={rgb, 255:red, 0; green, 0; blue, 0 }  ][line width=0.75]    (-5.59,0) -- (5.59,0)(0,5.59) -- (0,-5.59)   ;
\draw [shift={(418.06,156)}, rotate = 165.58] [color={rgb, 255:red, 0; green, 0; blue, 0 }  ][fill={rgb, 255:red, 0; green, 0; blue, 0 }  ][line width=0.75]      (0, 0) circle [x radius= 3.35, y radius= 3.35]   ;
\draw [color={rgb, 255:red, 74; green, 144; blue, 226 }  ,draw opacity=1 ][line width=0.75]    (418.06,156) -- (442.31,154.55) ;
\draw [shift={(418.06,156)}, rotate = 356.58] [color={rgb, 255:red, 74; green, 144; blue, 226 }  ,draw opacity=1 ][fill={rgb, 255:red, 74; green, 144; blue, 226 }  ,fill opacity=1 ][line width=0.75]      (0, 0) circle [x radius= 3.35, y radius= 3.35]   ;
\draw [color={rgb, 255:red, 74; green, 144; blue, 226 }  ,draw opacity=1 ][line width=0.75]  [dash pattern={on 0.84pt off 2.51pt}]  (442.31,154.55) -- (466.55,153.1) ;
\draw    (508.4,164.88) -- (547,175.89) ;
\draw [shift={(547,175.89)}, rotate = 60.92] [color={rgb, 255:red, 0; green, 0; blue, 0 }  ][line width=0.75]    (-5.59,0) -- (5.59,0)(0,5.59) -- (0,-5.59)   ;
\draw [shift={(508.4,164.88)}, rotate = 15.92] [color={rgb, 255:red, 0; green, 0; blue, 0 }  ][fill={rgb, 255:red, 0; green, 0; blue, 0 }  ][line width=0.75]      (0, 0) circle [x radius= 3.35, y radius= 3.35]   ;
\draw [color={rgb, 255:red, 74; green, 144; blue, 226 }  ,draw opacity=1 ][line width=0.75]    (508.4,164.88) -- (485.31,157.91) ;
\draw [shift={(508.4,164.88)}, rotate = 196.79] [color={rgb, 255:red, 74; green, 144; blue, 226 }  ,draw opacity=1 ][fill={rgb, 255:red, 74; green, 144; blue, 226 }  ,fill opacity=1 ][line width=0.75]      (0, 0) circle [x radius= 3.35, y radius= 3.35]   ;
\draw [color={rgb, 255:red, 74; green, 144; blue, 226 }  ,draw opacity=1 ][line width=0.75]  [dash pattern={on 0.84pt off 2.51pt}]  (485.31,157.91) -- (466.55,153.1) ;
\draw    (380.02,133.52) -- (418.06,156) ;
\draw [shift={(418.06,156)}, rotate = 30.58] [color={rgb, 255:red, 0; green, 0; blue, 0 }  ][fill={rgb, 255:red, 0; green, 0; blue, 0 }  ][line width=0.75]      (0, 0) circle [x radius= 3.35, y radius= 3.35]   ;
\draw [shift={(378,132.32)}, rotate = 30.58] [color={rgb, 255:red, 0; green, 0; blue, 0 }  ][line width=0.75]      (0, 0) circle [x radius= 3.35, y radius= 3.35]   ;
\draw    (186.18,180.27) -- (213.66,182.16) ;
\draw [shift={(216,182.32)}, rotate = 3.93] [color={rgb, 255:red, 0; green, 0; blue, 0 }  ][line width=0.75]      (0, 0) circle [x radius= 3.35, y radius= 3.35]   ;
\draw [shift={(183.83,180.11)}, rotate = 3.93] [color={rgb, 255:red, 0; green, 0; blue, 0 }  ][line width=0.75]      (0, 0) circle [x radius= 3.35, y radius= 3.35]   ;
\draw    (286.92,231.24) -- (276.41,210.96) ;
\draw [shift={(276.41,210.96)}, rotate = 242.62] [color={rgb, 255:red, 0; green, 0; blue, 0 }  ][fill={rgb, 255:red, 0; green, 0; blue, 0 }  ][line width=0.75]      (0, 0) circle [x radius= 3.35, y radius= 3.35]   ;
\draw [shift={(288,233.32)}, rotate = 242.62] [color={rgb, 255:red, 0; green, 0; blue, 0 }  ][line width=0.75]      (0, 0) circle [x radius= 3.35, y radius= 3.35]   ;

\draw (114.5,63.64) node [anchor=north west][inner sep=0.75pt]  [font=\tiny]  {$dy=\gamma ^{\dagger }$};
\draw (50.42,94.93) node [anchor=north west][inner sep=0.75pt]  [font=\tiny]  {$y=\gamma $};
\draw (276.42,55.35) node [anchor=north west][inner sep=0.75pt]  [font=\tiny]  {$P( z,w)$};
\draw (251.83,69.45) node [anchor=north west][inner sep=0.75pt]  [font=\tiny]  {$\beta $};
\draw (320.17,65.45) node [anchor=north west][inner sep=0.75pt]  [font=\scriptsize]  {$\gamma $};
\draw (25,24.4) node [anchor=north west][inner sep=0.75pt]  [font=\footnotesize]  {$\text{The contribution of Type IV graphs } \ \text{vanishes}$};
\draw (337,26.4) node [anchor=north west][inner sep=0.75pt]  [font=\footnotesize]  {$\text{The contribution of Type V graphs } \ \text{vanishes}$};
\draw (204,249.4) node [anchor=north west][inner sep=0.75pt]  [font=\tiny]  {$\int P( z,w) =0$};
\draw (458,210.4) node [anchor=north west][inner sep=0.75pt]  [font=\tiny]  {$\int P( z,w) =0$};
\draw (259,195) node [anchor=north west][inner sep=0.75pt]   [align=left] {};
\draw (487,153) node [anchor=north west][inner sep=0.75pt]   [align=left] {};
\draw (19,62.4) node [anchor=north west][inner sep=0.75pt]    {$\Gamma _{IV} :$};
\draw (395,73.4) node [anchor=north west][inner sep=0.75pt]    {$\Gamma _{V} :$};
\draw    (485.22,206) -- (490.13,175.97) ;
\draw [shift={(490.45,174)}, rotate = 459.29] [color={rgb, 255:red, 0; green, 0; blue, 0 }  ][line width=0.75]    (10.93,-3.29) .. controls (6.95,-1.4) and (3.31,-0.3) .. (0,0) .. controls (3.31,0.3) and (6.95,1.4) .. (10.93,3.29)   ;
\draw    (236.57,245) -- (254.97,217.66) ;
\draw [shift={(256.09,216)}, rotate = 483.94] [color={rgb, 255:red, 0; green, 0; blue, 0 }  ][line width=0.75]    (10.93,-3.29) .. controls (6.95,-1.4) and (3.31,-0.3) .. (0,0) .. controls (3.31,0.3) and (6.95,1.4) .. (10.93,3.29)   ;

\end{tikzpicture}
  \caption{}\label{Type45}
\end{figure}
\begin{figure}
  \centering

\tikzset{every picture/.style={line width=0.75pt}} 

\begin{tikzpicture}[x=0.75pt,y=0.75pt,yscale=-1,xscale=1]

\draw [color={rgb, 255:red, 74; green, 144; blue, 226 }  ,draw opacity=1 ][line width=0.75]    (238,140.29) -- (260,104.29) ;
\draw [shift={(260,104.29)}, rotate = 301.43] [color={rgb, 255:red, 74; green, 144; blue, 226 }  ,draw opacity=1 ][fill={rgb, 255:red, 74; green, 144; blue, 226 }  ,fill opacity=1 ][line width=0.75]      (0, 0) circle [x radius= 3.35, y radius= 3.35]   ;
\draw [shift={(238,140.29)}, rotate = 301.43] [color={rgb, 255:red, 74; green, 144; blue, 226 }  ,draw opacity=1 ][fill={rgb, 255:red, 74; green, 144; blue, 226 }  ,fill opacity=1 ][line width=0.75]      (0, 0) circle [x radius= 3.35, y radius= 3.35]   ;
\draw [color={rgb, 255:red, 74; green, 144; blue, 226 }  ,draw opacity=1 ][fill={rgb, 255:red, 74; green, 144; blue, 226 }  ,fill opacity=1 ][line width=0.75]    (238,140.29) -- (276,180.29) ;
\draw [color={rgb, 255:red, 74; green, 144; blue, 226 }  ,draw opacity=1 ][line width=0.75]    (276,180.29) -- (317,175.29) ;
\draw [shift={(317,175.29)}, rotate = 353.05] [color={rgb, 255:red, 74; green, 144; blue, 226 }  ,draw opacity=1 ][fill={rgb, 255:red, 74; green, 144; blue, 226 }  ,fill opacity=1 ][line width=0.75]      (0, 0) circle [x radius= 3.35, y radius= 3.35]   ;
\draw [shift={(276,180.29)}, rotate = 353.05] [color={rgb, 255:red, 74; green, 144; blue, 226 }  ,draw opacity=1 ][fill={rgb, 255:red, 74; green, 144; blue, 226 }  ,fill opacity=1 ][line width=0.75]      (0, 0) circle [x radius= 3.35, y radius= 3.35]   ;
\draw [color={rgb, 255:red, 74; green, 144; blue, 226 }  ,draw opacity=1 ][line width=0.75]    (260,104.29) -- (308,100.29) ;
\draw [color={rgb, 255:red, 74; green, 144; blue, 226 }  ,draw opacity=1 ][line width=0.75]    (308,100.29) -- (326,136.29) ;
\draw [shift={(326,136.29)}, rotate = 63.43] [color={rgb, 255:red, 74; green, 144; blue, 226 }  ,draw opacity=1 ][fill={rgb, 255:red, 74; green, 144; blue, 226 }  ,fill opacity=1 ][line width=0.75]      (0, 0) circle [x radius= 3.35, y radius= 3.35]   ;
\draw [shift={(308,100.29)}, rotate = 63.43] [color={rgb, 255:red, 74; green, 144; blue, 226 }  ,draw opacity=1 ][fill={rgb, 255:red, 74; green, 144; blue, 226 }  ,fill opacity=1 ][line width=0.75]      (0, 0) circle [x radius= 3.35, y radius= 3.35]   ;
\draw [color={rgb, 255:red, 74; green, 144; blue, 226 }  ,draw opacity=1 ][line width=0.75]    (326,136.29) -- (317,175.29) ;
\draw    (226.45,61.49) -- (260,104.29) ;
\draw [shift={(260,104.29)}, rotate = 51.91] [color={rgb, 255:red, 0; green, 0; blue, 0 }  ][fill={rgb, 255:red, 0; green, 0; blue, 0 }  ][line width=0.75]      (0, 0) circle [x radius= 3.35, y radius= 3.35]   ;
\draw [shift={(225,59.64)}, rotate = 51.91] [color={rgb, 255:red, 0; green, 0; blue, 0 }  ][line width=0.75]      (0, 0) circle [x radius= 3.35, y radius= 3.35]   ;
\draw    (246,43.29) -- (260,104.29) ;
\draw [shift={(246,43.29)}, rotate = 122.07] [color={rgb, 255:red, 0; green, 0; blue, 0 }  ][line width=0.75]    (-5.59,0) -- (5.59,0)(0,5.59) -- (0,-5.59)   ;
\draw    (276,180.29) -- (254,235.29) ;
\draw [shift={(254,235.29)}, rotate = 156.8] [color={rgb, 255:red, 0; green, 0; blue, 0 }  ][line width=0.75]    (-5.59,0) -- (5.59,0)(0,5.59) -- (0,-5.59)   ;
\draw [shift={(276,180.29)}, rotate = 111.8] [color={rgb, 255:red, 0; green, 0; blue, 0 }  ][fill={rgb, 255:red, 0; green, 0; blue, 0 }  ][line width=0.75]      (0, 0) circle [x radius= 3.35, y radius= 3.35]   ;
\draw    (317,175.29) -- (360,208.29) ;
\draw [shift={(360,208.29)}, rotate = 82.5] [color={rgb, 255:red, 0; green, 0; blue, 0 }  ][line width=0.75]    (-5.59,0) -- (5.59,0)(0,5.59) -- (0,-5.59)   ;
\draw [shift={(317,175.29)}, rotate = 37.5] [color={rgb, 255:red, 0; green, 0; blue, 0 }  ][fill={rgb, 255:red, 0; green, 0; blue, 0 }  ][line width=0.75]      (0, 0) circle [x radius= 3.35, y radius= 3.35]   ;
\draw    (326,136.29) -- (370,123.29) ;
\draw [shift={(370,123.29)}, rotate = 388.54] [color={rgb, 255:red, 0; green, 0; blue, 0 }  ][line width=0.75]    (-5.59,0) -- (5.59,0)(0,5.59) -- (0,-5.59)   ;
\draw [shift={(326,136.29)}, rotate = 343.54] [color={rgb, 255:red, 0; green, 0; blue, 0 }  ][fill={rgb, 255:red, 0; green, 0; blue, 0 }  ][line width=0.75]      (0, 0) circle [x radius= 3.35, y radius= 3.35]   ;
\draw    (308,100.29) -- (325,60.29) ;
\draw [shift={(325,60.29)}, rotate = 338.03] [color={rgb, 255:red, 0; green, 0; blue, 0 }  ][line width=0.75]    (-5.59,0) -- (5.59,0)(0,5.59) -- (0,-5.59)   ;
\draw [shift={(308,100.29)}, rotate = 293.03] [color={rgb, 255:red, 0; green, 0; blue, 0 }  ][fill={rgb, 255:red, 0; green, 0; blue, 0 }  ][line width=0.75]      (0, 0) circle [x radius= 3.35, y radius= 3.35]   ;
\draw    (192,145.29) -- (235.66,140.55) ;
\draw [shift={(238,140.29)}, rotate = 353.8] [color={rgb, 255:red, 0; green, 0; blue, 0 }  ][line width=0.75]      (0, 0) circle [x radius= 3.35, y radius= 3.35]   ;
\draw [shift={(192,145.29)}, rotate = 398.8] [color={rgb, 255:red, 0; green, 0; blue, 0 }  ][line width=0.75]    (-5.59,0) -- (5.59,0)(0,5.59) -- (0,-5.59)   ;
\draw    (202.03,161.44) -- (238,140.29) ;
\draw [shift={(238,140.29)}, rotate = 329.54] [color={rgb, 255:red, 0; green, 0; blue, 0 }  ][fill={rgb, 255:red, 0; green, 0; blue, 0 }  ][line width=0.75]      (0, 0) circle [x radius= 3.35, y radius= 3.35]   ;
\draw [shift={(200,162.64)}, rotate = 329.54] [color={rgb, 255:red, 0; green, 0; blue, 0 }  ][line width=0.75]      (0, 0) circle [x radius= 3.35, y radius= 3.35]   ;
\draw    (216.01,77.85) -- (260,104.29) ;
\draw [shift={(260,104.29)}, rotate = 31.02] [color={rgb, 255:red, 0; green, 0; blue, 0 }  ][fill={rgb, 255:red, 0; green, 0; blue, 0 }  ][line width=0.75]      (0, 0) circle [x radius= 3.35, y radius= 3.35]   ;
\draw [shift={(214,76.64)}, rotate = 31.02] [color={rgb, 255:red, 0; green, 0; blue, 0 }  ][line width=0.75]      (0, 0) circle [x radius= 3.35, y radius= 3.35]   ;
\draw [color={rgb, 255:red, 74; green, 144; blue, 226 }  ,draw opacity=1 ][line width=0.75]    (419,64.64) -- (485,64.64) ;
\draw    (446,105.64) -- (494,105.29) ;
\draw [shift={(494,105.29)}, rotate = 404.59] [color={rgb, 255:red, 0; green, 0; blue, 0 }  ][line width=0.75]    (-5.59,0) -- (5.59,0)(0,5.59) -- (0,-5.59)   ;
\draw    (485.65,151.64) -- (442,151.64) ;
\draw [shift={(488,151.64)}, rotate = 180] [color={rgb, 255:red, 0; green, 0; blue, 0 }  ][line width=0.75]      (0, 0) circle [x radius= 3.35, y radius= 3.35]   ;

\draw (223,25.4) node [anchor=north west][inner sep=0.75pt]  [font=\tiny]  {$dy=\gamma ^{\dagger }$};
\draw (180,72.4) node [anchor=north west][inner sep=0.75pt]  [font=\tiny]  {$y=\gamma $};
\draw (436,46.4) node [anchor=north west][inner sep=0.75pt]  [font=\tiny]  {$P( z,w)$};
\draw (407,60.4) node [anchor=north west][inner sep=0.75pt]  [font=\tiny]  {$\beta $};
\draw (489,60.4) node [anchor=north west][inner sep=0.75pt]  [font=\tiny]  {$\gamma $};
\draw (504,101.4) node [anchor=north west][inner sep=0.75pt]  [font=\tiny]  {$dy=\gamma ^{\dagger }$};
\draw (504,147.4) node [anchor=north west][inner sep=0.75pt]  [font=\tiny]  {$y=\gamma $};
\draw (66,34.4) node [anchor=north west][inner sep=0.75pt]    {$\Gamma _{1} :$};

\end{tikzpicture}
  \caption{}\label{Type1}
\end{figure}
\begin{figure}
  \centering

\tikzset{every picture/.style={line width=0.75pt}} 

\begin{tikzpicture}[x=0.75pt,y=0.75pt,yscale=-1,xscale=1]

\draw    (309.86,90.63) -- (332.42,143.09) ;
\draw [shift={(332.42,143.09)}, rotate = 111.73] [color={rgb, 255:red, 0; green, 0; blue, 0 }  ][line width=0.75]    (-5.59,0) -- (5.59,0)(0,5.59) -- (0,-5.59)   ;
\draw [shift={(309.86,90.63)}, rotate = 66.73] [color={rgb, 255:red, 0; green, 0; blue, 0 }  ][fill={rgb, 255:red, 0; green, 0; blue, 0 }  ][line width=0.75]      (0, 0) circle [x radius= 3.35, y radius= 3.35]   ;
\draw    (309.86,90.63) -- (269.81,112.56) ;
\draw [shift={(267.74,113.69)}, rotate = 151.29] [color={rgb, 255:red, 0; green, 0; blue, 0 }  ][line width=0.75]      (0, 0) circle [x radius= 3.35, y radius= 3.35]   ;
\draw [shift={(309.86,90.63)}, rotate = 151.29] [color={rgb, 255:red, 0; green, 0; blue, 0 }  ][fill={rgb, 255:red, 0; green, 0; blue, 0 }  ][line width=0.75]      (0, 0) circle [x radius= 3.35, y radius= 3.35]   ;
\draw    (309.86,90.63) -- (264.62,64.56) ;
\draw [shift={(262.59,63.39)}, rotate = 209.95] [color={rgb, 255:red, 0; green, 0; blue, 0 }  ][line width=0.75]      (0, 0) circle [x radius= 3.35, y radius= 3.35]   ;
\draw [shift={(309.86,90.63)}, rotate = 209.95] [color={rgb, 255:red, 0; green, 0; blue, 0 }  ][fill={rgb, 255:red, 0; green, 0; blue, 0 }  ][line width=0.75]      (0, 0) circle [x radius= 3.35, y radius= 3.35]   ;
\draw    (352.27,100.25) -- (312.16,91.15) ;
\draw [shift={(309.86,90.63)}, rotate = 192.78] [color={rgb, 255:red, 0; green, 0; blue, 0 }  ][line width=0.75]      (0, 0) circle [x radius= 3.35, y radius= 3.35]   ;
\draw [shift={(354.56,100.77)}, rotate = 192.78] [color={rgb, 255:red, 0; green, 0; blue, 0 }  ][line width=0.75]      (0, 0) circle [x radius= 3.35, y radius= 3.35]   ;
\draw    (309.86,49.67) -- (309.86,88.28) ;
\draw [shift={(309.86,90.63)}, rotate = 90] [color={rgb, 255:red, 0; green, 0; blue, 0 }  ][line width=0.75]      (0, 0) circle [x radius= 3.35, y radius= 3.35]   ;
\draw [shift={(309.86,47.32)}, rotate = 90] [color={rgb, 255:red, 0; green, 0; blue, 0 }  ][line width=0.75]      (0, 0) circle [x radius= 3.35, y radius= 3.35]   ;

\draw (101,53.4) node [anchor=north west][inner sep=0.75pt]    {$\Gamma _{2} :$};
\draw (226.5,187.72) node [anchor=north west][inner sep=0.75pt]  [font=\tiny]  {$\theta \left( f_{j}( y)\frac{\partial }{\partial y^{l_{j}}}\right) =f_{j}( \gamma ) \beta ^{l_{j} \dagger } +\mathop{\sum}\limits _{s=1}^{N}\frac{\partial f_{j}( \gamma )}{\partial \gamma ^{s}} \gamma ^{s\dagger } \beta ^{l_{j}}$};
\draw (292.35,103.84) node [anchor=north west][inner sep=0.75pt]    {$$};
\draw (225.42,53.93) node [anchor=north west][inner sep=0.75pt]  [font=\tiny]  {$y=\gamma $};
\draw (348.5,140.64) node [anchor=north west][inner sep=0.75pt]  [font=\tiny]  {$dy=\gamma ^{\dagger }$};

\end{tikzpicture}
  \caption{}\label{Type2}
\end{figure}
From the above discussion, we conclude that
$$
\Trace(\underline{\mathrm{\mathbf{1}}})=e^{\frac{\pi\theta}{i\hbar}}e^{\mathop{\sum}\limits_{\small{\Gamma_1\in \mathfrak{G}_1}}\frac{W_{\Gamma_1}}{{\mathrm{Aut}(\Gamma_1)|}}}.
$$
In particular, $  \Trace_{(k),0}=e^{\mathop{\sum}\limits_{\small{\Gamma_1\in \mathfrak{G}_1}}\frac{W_{\Gamma_1}}{{\mathrm{Aut}(\Gamma_1)|}}}$ is a product of weights of some wheel diagrams in $\mathfrak{G}_1$ and $  \Trace_{(k+1),-1}$ is the sum of products of $  \Trace_{(k),0}$ and the weight of a diagram in $\mathfrak{G}_2$.

And we have
\begin{equation}\label{FeynmanBVtreeOneloop}
  \Trace_{(k+1),-1}=\sum_{j=1}^{k+1} (-1)^{\bullet_j} \Delta_{\mathrm{BV}}(\dashint_{X}dz_j\frac{\pi}{i}(\sum_{s=1}^{N}\frac{\partial f_j(\gamma)}{\partial \gamma^s}\gamma^{s\dagger}\beta^{l_j}d\bar{z}_j+f_j(\gamma)\beta^{l_j\dagger}d\bar{z}_j)\wedge \mathrm{\mathbf{Tr}}_{(k),0}\{\xi_1\wedge\cdots\wedge \widehat{\xi}_j\wedge\cdots\wedge \xi_{k+1}\}).
\end{equation}
See Fig \ref{TraceTreeOneloop}.
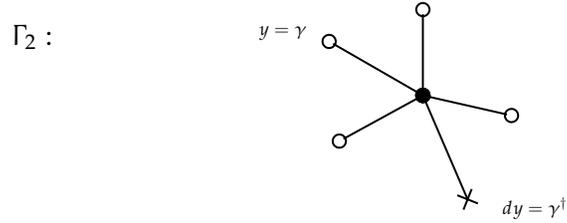
\begin{figure}
  \centering

\tikzset{every picture/.style={line width=0.75pt}} 

\begin{tikzpicture}[x=0.75pt,y=0.75pt,yscale=-1,xscale=1]

\draw [color={rgb, 255:red, 74; green, 144; blue, 226 }  ,draw opacity=1 ][line width=0.75]    (291.11,122.44) -- (302.43,103.93) ;
\draw [shift={(302.43,103.93)}, rotate = 301.43] [color={rgb, 255:red, 74; green, 144; blue, 226 }  ,draw opacity=1 ][fill={rgb, 255:red, 74; green, 144; blue, 226 }  ,fill opacity=1 ][line width=0.75]      (0, 0) circle [x radius= 3.35, y radius= 3.35]   ;
\draw [shift={(291.11,122.44)}, rotate = 301.43] [color={rgb, 255:red, 74; green, 144; blue, 226 }  ,draw opacity=1 ][fill={rgb, 255:red, 74; green, 144; blue, 226 }  ,fill opacity=1 ][line width=0.75]      (0, 0) circle [x radius= 3.35, y radius= 3.35]   ;
\draw [color={rgb, 255:red, 74; green, 144; blue, 226 }  ,draw opacity=1 ][fill={rgb, 255:red, 74; green, 144; blue, 226 }  ,fill opacity=1 ][line width=0.75]    (291.11,122.44) -- (310.66,143.01) ;
\draw [color={rgb, 255:red, 74; green, 144; blue, 226 }  ,draw opacity=1 ][line width=0.75]    (310.66,143.01) -- (331.74,140.44) ;
\draw [shift={(331.74,140.44)}, rotate = 353.05] [color={rgb, 255:red, 74; green, 144; blue, 226 }  ,draw opacity=1 ][fill={rgb, 255:red, 74; green, 144; blue, 226 }  ,fill opacity=1 ][line width=0.75]      (0, 0) circle [x radius= 3.35, y radius= 3.35]   ;
\draw [shift={(310.66,143.01)}, rotate = 353.05] [color={rgb, 255:red, 74; green, 144; blue, 226 }  ,draw opacity=1 ][fill={rgb, 255:red, 74; green, 144; blue, 226 }  ,fill opacity=1 ][line width=0.75]      (0, 0) circle [x radius= 3.35, y radius= 3.35]   ;
\draw [color={rgb, 255:red, 74; green, 144; blue, 226 }  ,draw opacity=1 ][line width=0.75]    (302.43,103.93) -- (327.11,101.87) ;
\draw [color={rgb, 255:red, 74; green, 144; blue, 226 }  ,draw opacity=1 ][line width=0.75]    (327.11,101.87) -- (336.37,120.38) ;
\draw [shift={(336.37,120.38)}, rotate = 63.44] [color={rgb, 255:red, 74; green, 144; blue, 226 }  ,draw opacity=1 ][fill={rgb, 255:red, 74; green, 144; blue, 226 }  ,fill opacity=1 ][line width=0.75]      (0, 0) circle [x radius= 3.35, y radius= 3.35]   ;
\draw [shift={(327.11,101.87)}, rotate = 63.44] [color={rgb, 255:red, 74; green, 144; blue, 226 }  ,draw opacity=1 ][fill={rgb, 255:red, 74; green, 144; blue, 226 }  ,fill opacity=1 ][line width=0.75]      (0, 0) circle [x radius= 3.35, y radius= 3.35]   ;
\draw [color={rgb, 255:red, 74; green, 144; blue, 226 }  ,draw opacity=1 ][line width=0.75]    (336.37,120.38) -- (331.74,140.44) ;
\draw    (285.88,82.81) -- (302.43,103.93) ;
\draw [shift={(302.43,103.93)}, rotate = 51.91] [color={rgb, 255:red, 0; green, 0; blue, 0 }  ][fill={rgb, 255:red, 0; green, 0; blue, 0 }  ][line width=0.75]      (0, 0) circle [x radius= 3.35, y radius= 3.35]   ;
\draw [shift={(284.43,80.96)}, rotate = 51.91] [color={rgb, 255:red, 0; green, 0; blue, 0 }  ][line width=0.75]      (0, 0) circle [x radius= 3.35, y radius= 3.35]   ;
\draw    (295.23,72.55) -- (302.43,103.93) ;
\draw [shift={(295.23,72.55)}, rotate = 122.08] [color={rgb, 255:red, 0; green, 0; blue, 0 }  ][line width=0.75]    (-5.59,0) -- (5.59,0)(0,5.59) -- (0,-5.59)   ;
\draw    (310.66,143.01) -- (299.34,171.3) ;
\draw [shift={(299.34,171.3)}, rotate = 156.8] [color={rgb, 255:red, 0; green, 0; blue, 0 }  ][line width=0.75]    (-5.59,0) -- (5.59,0)(0,5.59) -- (0,-5.59)   ;
\draw [shift={(310.66,143.01)}, rotate = 111.8] [color={rgb, 255:red, 0; green, 0; blue, 0 }  ][fill={rgb, 255:red, 0; green, 0; blue, 0 }  ][line width=0.75]      (0, 0) circle [x radius= 3.35, y radius= 3.35]   ;
\draw    (331.74,140.44) -- (353.86,157.42) ;
\draw [shift={(353.86,157.42)}, rotate = 82.50999999999999] [color={rgb, 255:red, 0; green, 0; blue, 0 }  ][line width=0.75]    (-5.59,0) -- (5.59,0)(0,5.59) -- (0,-5.59)   ;
\draw [shift={(331.74,140.44)}, rotate = 37.51] [color={rgb, 255:red, 0; green, 0; blue, 0 }  ][fill={rgb, 255:red, 0; green, 0; blue, 0 }  ][line width=0.75]      (0, 0) circle [x radius= 3.35, y radius= 3.35]   ;
\draw    (336.37,120.38) -- (359,113.7) ;
\draw [shift={(359,113.7)}, rotate = 388.54] [color={rgb, 255:red, 0; green, 0; blue, 0 }  ][line width=0.75]    (-5.59,0) -- (5.59,0)(0,5.59) -- (0,-5.59)   ;
\draw [shift={(336.37,120.38)}, rotate = 343.54] [color={rgb, 255:red, 0; green, 0; blue, 0 }  ][fill={rgb, 255:red, 0; green, 0; blue, 0 }  ][line width=0.75]      (0, 0) circle [x radius= 3.35, y radius= 3.35]   ;
\draw    (327.11,101.87) -- (335.86,81.29) ;
\draw [shift={(335.86,81.29)}, rotate = 338.02] [color={rgb, 255:red, 0; green, 0; blue, 0 }  ][line width=0.75]    (-5.59,0) -- (5.59,0)(0,5.59) -- (0,-5.59)   ;
\draw [shift={(327.11,101.87)}, rotate = 293.02] [color={rgb, 255:red, 0; green, 0; blue, 0 }  ][fill={rgb, 255:red, 0; green, 0; blue, 0 }  ][line width=0.75]      (0, 0) circle [x radius= 3.35, y radius= 3.35]   ;
\draw    (267.46,125.01) -- (291.11,122.44) ;
\draw [shift={(291.11,122.44)}, rotate = 353.8] [color={rgb, 255:red, 0; green, 0; blue, 0 }  ][fill={rgb, 255:red, 0; green, 0; blue, 0 }  ][line width=0.75]      (0, 0) circle [x radius= 3.35, y radius= 3.35]   ;
\draw [shift={(267.46,125.01)}, rotate = 398.8] [color={rgb, 255:red, 0; green, 0; blue, 0 }  ][line width=0.75]    (-5.59,0) -- (5.59,0)(0,5.59) -- (0,-5.59)   ;
\draw    (273.6,132.74) -- (291.11,122.44) ;
\draw [shift={(291.11,122.44)}, rotate = 329.54] [color={rgb, 255:red, 0; green, 0; blue, 0 }  ][fill={rgb, 255:red, 0; green, 0; blue, 0 }  ][line width=0.75]      (0, 0) circle [x radius= 3.35, y radius= 3.35]   ;
\draw [shift={(271.57,133.93)}, rotate = 329.54] [color={rgb, 255:red, 0; green, 0; blue, 0 }  ][line width=0.75]      (0, 0) circle [x radius= 3.35, y radius= 3.35]   ;
\draw    (280.79,90.91) -- (302.43,103.93) ;
\draw [shift={(302.43,103.93)}, rotate = 31.02] [color={rgb, 255:red, 0; green, 0; blue, 0 }  ][fill={rgb, 255:red, 0; green, 0; blue, 0 }  ][line width=0.75]      (0, 0) circle [x radius= 3.35, y radius= 3.35]   ;
\draw [shift={(278.77,89.7)}, rotate = 31.02] [color={rgb, 255:red, 0; green, 0; blue, 0 }  ][line width=0.75]      (0, 0) circle [x radius= 3.35, y radius= 3.35]   ;
\draw [color={rgb, 255:red, 74; green, 144; blue, 226 }  ,draw opacity=1 ][line width=0.75]    (411.11,122.44) -- (422.43,103.93) ;
\draw [shift={(422.43,103.93)}, rotate = 301.43] [color={rgb, 255:red, 74; green, 144; blue, 226 }  ,draw opacity=1 ][fill={rgb, 255:red, 74; green, 144; blue, 226 }  ,fill opacity=1 ][line width=0.75]      (0, 0) circle [x radius= 3.35, y radius= 3.35]   ;
\draw [shift={(411.11,122.44)}, rotate = 301.43] [color={rgb, 255:red, 74; green, 144; blue, 226 }  ,draw opacity=1 ][fill={rgb, 255:red, 74; green, 144; blue, 226 }  ,fill opacity=1 ][line width=0.75]      (0, 0) circle [x radius= 3.35, y radius= 3.35]   ;
\draw [color={rgb, 255:red, 74; green, 144; blue, 226 }  ,draw opacity=1 ][fill={rgb, 255:red, 74; green, 144; blue, 226 }  ,fill opacity=1 ][line width=0.75]    (411.11,122.44) -- (430.66,143.01) ;
\draw [color={rgb, 255:red, 74; green, 144; blue, 226 }  ,draw opacity=1 ][line width=0.75]    (422.43,103.93) -- (447.11,101.87) ;
\draw [color={rgb, 255:red, 74; green, 144; blue, 226 }  ,draw opacity=1 ][line width=0.75]    (447.11,101.87) -- (448,135.3) ;
\draw [shift={(448,135.3)}, rotate = 88.48] [color={rgb, 255:red, 74; green, 144; blue, 226 }  ,draw opacity=1 ][fill={rgb, 255:red, 74; green, 144; blue, 226 }  ,fill opacity=1 ][line width=0.75]      (0, 0) circle [x radius= 3.35, y radius= 3.35]   ;
\draw [shift={(447.11,101.87)}, rotate = 88.48] [color={rgb, 255:red, 74; green, 144; blue, 226 }  ,draw opacity=1 ][fill={rgb, 255:red, 74; green, 144; blue, 226 }  ,fill opacity=1 ][line width=0.75]      (0, 0) circle [x radius= 3.35, y radius= 3.35]   ;
\draw [color={rgb, 255:red, 74; green, 144; blue, 226 }  ,draw opacity=1 ][line width=0.75]    (448,135.3) -- (430.66,143.01) ;
\draw    (405.88,82.81) -- (422.43,103.93) ;
\draw [shift={(422.43,103.93)}, rotate = 51.91] [color={rgb, 255:red, 0; green, 0; blue, 0 }  ][fill={rgb, 255:red, 0; green, 0; blue, 0 }  ][line width=0.75]      (0, 0) circle [x radius= 3.35, y radius= 3.35]   ;
\draw [shift={(404.43,80.96)}, rotate = 51.91] [color={rgb, 255:red, 0; green, 0; blue, 0 }  ][line width=0.75]      (0, 0) circle [x radius= 3.35, y radius= 3.35]   ;
\draw    (415.23,72.55) -- (422.43,103.93) ;
\draw [shift={(415.23,72.55)}, rotate = 122.08] [color={rgb, 255:red, 0; green, 0; blue, 0 }  ][line width=0.75]    (-5.59,0) -- (5.59,0)(0,5.59) -- (0,-5.59)   ;
\draw    (430.66,143.01) -- (419.34,171.3) ;
\draw [shift={(419.34,171.3)}, rotate = 156.8] [color={rgb, 255:red, 0; green, 0; blue, 0 }  ][line width=0.75]    (-5.59,0) -- (5.59,0)(0,5.59) -- (0,-5.59)   ;
\draw [shift={(430.66,143.01)}, rotate = 111.8] [color={rgb, 255:red, 0; green, 0; blue, 0 }  ][fill={rgb, 255:red, 0; green, 0; blue, 0 }  ][line width=0.75]      (0, 0) circle [x radius= 3.35, y radius= 3.35]   ;
\draw    (448,135.3) -- (474,148.3) ;
\draw [shift={(474,148.3)}, rotate = 71.57] [color={rgb, 255:red, 0; green, 0; blue, 0 }  ][line width=0.75]    (-5.59,0) -- (5.59,0)(0,5.59) -- (0,-5.59)   ;
\draw [shift={(448,135.3)}, rotate = 26.57] [color={rgb, 255:red, 0; green, 0; blue, 0 }  ][fill={rgb, 255:red, 0; green, 0; blue, 0 }  ][line width=0.75]      (0, 0) circle [x radius= 3.35, y radius= 3.35]   ;
\draw    (447.11,101.87) -- (470.86,88.29) ;
\draw [shift={(470.86,88.29)}, rotate = 375.24] [color={rgb, 255:red, 0; green, 0; blue, 0 }  ][line width=0.75]    (-5.59,0) -- (5.59,0)(0,5.59) -- (0,-5.59)   ;
\draw [shift={(447.11,101.87)}, rotate = 330.24] [color={rgb, 255:red, 0; green, 0; blue, 0 }  ][fill={rgb, 255:red, 0; green, 0; blue, 0 }  ][line width=0.75]      (0, 0) circle [x radius= 3.35, y radius= 3.35]   ;
\draw    (387.46,125.01) -- (411.11,122.44) ;
\draw [shift={(411.11,122.44)}, rotate = 353.8] [color={rgb, 255:red, 0; green, 0; blue, 0 }  ][fill={rgb, 255:red, 0; green, 0; blue, 0 }  ][line width=0.75]      (0, 0) circle [x radius= 3.35, y radius= 3.35]   ;
\draw [shift={(387.46,125.01)}, rotate = 398.8] [color={rgb, 255:red, 0; green, 0; blue, 0 }  ][line width=0.75]    (-5.59,0) -- (5.59,0)(0,5.59) -- (0,-5.59)   ;
\draw    (393.6,132.74) -- (411.11,122.44) ;
\draw [shift={(411.11,122.44)}, rotate = 329.54] [color={rgb, 255:red, 0; green, 0; blue, 0 }  ][fill={rgb, 255:red, 0; green, 0; blue, 0 }  ][line width=0.75]      (0, 0) circle [x radius= 3.35, y radius= 3.35]   ;
\draw [shift={(391.57,133.93)}, rotate = 329.54] [color={rgb, 255:red, 0; green, 0; blue, 0 }  ][line width=0.75]      (0, 0) circle [x radius= 3.35, y radius= 3.35]   ;
\draw    (210.76,130.28) -- (224,157.3) ;
\draw [shift={(224,157.3)}, rotate = 108.89] [color={rgb, 255:red, 0; green, 0; blue, 0 }  ][line width=0.75]    (-5.59,0) -- (5.59,0)(0,5.59) -- (0,-5.59)   ;
\draw [shift={(210.76,130.28)}, rotate = 63.89] [color={rgb, 255:red, 0; green, 0; blue, 0 }  ][fill={rgb, 255:red, 0; green, 0; blue, 0 }  ][line width=0.75]      (0, 0) circle [x radius= 3.35, y radius= 3.35]   ;
\draw    (210.76,130.28) -- (188.15,141.14) ;
\draw [shift={(186.03,142.16)}, rotate = 154.34] [color={rgb, 255:red, 0; green, 0; blue, 0 }  ][line width=0.75]      (0, 0) circle [x radius= 3.35, y radius= 3.35]   ;
\draw [shift={(210.76,130.28)}, rotate = 154.34] [color={rgb, 255:red, 0; green, 0; blue, 0 }  ][fill={rgb, 255:red, 0; green, 0; blue, 0 }  ][line width=0.75]      (0, 0) circle [x radius= 3.35, y radius= 3.35]   ;
\draw    (210.76,130.28) -- (185.1,117.31) ;
\draw [shift={(183,116.25)}, rotate = 206.82] [color={rgb, 255:red, 0; green, 0; blue, 0 }  ][line width=0.75]      (0, 0) circle [x radius= 3.35, y radius= 3.35]   ;
\draw [shift={(210.76,130.28)}, rotate = 206.82] [color={rgb, 255:red, 0; green, 0; blue, 0 }  ][fill={rgb, 255:red, 0; green, 0; blue, 0 }  ][line width=0.75]      (0, 0) circle [x radius= 3.35, y radius= 3.35]   ;
\draw    (234.7,135.04) -- (213.06,130.74) ;
\draw [shift={(210.76,130.28)}, rotate = 191.26] [color={rgb, 255:red, 0; green, 0; blue, 0 }  ][line width=0.75]      (0, 0) circle [x radius= 3.35, y radius= 3.35]   ;
\draw [shift={(237,135.5)}, rotate = 191.26] [color={rgb, 255:red, 0; green, 0; blue, 0 }  ][line width=0.75]      (0, 0) circle [x radius= 3.35, y radius= 3.35]   ;
\draw    (210.76,110.32) -- (210.76,127.93) ;
\draw [shift={(210.76,130.28)}, rotate = 90] [color={rgb, 255:red, 0; green, 0; blue, 0 }  ][line width=0.75]      (0, 0) circle [x radius= 3.35, y radius= 3.35]   ;
\draw [shift={(210.76,107.97)}, rotate = 90] [color={rgb, 255:red, 0; green, 0; blue, 0 }  ][line width=0.75]      (0, 0) circle [x radius= 3.35, y radius= 3.35]   ;

\draw (66,108.4) node [anchor=north west][inner sep=0.75pt]    {$\Trace_{( k+1) ,-1} =\sum \limits_{j=1}^{k+1}$};
\draw (495,109.4) node [anchor=north west][inner sep=0.75pt]    {$\cdots $};
\draw (331,179.4) node [anchor=north west][inner sep=0.75pt]    {$\underbrace{\ \ \ \ \ \ \ \ \ \ \ \ \ \ \ \ \ \ \ \ \ \ }_{\Trace_{( k-1) ,0}}$};
\draw (95,206.37) node [anchor=north west][inner sep=0.75pt]  [font=\tiny]  {$\theta \left( f_{j}( y)\frac{\partial }{\partial y^{l_{j}}}\right) =f_{j}( \gamma ) \beta ^{l_{j} \dagger } +\mathop{\sum} \limits_{s=1}^{N}\frac{\partial f_{j}( \gamma )}{\partial \gamma ^{s}} \gamma ^{s\dagger } \beta ^{l_{j}}$};
\draw (198,133.4) node [anchor=north west][inner sep=0.75pt]    {$$};
\draw    (178.42,201.97) -- (198.19,154.84) ;
\draw [shift={(198.97,153)}, rotate = 472.76] [color={rgb, 255:red, 0; green, 0; blue, 0 }  ][line width=0.75]    (10.93,-3.29) .. controls (6.95,-1.4) and (3.31,-0.3) .. (0,0) .. controls (3.31,0.3) and (6.95,1.4) .. (10.93,3.29)   ;

\end{tikzpicture}
  \caption{}\label{TraceTreeOneloop}
\end{figure}

Finally, according to the computation in \cite{2011WittenGenus,gorbounov2016chiral}, we have
$$
\mathop{\sum}\limits_{\small{\Gamma_1\in \mathfrak{G}_1}}\frac{W_{\Gamma_1}}{{\mathrm{Aut}(\Gamma_1)|}}=p^*\log \mathrm{Wit}_N(\tau)+\frac{1}{32\pi^4}\widehat{E}_2\cdot p^*\mathrm{tr}(\mathrm{At}^2)
$$
which is cohomologous to $p^*\log \mathrm{Wit}_N(\tau)$ in $\mathrm{C}_{\mathrm{Lie}}(\widetilde{W}_N, \mathrm{GL}_N;\Omega_{\fO})$ since $p^*\mathrm{tr}(\mathrm{At}^2)$ is a coboundary.
\bibliographystyle{amsplain}
\providecommand{\bysame}{\leavevmode\hbox to3em{\hrulefill}\thinspace}
\providecommand{\MR}{\relax\ifhmode\unskip\space\fi MR }
\providecommand{\MRhref}[2]{%
  \href{http://www.ams.org/mathscinet-getitem?mr=#1}{#2}
}
\providecommand{\href}[2]{#2}

\end{document}